\documentclass[12pt,letterpaper]{amsart}
\usepackage[letterpaper,dvips,pdftex]{geometry}
\usepackage{graphicx} % Required for inserting images
\usepackage[utf8]{inputenc}
\usepackage{amsmath,amsthm}
\usepackage[font={small}]{caption}
\usepackage{amsfonts}
\usepackage{enumitem}
\usepackage{mathtools}
\usepackage{tikz}
\usepackage{tikz-cd}
\usetikzlibrary{shapes.geometric,fit}
\usepackage{hyperref}
\usepackage[maxalphanames=4, style=alphabetic]{biblatex}

\usepackage{tabularx}

\title[Distinguished basis in the wrapped Floer cohomology]{Distinguished bases in the wrapped Floer
cohomology of tropical Lagrangian surfaces}
\author{Jaewon Chang }

\theoremstyle{definition}
\newtheorem{theorem}{Theorem}[section]
\newtheorem{lemma}[theorem]{Lemma}
\newtheorem{definition}[theorem]{Definition}
\newtheorem{lemmadefinition}[theorem]{Lemma-Definition}
\newtheorem{corollary}[theorem]{Corollary}

\newtheorem{proposition}[theorem]{Proposition}
\newtheorem{remark}[theorem]{Remark}
\newtheorem{assumption}[theorem]{Assumption}
\newtheorem{notation}[theorem]{Notation}

\newtheorem*{assumption*}{Assumption}
\newtheorem*{acknowledgment*}{Acknowledgment}

\newcommand{\round}{{\partial}}

\newcommand{\norm}[1]{\bigl\lVert#1\bigr\rVert}

\newcommand{\hclim}[1]{\underset{#1}{\textrm{hocolim}}\,}
\newcommand{\colim}[1]{\underset{#1}{\textrm{colim}}\,}
\newcommand{\Groth}[1]{\underset{#1}{\textrm{Groth}}\,}

\newcommand{\bC}{{\mathbb{C}}}
\newcommand{\bR}{{\mathbb{R}}}
\newcommand{\bZ}{{\mathbb{Z}}}
\newcommand{\bQ}{{\mathbb{Q}}}
\newcommand{\bK}{{\mathbb{K}}}

\newcommand{\PoP}{\text{pants}}

\DeclareMathOperator{\Ima}{Im}
\DeclareMathOperator{\Rea}{Re}
\DeclareMathOperator{\Fuk}{Fuk}
\DeclareMathOperator{\Coh}{Coh}
\DeclareMathOperator{\Hom}{Hom}
\DeclareMathOperator{\Log}{log}
\DeclareMathOperator{\val}{val}
\DeclareMathOperator{\hull}{hull}
\DeclarePairedDelimiter\abs{\lvert}{\rvert}%
\newcommand{\bump}[1]{\chi_{\scriptscriptstyle #1}}

\numberwithin{equation}{section}

\begin{document}

\begin{abstract}
    We introduce a new construction for tropical Lagrangian surfaces in $(\bC^*)^2$. This construction makes the surfaces special Lagrangian, which gives a strong control over the asymptotic behavior of holomorphic disks near each cylindrical end. As a result, we find a distinguished basis of the wrapped Floer cohomology ring $HW^\bullet(L_f, L_f)$ for several different examples and find the corresponding distinguished bases of functions on the mirror curves.
\end{abstract}

\maketitle

\section{Introduction}

Beginning from the observation made by physicists that Calabi-Yau manifolds come in pairs of ``mirror dual" families, the mirror symmetry conjectures were proposed as a way to formulate the intricate connection between symplectic geometry on a space $X$ and complex geometry on its dual $\Check{X}$. One version of this is the homological mirror symmetry conjecture, which claims that the Fukaya category $\Fuk(X)$ and the derived category of coherent sheaves $D^b Coh(\Check{X})$ are quasi-isomorphic $A_\infty$-categories. In particular, this implies a cohomology-level isomorphism between $\Hom$ spaces in $\Fuk(X)$, namely Lagrangian Floer chain complexes, and the corresponding $\Hom$ spaces in $D^b Coh(\Check{X})$.

A particularly remarkable feature of mirror symmetry is that, in some circumstances, it predicts the existence of distinguished ``canonical" bases of the ring of functions (or the homogeneous coordinate ring) on the algebraic space $\Check{X}$. Indeed, global functions on $\Check{X}$ (respectively, sections of the powers $O(k)$ of an ample line bundle $O(1)$) correspond under mirror symmetry to the Floer cohomology group of a certain distinguished Lagrangian submanifold $L_0$ of $X$ (respectively, of a pair of Lagrangians $L_0$ and $L_k$), the mirror of the structure sheaf of $\Check{X}$. Lagrangian Floer complexes have preferred bases consisting of intersection points between the given Lagrangians (possibly after a suitable Hamiltonian perturbation), so when the Floer differential vanishes, this gives a preferred basis for the Floer cohomology.

For example, Polishchuk and Zaslow \cite{PZ1} studied the Floer cohomology groups of a collection of simple closed curves $L_k$ on the 2-torus which correspond under mirror symmetry to powers of a fixed degree 1 line bundle on an elliptic curve; remarkably, they found that the generators of $HF(L_0, L_k)$ (the intersection points) correspond under mirror symmetry to level $k$ theta functions -- an unexpected connection between symplectic geometry and classical complex analysis. Subsequently, Abouzaid \cite{Abo1} showed that when $\check{X}$ is a projective toric variety, the bases given by Lagrangian intersections on $X$ and toric monomials on $\check{X}$ match under mirror symmetry. This led Gross, Siebert, Hacking, and Keel to conjecture that canonical (or ``theta") bases of functions exist in a much greater generality \cite{GS1}\cite{GHK1}. This is especially interesting in the context of cluster varieties, where there were preexisting representation-theoretic and combinatorial notions of ``canonical bases", which Gross-Hacking-Keel-Kontsevich \cite{GHKK1} reformulated in the language of mirror symmetry.

The above results are obtained in the setting of mirror symmetry for log Calabi-Yau spaces, i.e. they concern spaces $\check{X}$ whose first Chern class is represented by an effective divisor. The goal of this paper is to determine the extent to which mirror symmetry continues to yield canonical bases of regular functions outside of the log Calabi-Yau setting.

However, on the symplectic geometry side, finding a distinguished basis is difficult even for the simplest cases. For example, Hicks' work on tropical Lagrangians~\cite{Hic2} constructs a Lagrangian $L_q$ mirror to the structure sheaf of the divisor
$$ D_{q} = \{ 1 - z_1 - z_2 + q^{-1}z_1 z_2 = 0 \} \subseteq (\bC^*)^2$$
for each $q \in \bC^*$, where each $L_q$ is a genus zero surface with four cylindrical ends (or punctures). In particular the Floer cohomology of $L_{q}$ is isomorphic to the ring of algebraic functions on $D_{q}$, which is the quotient ring
\begin{equation}\label{eqn:intro-quotient-ring}
    \bK[z_1^{\pm}, z_2^{\pm}] / (1-z_1-z_2+q^{-1}z_1 z_2)
\end{equation}
where $\bK$ is the coefficient field. Apparently, we don't have a distinguished basis for this ring because it is presented as a quotient. The Floer cohomology calculation done by Abouzaid and Auroux \cite{AA1} for mirrors of general hypersurfaces in $(\bC^*)^n$ uses different Lagrangian submanifolds, but it also fails to give a basis of functions on $D_{q}$, as the degree $0$ Floer complex is again $\bK[x^{\pm}, y^{\pm}]$ and the cohomology is its quotient by the ideal generated by $1-x-y+q^{-1}xy$.

In this paper, we construct a Lagrangian $L_f$ from a Laurent polynomial $f$ and find a distinguished basis for its Floer cohomology. This is related to the works \cite{Mat1}\cite{Hic2} to construct a tropical Lagrangian $L(\phi)$ whose image under the torus fibration, which is the map $\Log: (\bC^*)^2 \to \bR^2$ in the case of $(\bC^*)^2$, is sufficiently close to a collection of lines called the tropical hypersurface $V(\phi)$. However, instead of constructing tropical Lagrangian directly from the tropical structure, we start from a complex hypersurface of $(\bC^*)^2$ with the same tropical structure and construct the Lagrangian by performing a HyperKähler rotation.

In Section 2, we construct a tropical Lagrangian $L_f$ and a family of Hamiltonian perturbations for wrapped Floer theory for every Laurent polynomial $f$ with a tropical smoothness condition. The analysis on the cylindrical region near each pucture is crucial; in particular, all the degree 0 generators of the wrapped Floer complex are indexed by the pair $(Z_{\alpha,r}^+, j)$ of a cylindrical end $Z_{\alpha,r}^+$ of $L_f$ and a nonnegative (or positive) integer $j$. Here, $\alpha \in \bZ^2$ is the asymptotic direction of the cylindrical end $Z_{\alpha,r}^+$ on the base space $\bR^2$, and we denote the corresponding generators by $x_\alpha^j$ (In the actual statements, however, we have an auxiliary term $x_\alpha^e$ to distinguish them from degree 1 generators). Using these definitions and observations, we construct a wrapped Floer cohomology of these Lagrangians in Section 4. 

The model example we will consider is the Lagrangian $L_q$ defined above. In Section 5, we find a distinguished basis of the wrapped Floer cohomology of $L_q$ and the corresponding elements in (\ref{eqn:intro-quotient-ring}) up to additive constants as follows. Note that, in general, we are required to work over a non-archimedean field called the Novikov field for non-exact Lagrangians $L_q$.
\begin{theorem}[Theorem~\ref{thm:HMS-for-Lq}]
    For $q \in \bC \setminus \bR$, there exist constants $a_i$, $b_i$ for $i \in \bZ$ and $A_i$, $B_i$ for $i \in \bZ_{>0}$ such that there exists a basis of $HW^0 (L_q, L_q)$ consists of the following elements, together with the corresponding elements in the analogous version of (\ref{eqn:intro-quotient-ring}) over the Novikov field:
    \begin{itemize}
        \item $x^i + b_i y^0$,\quad for $i>0$,\quad corresponding to $z_1^i + B_i$.
        \item $x^i + b_i y^0$,\quad for $i<0$,\quad corresponding to $z_1^i$.
        \item $y^i + a_i x^0$,\quad for $i>0$,\quad corresponding to $z_2^i + A_i$.
        \item $y^i + a_i x^0$,\quad for $i<0$,\quad corresponding to $z_2^i$.
        \item $x^0 + y^0$,\quad corresponding to the unit $1$.
    \end{itemize}
\end{theorem}
We also discuss the case of the pair of pants in Section~\ref{subsect:pair-of-pants}, where there is a cylindrical end in a diagonal direction (unlike the ends of $L_q$). Theorem~\ref{thm:HMS-for-LPoP-diag} describes the generators of the Floer cohomology and the corresponding functions.

\subsection*{acknowledgment}
This is the extended version of the author's Ph.D. thesis at Harvard University. I would like to thank Denis Auroux for helpful discussions. This paper was partially supported by NSF grant DMS-2202984.

\section{Tropical Lagrangians in \texorpdfstring{$(\bC^*)^2$}{(C*)²}}

\subsection{Tropical Structure of Riemann Surfaces}

We start from a Laurent polynomial on $(\bC^*)^2$ of the form
$f = \sum_{\alpha \in A} c_\alpha z^\alpha$,
where $A$ is a finite subset of $\bZ^2$, and $z^\alpha = z_1^{\alpha_1}z_2^{\alpha_2}$ for $\alpha = (\alpha_1 , \alpha_2 ) \in \bZ^2.$

\begin{definition}\label{def:trop-poly-of-f}
    The \textit{tropical polynomial} of $f$ is given by
    $$ \phi(f)(x_1, x_2) = \bigoplus_{\alpha \in A} \log|c_\alpha| \odot x^{\odot \alpha}. $$
\end{definition}

If $\alpha = (\alpha_1 , \alpha_2 )$, then
$$x^{\odot \alpha} = x_1^{\odot \alpha_1} \odot x_1^{\odot \alpha_2} = \alpha_1 x_1 + \alpha_2 x_2,$$
so the tropical polynomial can also be written by 
$$ \phi(f)(x_1, x_2) = \max_{\alpha \in A}\, \left( \log\abs{c_\alpha}+ (\alpha_1 x_1 + \alpha_2 x_2) \right). $$

The term inside the $\max$ function is the log norm of the monomial $c_\alpha z^\alpha$ with $x_1 = \log \abs{z_1}$ and $x_2 = \log \abs{z_2}$. This means that the value of the tropical polynomial tells which monomial in $f$ is maximal. 

$\phi(f)$ is a piecewise linear function, and its set of singularity is a $1$-dimensional skeleton inside $\bR^2$, called the $\textit{tropical skeleton}$ $V(f)$ of $f$. Each connected component of its complement is where one of the monomial $c_\alpha z^\alpha$ has a strictly bigger norm than other terms, and we call this component $C_\alpha$. 

Each unbounded component $C_\alpha$ comes from the most ``extreme" terms of $f$. Specifically, if we consider the convex hull of $A$ in $\bR^2$, each unbounded component $C_\alpha$ comes from the boundary of the hull, as stated by the following proposition.

\begin{proposition}\label{prop:c-alpha-hull}
    Let $\hull(A)$ be the convex hull of $A$.
    \begin{enumerate}
        \item For each vertex $\alpha$ of the polygon $\hull(A)$, $C_\alpha$ is unbounded.
        \item If $C_\alpha$ is unbounded for $\alpha \in A$, then $\alpha$ lies in the boundary of $\hull(A)$.
    \end{enumerate}
\end{proposition}

\begin{proof}
    If $\alpha$ is a vertex of $\hull(A)$, then there is a linear function $l:\bR^2 \to \bR$ such that $\alpha$ is the strict maximum of $l$ on $\hull(A)$. If $l(x_1, x_2) = k_1 x_1 + k_2 x_2$ is one such function, then by the maximality, the term $c_\alpha z^\alpha$ is strictly bigger than the other terms in $f$ when $z_1 = R^{k_1}$ and $z_2 = R^{k_2}$ for sufficiently large $R$. Thus the ray $\{(k_1 x, k_2 x ) \,|\, x\ge 0 \}$ should lie in $C_\alpha$ outside of a compact region.

    Next, we show that if $\alpha$ lies in the interior of $\hull(A)$, then $C_\alpha$ is bounded (or empty). After dividing all the terms by $z^\alpha$, we may assume that $\alpha = (0,0)$. In this case, any ray in $\bR^2$ starting from the origin intersects with one of the boundary edges (or its endpoint) of $\hull(A)$. If the endpoints of the boundary edge are $\beta_1, \beta_2 \in A$, then the intersection point $\beta$ is equal to $t\beta_1 + (1-t)\beta_2$ for some $0 \le t \le 1$, and the distance between $O$ and $\beta$ is exactly the log norm of $z^{\beta}$. Also, we have that 
    $$\log\abs*{z^\beta } = t \log\abs*{z^{\beta_1}} + (1-t)\log\abs*{z^{\beta_2}} \le \max\left( \log\abs*{z^{\beta_1}}, \log\abs*{z^{\beta_2}} \right). $$
    Hence, if we let $d$ be the distance between $\alpha$ and the boundary of $\hull(A)$, and set 
    $$M = \max_{\beta_1, \beta_2 \in A} \log\abs*{{c_{\beta_1}}c_{\beta_2}^{-1} }, $$
    then $C_\alpha$ is contained in the bounded set $\frac{M}{d} \hull(A)$.
\end{proof}

Consider the divisor $D_f \subseteq (\bC^*)^2$. We use the tropical skeleton of $f$ because it tells us about the approximate behavior of $D_f$ at infinity. Indeed, when the log norm of $z$ is in one of $C_\alpha$ and is large enough, then $f(z)$ cannot be zero because the term $c_\alpha z^\alpha$ in $f$ would be bigger than the sum of the others. Hence, $D_f$ gets close to the ray in the tropical skeleton. To simplify the description, we only consider Laurent polynomials $f = \sum_{\alpha \in A} c_\alpha z^\alpha$ satisfying the following.

\begin{assumption}{\label{ass:trop-ftn}}
    The convex hull of $A$ has no non-vertex lattice points on its boundary.
\end{assumption}

This assumption implies that for each ray $l_{\alpha\beta}$ between $C_\alpha$ and $C_\beta$, the terms $c_\alpha z^\alpha$ and $c_\beta z^\beta$ get significantly bigger than the others. Therefore, at infinity, $D_f$ approaches to the complex submanifold $\{c_\alpha z^\alpha = c_\beta z^\beta \}$.

Also, the non-existence of the boundary lattice points implies that $\alpha - \beta$ is a pair of coprime integers, and in this case, the complex curve $\{c_\alpha z^\alpha = c_\beta z^\beta \}$ is topologically isomorphic to a cylinder. Indeed, if we set $\log z_j = p_j + i \theta_j$ for $j=1,2$, then the defining equation is equivalent to

\begin{equation*}
    \begin{cases}
        (\alpha_1 - \beta_1)p_1 + (\alpha_2 - \beta_2)p_2 = \log\abs*{c_\beta} - \log\abs*{c_\alpha},\\
        (\alpha_1 - \beta_1)\theta_1 + (\alpha_2 - \beta_2)\theta_2 = \arg{c_\beta} - \arg{c_\alpha}.
    \end{cases}
\end{equation*}

Thus, when $\alpha_1 - \beta_1$ and $\alpha_2 - \beta_2$ are coprime to each other, this set is a product of a line in the $(p_1, p_2)$-plane and an $S^1$ in the $(\theta_1, \theta_2)$-torus.

Also, we consider the following coordinate system to describe the points near the cylinder. Let $\Xi_{\alpha,r}$ be the cylinder $\{ z^\alpha = r \} \subseteq (\bC^*)^2$, where $\alpha = (\alpha_1, \alpha_2)$ is a pair of coprime integers and $r$ is a nonzero complex number. If we use the same cylindrical coordinate system $(p_1, p_2, \theta_1, \theta_2)$ as above, $Z$ is a product of the line $\alpha_1 p_1 + \alpha_2 p_2 = \log |r|$ and a circle $\alpha_1 \theta_1 + \alpha_2 \theta_2 = \arg{r}.$ Then, we consider the new coordinate system given by
\begin{align*}
%    p_\alpha &= \frac{1}{\abs{\alpha}}(\alpha_1 p_1 + \alpha_2 p_2)&p_\alpha^\perp = \frac{1}{\abs{\alpha}}(\alpha_2 p_1 - \alpha_1 p_2)\\
%    \theta_\alpha &= \frac{1}{\abs{\alpha}}(\alpha_1 \theta_1 + \alpha_2 \theta_2)&\theta_\alpha^\perp = \frac{1}{\abs{\alpha}}(\alpha_2 \theta_1 - \alpha_1 \theta_2),
    p_\alpha^\perp &= \alpha_1 p_1 + \alpha_2 p_2,&p_\alpha = \frac{1}{\abs{\alpha}^2}(\alpha_2 p_1 - \alpha_1 p_2),\\
    \vartheta_\alpha^\perp &= \alpha_1 \theta_1 + \alpha_2 \theta_2,&\vartheta_\alpha = \frac{1}{\abs{\alpha}^2}(\alpha_2 \theta_1 - \alpha_1 \theta_2).
\end{align*}
(This notation is for complex curves, and we will use the coordinates $p_\alpha$ and $\theta_\alpha$ and cylinder $Z_{\alpha, r}$ for the Lagrangian case.)
The cylinder $\Xi_{\alpha, r}$ can be described as $\{p_\alpha^\perp = \log|r|,\,\, \vartheta_\alpha^\perp \equiv \arg{r} \pmod{2\pi} \}$, and the points on $Z$ can be uniquely described by $p_\alpha$ and $\vartheta_\alpha$ modulo $2\pi$. The $(\theta_1, \theta_2)$-torus cannot be expressed as a direct product of two $S^1$ formed by $\vartheta_\alpha$ and $\vartheta_\alpha^\perp$ unless $\alpha$ is $(\pm 1, 0)$ or $(0, \pm 1)$. However, every point on the complement of $\{\vartheta_\alpha^\perp \equiv \arg{r} + \pi \pmod{2\pi} \}$ can be uniquely expressed by $p_\alpha$, $p_\alpha^\perp$, $\vartheta_\alpha$ modulo $2\pi$ and $\vartheta_\alpha^\perp \in (\arg r - \pi, \arg r + \pi).$

%\vspace{0.3em}
%\begin{figure}[ht]
%    \includegraphics[width=0.6\textwidth]{figtest2.jpg}
%    \caption{}
    %\label{homotopy-method}
%\end{figure}
%\vspace{0.3em}

Since we will only need half of this cylinder, we orient the cylinder as follows. (Topologically, this is not an orientation of a cylinder, but rather an orientation of the line in the $(p_1, p_2)$-plane.)

\begin{definition}
    Consider the cylinder $\Xi_{\alpha, r} = \{ z^\alpha = r \}$, and suppose that the coordinate system $(p_\alpha, p_\alpha^\perp, \vartheta_\alpha, \vartheta_\alpha^\perp)$ is defined as above.
    \begin{itemize}
        \item The \textit{positive half} of $\Xi_{\alpha, r}$ is defined by 
        $$\Xi_{\alpha, r}^+ := \{ p_\alpha^\perp = \log|r|,\,\, \vartheta_\alpha^\perp \equiv \arg{r} \pmod{2\pi},\,\, p_\alpha > 0 \}.$$
        \item The vector $(\alpha_1, \alpha_2)$ is called the \textit{normal direction} of the cylinder $\Xi_{\alpha, r}^+$.
    \end{itemize}
\end{definition}

The half-infinite cylinder $\Xi_{\alpha, r}^+$ is the product of a ray in the $(p_1, p_2)$-plane and a circle in the $(\theta_1, \theta_2)$-plane, and the normal direction of the cylinder is orthogonal to the ray in the $(p_1, p_2)$-plane.

Now, we describe the behavior of the divisor $D_f = \{f = 0 \}$ using this definition.

\begin{definition}\label{defn:trop-ends-C}
    Let $f$ be a Laurent polynomial satisfying Assumption~\ref{ass:trop-ftn}, and denote the unbounded components of the complement of $V(f)$ by $C_{\beta_1}, C_{\beta_2}, \cdots, C_{\beta_k}$ in a counterclockwise order. Then, the \textit{holomorphic tropical cylindrical ends} of $f$ are the collection $\{\Xi_{\alpha_1, r_1}^+,\Xi_{\alpha_2, r_2}^+, \cdots, \Xi_{\alpha_m, r_m}^+\}$ of half-infinite cylinders in $(\bC^*)^2$, where $\alpha_j = \beta_{j+1} - \beta_j$ and $r_j = c_{\beta_j} c_{\beta_{j+1}}^{-1}.$ (Here, $\beta_{k+1} = \beta_1$)
\end{definition}

Note that $\beta_1, \beta_2, \cdots, \beta_k$ is ordered so that $Z_j$ are in the outward direction. In this case, we can show that $D_f$ is expressed as a union of graphs of functions on each $Z_{\alpha_m, r_m}^+$ at infinity.

\begin{proposition}\label{prop:trop-approx-C}
    Let $f = \sum_{\beta \in A}c_\beta z^\beta$ be a Laurent polynomial satisfying Assumption~\ref{ass:trop-ftn}, and set 
    $$M = \max_{\beta_1, \beta_2 \in A} \log\abs*{c_{\beta_1}c_{\beta_2}^{-1}}.$$ 
    Suppose that the tropical cylindrical ends of $f$ are $\{\Xi_{\alpha_1, r_1}^+, \Xi_{\alpha_2, r_2}^+, \cdots, \Xi_{\alpha_m, r_m}^+\}$. Then, there exist constants $b_1, b_2, b_3 > 0$ depending only on $A$ that satisfy the following.

    \begin{enumerate}
        \item The curve $D_f$ lies inside the union
        $$ \{|\log|z_1||, |\log|z_2|| \le b_1 M \} \quad \cup \quad  \bigsqcup_{j=1}^k \,\{ p_{\alpha_j} > b_2 M, \,\,\abs{p_{\alpha_j}^\perp - \log\abs{r_j}} < 1 \}.$$
        \item For $j=1,2,\cdots,k$, there exists a holomorphic function $\rho_j:\{\Rea z > b_2 M \} \to \bC$ such that the set $D_f \cap \{ p_{\alpha_j} > b_2 M, \,\,\abs{p_{\alpha_j}^\perp  - \log\abs{r_j}} < 1 \}$ is equal to the graph
        $$\{(p_{\alpha_j}, p_{\alpha_j}^\perp, \vartheta_{\alpha_j}, \vartheta_{\alpha_j}^\perp) \,|\, p_{\alpha_j}^\perp + i\vartheta_{\alpha_j}^\perp = \log{r_j} + \rho_j(p_{\alpha_j} + i\vartheta_{\alpha_j}) \}.  $$
        \item The function $\rho_j$ has the upper bound
        $$ \abs*{\rho_j(p_{\alpha_j} + i\vartheta_{\alpha_j})} \le M\exp(-b_3 p_{\alpha_j}). $$
    \end{enumerate}

\end{proposition}
%$\hfill\square$

%\begin{proof}
%%%%%%%%%%%%%    
    %[Proof here]
%\end{proof}

\subsection{HyperKähler Rotation of a Holomorphic Curve}

Throughout this section, we continue using the cylindrical coordinates on $(\bC^*)^2$ defined by $\log z_j = p_j + i\theta_j$, $j=1,2$. We equip $(\bC^*)^2$ with the standard symplectic form $\omega = dp_1 \wedge d\theta_1 + dp_2 \wedge d\theta_2$. Also, there is a complex-valued $2$-form called \textit{the holomorphic volume form}, which is defined by
\begin{align*}
    \Omega &= d\log z_1 \wedge d\log z_2\\
    &= (dp_1 + i d\theta_1) \wedge (dp_2 + i d\theta_2).
\end{align*}
One property of $\Omega$ is that it vanishes on any complex curve in $(\bC^*)^2$ because $z_1$ and $z_2$ are locally related by a holomorphic function. Considering its real and imaginary parts, we have three $2$-forms
\begin{align}
    \omega &= dp_1 \wedge d\theta_1 + dp_2 \wedge d\theta_2,\\
    \Rea\Omega &= dp_1 \wedge dp_2 - d\theta_1 \wedge d\theta_2,\\
    \Ima\Omega &= dp_1 \wedge d\theta_2 - dp_2 \wedge d\theta_1.
\end{align}

%[Hyperkähler Rotation: intro here]

We consider the mapping $\Psi$ called \textit{the Hyperkähler rotation} defined by
\begin{align*}
    \Psi: (\bC^*)^2 &\to (\bC^*)^2\\
     (p_1, p_2, \theta_1, \theta_2) &\mapsto (p_1, p_2, \theta_2, -\theta_1)
\end{align*}
Under this automorphism on $(\bC^*)^2$, the $2$-forms $\omega$, $\Rea\Omega$, and $\Ima\Omega$ are related to each other:
\begin{equation}
    \Psi^*\omega = -\Ima\Omega,\quad \Psi^*\Rea\Omega = \Rea\Omega,\quad \Psi^*\Ima\Omega = \omega.
\end{equation}
In particular, both $\omega$ and $\Rea\Omega$ vanish on the Hyperkähler rotation of a holomorphic curve. Recall the following definition.

\begin{definition}
    A Lagrangian submanifold $L$ of $(\bC^*)^2$ is called a \textit{special Lagrangian} if $\Rea\Omega\mid_{L} = 0.$
\end{definition}

Now, for any Laurent polynomial $f$ satisfying the assumption~\ref{ass:trop-ftn}, the Hyperkähler rotation $L_f$ of the divisor $D_f$ is a special Lagrangian. Also, Proposition~\ref{prop:trop-approx} implies that $L_f$ is exponentially close to a cylinder on each end. To describe this, we use the version of the definition for cylindrical ends by taking a pullback.

\begin{definition}[Definition~\ref{defn:trop-ends-C}, A-side]\label{defn:trop-ends}
    For a Laurent polynomial $f$ satisfying Assumption~\ref{ass:trop-ftn}, denote the unbounded components of the complement of $V(f)$ by $C_{\beta_1}, C_{\beta_2}, \cdots, C_{\beta_k}$ in a counterclockwise order. Then, the \textit{(Lagrangian) tropical cylindrical ends} of $f$ are the collection $\{Z_{\alpha_1, r_1}^+, Z_{\alpha_2, r_2}^+, \cdots, Z_{\alpha_m, r_m}^+\}$ where $Z_{\alpha_j, r_j}^+ = \Psi \Xi_{\alpha_j, r_j}^+.$
\end{definition}

We need the ``rotated" version of the coordinate system $(p_\alpha, p_\alpha^\perp, \vartheta_\alpha, \vartheta_\alpha^\perp)$ to describe the approximate behavior of the Lagrangian $L_f$. It is given by
\begin{align*}
    p_\alpha^\perp &= \alpha_1 p_1 + \alpha_2 p_2, & p_\alpha &= \frac{1}{\abs{\alpha}^2}(\alpha_2 p_1 - \alpha_1 p_2),\\
    \theta_\alpha^\perp &= \frac{1}{\abs{\alpha}^2}(\alpha_1 \theta_1 + \alpha_2 \theta_2), & \theta_\alpha &= -\alpha_2 \theta_1 + \alpha_1 \theta_2.
\end{align*}

Note that $(p_\alpha, \theta_\alpha^\perp)$ is a coordinate system for a cylindrical end $Z_{\alpha, r}$. The existence of the coefficient $\frac{1}{\abs{\alpha}^2}$ makes the system incompatible with the standard complex structure, but the two $2$-forms we have interest in can be expressed in the usual way (with some sign difference):
\begin{align}
    \omega &= -dp_\alpha \wedge d\theta_\alpha + dp_\alpha^\perp \wedge d\theta_\alpha^\perp,\label{eqn:alpha-sympl}\\
    \Rea\Omega &= dp_\alpha \wedge dp_\alpha^\perp + d\theta_\alpha \wedge d\theta_\alpha^\perp.\label{eqn:alpha-Reomega}
\end{align}
The former identity implies that the neighborhood of the cylinder $Z_{\alpha, r}$ can be identified as a neighborhood of the zero section of the cotangent bundle $T^* Z_{\alpha, r}$, where $p_\alpha^\perp$ and $\theta_\alpha$ represent the cotangent directions dual to $\theta_\alpha^\perp$ and $p_\alpha$, respectively. In particular, for a smooth function $g: Z_{\alpha, r} \to \bR$, the graph of $dg$ is expressed by the local coordinates
\begin{equation}\label{eqn:cyl-cotangent}
    \theta_\alpha - \arg r = \frac{\round g}{\round p_\alpha},\quad p_\alpha^\perp - \log|r| = \frac{\round g}{\round \theta_\alpha^\perp}.
\end{equation}
On the other hand, the 2-form $\Rea\Omega$ will play a key role because it vanishes on both special Lagrangians and holomorphic disks. Hence, we consider the following primitives of $\omega$ and $\Rea\Omega$.
%The tautological $1$-form of this cotangent bundle corresponds to the following multi-valued differential form on $(\bC^*)^2$.

\begin{definition}\label{defn:canonical-angular-1-form}
    For a pair $\alpha = (\alpha_1, \alpha_2) \in \bZ^2$ of coprime integers and $r \in \bC^*$,
    \begin{itemize}
        \item The \textit{canonical 1-form of the (Lagrangian) cylindrical end} $Z_{\alpha, r}$ is the multi-valued 1-form on $(\bC^*)^2$ defined by
    $$ \lambda_\alpha = ({\theta_\alpha - \arg r} )dp_\alpha + (p_\alpha^\perp - \log |r|) d\theta_\alpha^\perp. $$
        \item The \textit{angular 1-form of the (Lagrangian) cylindrical end} $Z_{\alpha, r}$ is the multi-valued 1-form on $(\bC^*)^2$ defined by
    $$ \eta_\alpha = -(p_\alpha^\perp - \log |r|)dp_\alpha + ({\theta_\alpha - \arg r} ) d\theta_\alpha^\perp. $$
    \end{itemize}
\end{definition}

The only multi-valued part of $\lambda_\alpha$ and $\eta_\alpha$ is $\theta_\alpha - \arg r$, whose value is uniquely determined modulo $2\pi$.

\begin{proposition}[Proposition~\ref{prop:trop-approx-C}, A-side]\label{prop:trop-approx}
    Let $f = \sum_{\beta \in A}c_\beta z^\beta$ be a Laurent polynomial satisfying Assumption~\ref{ass:trop-ftn}, and set 
    $$M = \max_{\beta_1, \beta_2 \in A} \log\abs*{c_{\beta_1}c_{\beta_2}^{-1}}.$$ 
    Suppose that the tropical cylindrical ends of $f$ is $\{Z_{\alpha_1, r_1}^+, Z_{\alpha_2, r_2}^+, \cdots, Z_{\alpha_m, r_m}^+\}$. Then, there exist constants $b_1, b_2, b_3 > 0$ depending only on $A$ that satisfy the following.

    \begin{enumerate}
        \item The Lagrangian $L_f$ lies inside the union
        $$ \{|\log|z_1||, |\log|z_2|| \le b_1 M \} \quad \cup \quad  \bigsqcup_{j=1}^k \,\{ p_{\alpha_j} > b_2 M, \,\,\abs{p_{\alpha_j}^\perp - \log\abs{r_j}} < 1 \}.$$
        \item For $j=1,2,\cdots,k$, there exists a harmonic function $g_j:Z_{\alpha_j, r_j}^+ \to \bR$ such that the set $L_f \cap \{ p_{\alpha_j} > b_2 M, \,\,\abs{p_{\alpha_j}^\perp - \log\abs{r_j}} < 1 \}$ is equal to the graph of $dg_j$ in $T^* Z_{\alpha_j, r_j}^+$ under the identification (\ref{eqn:cyl-cotangent}).
        \item The function $g_j$ has the upper bound
        $$ \norm{d g_j(p_{\alpha_j}, \theta_{\alpha_j}^\perp)} \le M\exp(-b_3 p_{\alpha_j}). $$
    \end{enumerate}

\end{proposition}

\begin{proof}
    This is the direct corollary of the Proposition~\ref{prop:trop-approx-C}. The function $g_j$ is the primitive of the ``rotated" function $\rho_j$ in the Proposition~\ref{prop:trop-approx-C}, and the Cauchy-Riemann equation for $\rho_j$ implies that $g_j$ should be harmonic.
\end{proof}

\subsection{Cylindrical Lagrangians and Wrapping Hamiltonians}\label{subsect:Lagn-and-Ham-def}

%The way we define the wrapped Floer cohomology of $L_f$ is similar to the idea of defining \textit{monomially admissible Fukaya-Seidel category}, where the objects are \textit{monomially admissible} Lagrangians \cite{Han1}. 

%[dddddddddddd]

Consider a Laurent polynomial $f$ and constants $M, b_1, b_2, b_3 > 0$ defined in Proposition~\ref{prop:trop-approx}. For the constructions of Lagrangians and Hamiltonians throughout this section, we fix a smooth function $\chi: \bR \to [0,1]$ satisfying $\chi(p) = 0$ for $p \le 0$, $\chi(p) = 1$ for $p \ge 1$, and $0 < \chi'(p) \le 2$ for $0<p<1$. Also, for $0 < R_1 < R_2$, set
\begin{equation}
    \bump{R_1, R_2}(p) = \chi\left(\frac{p-R_1}{R_2-R_1} \right).
\end{equation}
so that $\bump{R_1, R_2}(p)$ is 0 for $p \le R_1$ and 1 for $p \ge R_2$.

%We need our Lagrangian to be cylindrical at infinity to achieve monomial transversality. These special Lagrangians never satisfy the condition, but we can consider a family of Lagrangians $\{\widetilde{L_{q, R}}\}_{R > 0}$, which are cylindrical at infinity and sufficiently close to $L_q$.

\begin{definition}\label{def:Lf-tilde-cyl}
    For $R > b_2 M$, the Lagrangian $\widetilde{L}_{f, R}^{\text{cyl}}$ is defined as follows.
    \begin{itemize}
        \item $\widetilde{L}_{f, R}^{\text{cyl}}$ is equal to $L_f$ outside the cylindrical ends.
        \item In the region $\{ p_{\alpha_j} > b_2 M, \,\,\abs{p_{\alpha_j}^\perp - \log\abs{r_j}} < 1 \}$ where $L_f$ is expressed by the graph of $d g_j$, the Lagrangian $\widetilde{L}_{f,R}^{\text{cyl}}$ is equal to the graph of $d((1-\bump{R,2R}) g_j)$ under the identification~(\ref{eqn:cyl-cotangent}).
    \end{itemize}
\end{definition}

\begin{remark}
    Typically, we use the Hamiltonian on $p_1$ and $p_2$, which is either quadratic or eventually linear. For example, if we consider the Hamiltonian
    $$ H(p_1, p_2) = \pi p_1^2 + \pi p_2^2,$$
    then its Hamiltonian flow is given by
    $$\phi_H^t: (p_1, \theta_1, p_2, \theta_2) \mapsto (p_1, \theta_1 + 2t\pi p_1, p_2, \theta_2 + 2t\pi p_2).$$
    Thus, for the zero section $L_0$ of the torus fibration given by the projection onto the $(p_1, p_2)$-plane, $L_0$ and its perturbation $\phi_H^1 L_0$ intersect at the integer points. %[[[[[reference]]]]]
    For the cylindrical Lagrangian $Z = \{p_2 = \theta_1 = 0\}$, however, $Z$ and its perturbation $\phi_H^1 Z$ do not intersect transversely, and they rather meet at infinitely many circles. %In this case, we can compute its Floer cohomology by equipping a Morse function on each intersecting circle. [[[[[reference]]]]] 
    In our setting, we avoid this issue by using the perturbed Lagrangian $\widetilde{L}_{f, R}$ rather than $\widetilde{L}_{f, R}^{\text{cyl}}$, which is cylindrical at infinity.
\end{remark}
To define the Hamiltonian, we first set 
\begin{equation}\label{eqn:mu-phi}
    \mu_\varphi(\theta) := -\cos(\theta-\varphi).
\end{equation}
This defines a family of Morse function $\{ \mu_\varphi: S^1 \to \bR \}_{\varphi \in S^1}$ such that
\begin{itemize}
    \item $\mu_\varphi$ has only one minimum at $\varphi$ and one maximum at $\varphi + \pi$.
    \item $\abs{\mu_\varphi(\theta)} \le 1$ and $\abs{\mu'_\varphi(\theta)} \le 1$ for every $\theta \in S^1$.
\end{itemize}
These functions will be used to control the $\theta$-coordinate of the generators at infinity.

%When we define the Hamiltonian $H$, we want to make $H$ have a specific value on sufficiently small neighborhoods of the cylindrical ends.

For a Laurent polynomial $f = \sum_{\alpha \in A} c_\alpha z^\alpha$ satisfying Assumption~\ref{ass:trop-ftn} and with tropical cylindrical ends $\{Z_{\alpha_1, r_1}, \cdots, Z_{\alpha_m, r_m}^+\}$, we fix the following two data depending on $A$.

\begin{itemize}
    \item A collection $\{\varphi_{\alpha_j} \}_{1 \le j \le m} \subseteq \bR \setminus \bQ$ of rationally independent irrational numbers. These will be the $\theta_\alpha^\perp$-coordinates of generators on each cylindrical end.
    \item A unit vector $(a_1, a_2)$ in $\bR^2$ not parallel or orthogonal to any $\alpha_j$.
\end{itemize}

Let $K$ be the compact subset $L_f$ consisting of all regions not expressed as a graph over any cylindrical end.

\begin{definition}\label{def:Lf-tilde}
    For a choice of collection $\{\varphi_{\alpha_j} \}_{1 \le j \le m}$ and a unit vector $(a_1, a_2)$ as above, define $\widetilde{L}_{f,R}$ as a Lagrangian obtained by perturbing each cylindrical end of $\widetilde{L}_{f,R}^{\text{cyl}}$ in the following way.

    \begin{itemize}
        \item On the compact region containing $K \cup \bigsqcup_{j}\{p_{\alpha_j} \le 2R,\,\,\abs{p_{\alpha_j}^\perp - \log\abs{r_j}} < 1 \}$, set $$\widetilde{L}_{f,R} = \widetilde{L}_{f,R}^{\text{cyl}}.$$
        \item On each region $\{p_{\alpha} > 2R,\,\, \abs{p_{\alpha}^\perp - \log\abs{r}} < 1 \}$ where $\widetilde{L}_{f,R}^{\text{cyl}}$ is a cylinder, set $\widetilde{L}_{f,R}$ as a graph of the differential of the function
        \begin{equation}
            \begin{split}
                H_\theta &= R^{-3} \bump{2R, 3R}(p_\alpha) \mu_{\varphi_\alpha} (|\alpha|^2 \theta_\alpha^\perp)\\
                &= R^{-3} \bump{2R, 3R}(p_\alpha) \mu_{\varphi_\alpha} (\alpha_1 \theta_1 + \alpha_2 \theta_2).
            \end{split}
        \end{equation}
    \end{itemize}
\end{definition}

\begin{definition}\label{def:Hf-tilde}
    For a choice of collection $\{\varphi_{\alpha_j} \}_{1 \le j \le m}$ and a unit vector $(a_1, a_2)$ as above, define the Hamiltonian $H_{f, R}$ as the following function on $(p_1, p_2)$ (or $(p_\alpha, p_\alpha^\perp)$ after a change of coordinates).
    \begin{itemize}
        \item On the compact region containing $K \cup \bigsqcup_{j}\{p_{\alpha_j} \le 2R,\,\, \abs{p_{\alpha_j}^\perp - \log\abs{r_j}} < 1 \}$,
        \begin{align}
            H_{f, R} &= R^{-1}(a_1 p_1 + a_2 p_2).
        \end{align}
        \item On each cylindrical end $\{p_\alpha > 2R,\,\, \abs{p_{\alpha}^\perp - \log\abs{r}} < 1 \}$, if $a_\alpha$ and $a_\alpha^\perp$ are the constants satisfying $a_1 p_1 + a_2 p_2 = a_\alpha p_\alpha + a_\alpha^\perp p_\alpha^\perp,$
        \begin{align}\label{eqn:HfR-formula-cyl}
            H_{f, R} &= (2\pi + 2 R^{-1} \abs{a_\alpha} ) \bump{R^2 + 3R,R^2 + 4R}(p_\alpha) p_\alpha + R^{-1} (a_\alpha p_\alpha + a_\alpha^\perp p_\alpha^\perp).
        \end{align}
    \end{itemize}
\end{definition}

\begin{remark}
    The definition above defines the Hamiltonian only near each of the cylindrical regions, and one needs to extend this to a function in $\bR^2$ to define a Floer theory for these Lagrangians. One way to achieve this is to use a twisting Hamiltonian for monomially admissible Lagrangians \cite{Han1}. Namely, we divide $\bR^2$ into several regions so that each region contains one of the cylindrical ends and extend the Hamiltonian to be linear outside of a compact set on each region, smoothly connecting them in between. See Section~\ref{subsect:distinguished-basis} for the details. However, we will only consider the Hamiltonian near the cylindrical ends since every Lagrangian and disk we consider lies in this region.
\end{remark}

We consider only sufficiently large of $R$ after fixing the values of all the other variables, including the unit vector $(a_1, a_2)$ and the integer $k$. This guarantees that all the terms of the Hamiltonian $k H_{f,R}$ on the cylindrical region $\{p_\alpha > 2R\}$ (or, $k$ times the equation (\ref{eqn:HfR-formula-cyl})) are sufficiently small except for the ``wrapping'' part $2k\pi\chi(p_\alpha) p_\alpha$. For such $R$, we divide $\widetilde{L}_{f, R}$ and its Hamiltonian perturbation by $kH_{f, R}$ into several parts.

\vspace{0.3em}
\begin{figure}[ht]
    \centering
    \includegraphics[width=0.7\textwidth]{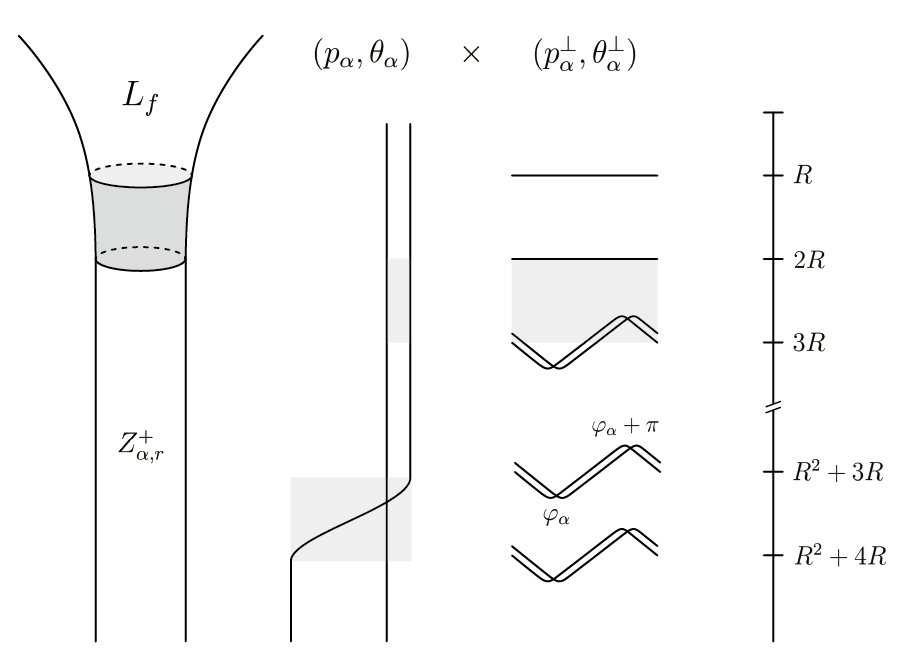}
    \caption{The Lagrangian $\widetilde{L}_{f,R}$ and its Hamiltonian perturbation. The rightmost axis shows the value of $p_\alpha$. The transitions happen in shaded regions.}
    \label{fig:lagn-ham-desc}
\end{figure}
\vspace{0.3em}

\begin{itemize}
    \item When $p_\alpha \le R$ (including the interior region), $\widetilde{L}_{f, R}$ is equal to $L_f$ and the Hamiltonian perturbation is equal to a translation in the $(\theta_1, \theta_2)$-coordinate plane. The Lagrangian $\widetilde{L}_{f, R}$ and its perturbation are both Hyperkähler rotations of complex curves in $(\bC^*)^2$, so we have a finite number of isolated simple intersection points between them when $R$ is sufficiently large.

    \item when $R \le p_\alpha \le 2R$, each end of the Lagrangian $\widetilde{L}_{f, R}$ deforms to $Z_{\alpha,r}^+$. The Hamiltonian perturbation is still a translation proportional to $R^{-1}$, and $L_f$ is exponentially close to each $Z_{\alpha,r}^+$ (Proposition~\ref{prop:trop-approx}, (c)), so there are no intersection points in this region if we choose $R$ large enough.

    \item when $2R \le p_\alpha \le 3R$, the Lagranigan $\widetilde{L}_{f, R}$ is cylindrical, and $H_\theta$ changes from $0$ to the Morse function $R^{-3} \mu_{\varphi_\alpha}(|\alpha|^2 \theta_\alpha^\perp)$. The Hamiltonian flow still gives a translation by a value proportional to $R^{-1}$, so there are no intersection points in this region if we choose $R$ large enough.

    \item The part $3R \le p_\alpha \le R^2+3R$ is the ``stretching" region. There are no intersection points because the projection of the Lagrangians onto the $(p_\alpha, \theta_\alpha)$-plane are parallel lines.

    \item when $R^2 + 3R \le p_\alpha$, the Lagrangian $\widetilde{L}_{f, R}$ is cylindrical, and the Hamiltonian ``wraps around" the $\bR_{>0}$ part of $Z_{\alpha, r}^+$.
\end{itemize}

Let's look at the $p_\alpha \ge R^2 + 3R$ part in detail. On this region, the Hamiltonian is the sum of a function on $p_\alpha$ and the linear function $R^{-1}a_\alpha^\perp p_\alpha^\perp$ on $p_\alpha^\perp$. From the equation (\ref{eqn:alpha-sympl}), the Hamiltonian flow of this Hamiltonian can be expressed by
\begin{equation}\label{eqn:flow-of-HfR}
    \phi_{H}^t: (p_\alpha, \theta_\alpha, p_\alpha^\perp, \theta_\alpha^\perp) \mapsto \left(p_\alpha,\,\,\theta_\alpha + \frac{\round H_{f, R}}{\round p_\alpha} t,\,\, p_\alpha^\perp,\,\,  \theta_\alpha^\perp + R^{-1}a_\alpha^\perp \right),
\end{equation}
where the partial derivative $h := \frac{\round H_{f, R}}{\round p_\alpha}$ is a function on $p_\alpha$ and increases from $R^{-1}a_\alpha$ to a number between $2k\pi$ and $2(k+1)\pi$. 

On the other hand, the Lagrangian $\widetilde{L}_{f, R}$ can be viewed as a product of a line $\theta_\alpha = \arg{r}$ in the $(p_\alpha, \theta_\alpha)$-plane and the graph of $$d\mu_{\varphi_\alpha} = |\alpha|^2\sin(|\alpha|^2\theta_\alpha^\perp-\varphi_\alpha) d\theta_\alpha^\perp$$ 
in the $(p_\alpha^\perp, \theta_\alpha^\perp)$-plane. For the latter part, the Hamiltonian flow translates the domain of $\mu_{\varphi_\alpha}$ by $R^{-3} a_\alpha^\perp$, so that it is a graph of 
$$d\mu_{\varphi_\alpha + R^{-1} a_\alpha^\perp} = |\alpha|^2\sin(|\alpha|^2\theta_\alpha^\perp-\varphi_\alpha - R^{-1} a_\alpha^\perp) d\theta_\alpha^\perp.$$ 
Also, $a_\alpha^\perp$ is nonzero since $(a_1, a_2)$ is not orthogonal to $\alpha$. Therefore, if $R$ is large enough, these two graphs intersect at the points where $|\alpha|^2\theta_\alpha^\perp = \alpha_1 \theta_1 + \alpha_2 \theta_2$ is equal to $\varphi_\alpha + \frac{1}{2} R^{-1} a_\alpha^\perp$ or $\varphi_\alpha + \frac{1}{2} R^{-1} a_\alpha^\perp +\pi$ modulo $2 \pi$. There are $2|\alpha|^2$ such intersection points in general, but for simplicity, we will assume that $|\alpha| = 1$ throughout the paper. This happens when the cylindrical end $Z_{\alpha,r}^+$ is either horizontal or vertical in the $(p_1, p_2)$-plane, and in this case, there are two intersection points where $\theta_\alpha^\perp = \varphi_\alpha + \frac{1}{2} R^{-1} a_\alpha^\perp$ and where $\theta_\alpha^\perp = \varphi_\alpha + \frac{1}{2} R^{-1} a_\alpha^\perp + \pi$. For the case $\abs{\alpha} > 1$, see Remark~\ref{rem:alpha-bigger} and Section~\ref{subsect:pair-of-pants}.
%$|\alpha|^2\theta_\alpha^\perp = \varphi_\alpha + \frac{1}{2} R^{-1} a_\alpha^\perp$, or $\theta_\alpha^\perp = \varphi_\alpha + \frac{1}{2} R^{-1} a_\alpha^\perp + \pi$.

In conclusion, when $|\alpha|=1$, the Lagrangian $\widetilde{L}_{f,R}$ and its Hamiltonian perturbation $\phi_H^1 \widetilde{L}_{f,R}$ intersect at the points where
%Since the cylindrical end of $\widetilde{L}_{f,R}$ can be expressed by $\{ p_\alpha \ge R,\,\, p_\alpha^\perp = \log\abs{r},\,\,\theta_\alpha = \arg{r} \}$ for some nonzero complex number $r$, The Lagrangian $\widetilde{L}_{f,R}$ and its Hamiltonian perturbation $\phi_H^1 \widetilde{L}_{f,R}$ intersects on the points where
\begin{itemize}
    \item $\theta_\alpha^\perp$ is either $\varphi_\alpha + \frac{1}{2}R^{-1} a_\alpha^\perp$ or $\varphi_\alpha + \pi + \frac{1}{2} R^{-1} a_\alpha^\perp$, and
    \item $h (p_\alpha)$ is $2\pi$ times an integer.
\end{itemize}
Since $h$ increases from $R^{-1} a_\alpha$ to a number between $2k\pi$ and $2(k+1)\pi$, the number of possible values for $p_\alpha$ is either $k$ or $k+1$, depending on the sign of $a_\alpha$. By definition, the sign of $a_\alpha$ is equal to the sign of $a_1 \alpha_2 - a_2\alpha_1$, which is nonzero since the vector $(a_1, a_2)$ is chosen to be non-parallel to $\alpha$.

\begin{notation}\label{not:def-xjxe}
    When $|\alpha|=1$, we denote the intersection points in $\{p_\alpha \ge R^2 + 3R\}$ by:
    \begin{itemize}
        \item $x_\alpha^{j}x_\alpha^{e}$ when $h(p_\alpha) = 2j\pi$ and $\theta_\alpha^\perp = \varphi_\alpha + \frac{1}{2} R^{-1} a_\alpha^\perp, \quad j = (0,)1,2,\cdots,k$.
        \item $x_\alpha^{j}x_\alpha^{f}$ when $h(p_\alpha) = 2j\pi$ and $\theta_\alpha^\perp = \varphi_\alpha+\pi+ \frac{1}{2} R^{-1} a_\alpha^\perp, \quad j = (0,)1,2,\cdots,k$.
    \end{itemize}
    The index $j$ is from $1$ to $k$ if $a_1 \alpha_2 - a_2\alpha_1 >0$, and from $0$ to $k$ if $a_1 \alpha_2 - a_2\alpha_1 < 0$.
\end{notation}

\begin{remark}
    The product notation comes from the identification of a Lagrangian cylinder $Z = \{p_2 = \theta_1 = 0\}$ with the product of two Lagrangians $\bR_{>0} \subseteq \bC^*$ and $S^1 \subseteq \bC^*$. The wrapped Floer cohomology of $\bR_{>0} \subseteq \bC^*$ is isomorphic to a ring of Laurent polynomials $k[x^\pm]$ while the Floer cohomology of $S^1$ is isomorphic to a cohomology ring $H^\bullet S^1$, which is generated by a degree $0$ element $e$ and a degree $1$ element $f$. However, the same identification works for the cylinder $Z_{\alpha, r}$ only when $|\alpha|=1$.
\end{remark}

\begin{remark}\label{rem:alpha-bigger}
    When $\abs{\alpha} > 1$, the identification between $(\theta_1 , \theta_2)$ and $(\theta_\alpha, \theta_\alpha^\perp)$ is not one-to-one, but rather $\abs{\alpha}^2$-to-one (in either direction) in its nature. Namely, when the value of $\theta_\alpha = -\alpha_2 \theta_1 + \alpha_1 \theta_2$ and $ \abs{\alpha}^2\theta_\alpha^\perp = \alpha_1 \theta_1 + \alpha_2 \theta_2$ are determined modulo $2 \pi$, there are $\abs{\alpha}^2$ possiblities for such pair $(\theta_1 , \theta_2)$ modulo $2 \pi$ since the pair
    $$\left( \theta_1 + \frac{\alpha_1}{\abs{\alpha}^2} \cdot 2k\pi , \quad \theta_2 + \frac{\alpha_2}{\abs{\alpha}^2} \cdot 2k\pi \right)$$
    for any integer $k$ also satisfies the condition. On the other hand, the value of $\theta_\alpha$ and $\abs{\alpha}^2 \theta_\alpha^\perp$ modulo $2 \pi \abs{\alpha}^2$ uniquely determines the value of $\theta_1$ and $\theta_2$ by 
    $$ \theta_1 = - \frac{\alpha_2}{\abs{\alpha}^2} \theta_\alpha + \frac{\alpha_1}{\abs{\alpha}^2} \abs{\alpha}^2 \theta_\alpha^\perp, \quad \theta_2 = \frac{\alpha_1}{\abs{\alpha}^2} \theta_\alpha + \frac{\alpha_2}{\abs{\alpha}^2} \abs{\alpha}^2 \theta_\alpha^\perp, $$
    and there are $\abs{\alpha}^2$ different pairs $(\theta_\alpha, \theta_\alpha^\perp)$ giving the same $\theta_1$ and $\theta_2$. Therefore, if we use Notation~\ref{not:def-xjxe} to describe the intersection points, there are $\abs{\alpha}^2$ different points in $(\bC^*)^2$ denoted as $x_\alpha^j x_\alpha^e$ (and $x_\alpha^j x_\alpha^f$, respectively) for each $j$, namely the points satisfying
    $$\theta_\alpha = \arg{r} \quad \text{and} \quad \abs{\alpha}^2 \theta_\alpha^\perp = \varphi_\alpha + \frac{1}{2} R^{-1} a_\alpha^\perp$$
    modulo $2 \pi$. For a detailed discussion, see Section~\ref{subsect:pair-of-pants}.
\end{remark}

\begin{remark}\label{rem:cyl-end-desc-for-perturbation}
    Near each cylindrical end $Z_{\alpha,r}^+$, the Lagrangians $L_f$, $\widetilde{L}_{f,R}^{\text{cyl}}$ and $\widetilde{L}_{f,R}$ are expressed as a graph of $dg$ for some function $g$ on $Z_{\alpha,r}^+ \cap \{p_\alpha \ge R\}$. Moreover, even though the Hamiltonian $H_{f, R}$ is not a function on $Z_{\alpha,r}^+ \cap \{p_\alpha \ge R\}$, the corresponding Hamiltonian perturbation of $\widetilde{L}_{f, R}$ is still a graph of some function. Indeed, on the region $\{p_\alpha \ge 3R\}$, the Lagrangian $\widetilde{L}_{f,R}$ is the graph of the differential of
    $$g_0(p_\alpha, \theta_\alpha^\perp) = R^{-3}\mu_{\varphi_\alpha}(\theta_\alpha^\perp),$$
    and the perturbation $\phi_{kH}^1 \widetilde{L}_{f,R}$ is the graph of the differential of
    $$g_k(p_\alpha, \theta_\alpha^\perp) = 2k\pi \bump{R^2 + 3R, R^2 +4R}(p_\alpha) p_\alpha + k R^{-1} a_\alpha p_\alpha + R^{-3}\mu_{\varphi_\alpha + kR^{-1}a_\alpha^\perp}(\theta_\alpha^\perp),$$
    where the first two terms are the $p_\alpha$ part of $H_{f,R}$ and the translation on the third term comes from the $p_\alpha^\perp$ part of $H_{f,R}$.
\end{remark}

When $R$ is large enough, the Floer theory of $\widetilde{L}_{f, R}$ with the Hamiltonian perturbation $kH_{f, R}$ in the interior region can be viewed as a Morse theory on $L_f$ with the Morse function 
$$(p_1, p_2, \theta_1, \theta_2) \mapsto a_1p_1 + a_2p_2.$$
The Morse function is harmonic due to the properties of special Lagrangians, so its critical points all have index $1$.
    
Moreover, the condition that $a_1 \alpha_2 - a_2\alpha_1$ is positive or negative is equivalent to that the corresponding cylindrical end $Z_{\alpha,r}^+$ points upward or downward for this Morse function. The generators $x_\alpha^0 x_\alpha^e$ and $x_\alpha^0 x_\alpha^f$ on each downward cylindrical end can be added to the Morse generators. Then, the definition of $H_{f, R}$ makes the Morse function bounded below by ``bending" the downward ends. This adds the same number of generators in degrees 0 and 1, so the number of interior generators equals the Euler characteristic of $L_f$. We can summarize these results as follows.

\vspace{0.3em}
\begin{figure}[ht]
    \centering
    \includegraphics[width=0.5\textwidth]{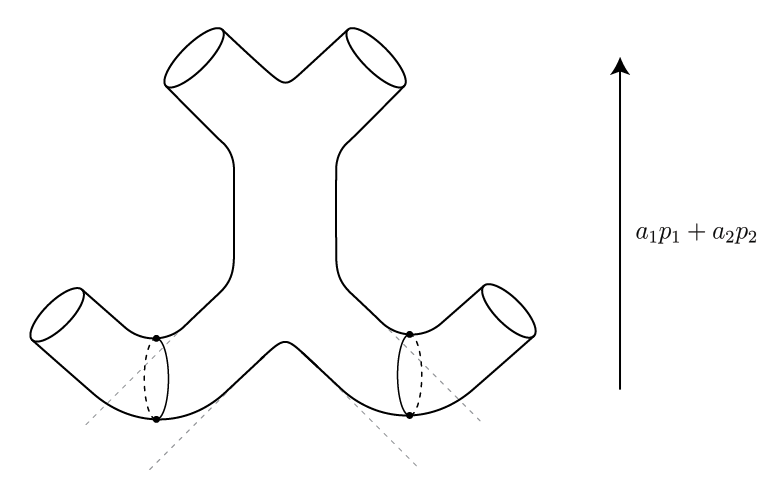}
    \caption{Morse function and bending the lower cylindrical ends.}
    \label{fig:bending}
\end{figure}
\vspace{0.3em}

\begin{lemma}\label{lem:Riemann-translation-intersection}
    For large enough $R>0$, The Lagrangian $L_f$ and its translation $L_f + R^{-1}(a_1 \theta_1 + a_2 \theta_2)$ have exactly $2g(L_f) + b(L_f) - 2$ intersection points, where $g(L_f)$ is the genus and $b(L_f)$ is the number of cylindrical ends, and all intersection points have cohomological degree $1$.
\end{lemma}
$\hfill\square$

\begin{remark}\label{rem:innermost-gen-morse}
    From the identification with the Morse theory, we can expect that the sum of $x_\alpha^0 x_\alpha^e$ would be the unit of the cohomology ring of $L_f$. In fact, this sum turns out to be a unit of the wrapped Floer cohomology ring of $L_f$. We will prove that in Section 5.
\end{remark}

Now we have a complete description of the Floer cochain complex $CF^\bullet(\widetilde{L}_{f, R}; -kH_{f, R})$ when $R$ is large enough. Note that we made the following additional assumption, and the main example in Section~\ref{sect:main-example} will also satisfy this.

\begin{assumption}\label{ass:cyl-end-is-not-diagonal}
    All the cylindrical ends $Z_{\alpha,r}^+$ of $L_f$ are in the direction $\alpha$ with $|\alpha| = 1$.
\end{assumption}

Under this assumption, we can summarize the above results as follows.

\begin{proposition}\label{prop:self-intersection}
    Let $f = \sum_{\beta \in A}c_\beta z^\beta$ be a Laurent polynomial satisfying Assumption~\ref{ass:trop-ftn} and Assumption~\ref{ass:cyl-end-is-not-diagonal}. Also, set $M > 0$ as Proposition~\ref{prop:trop-approx}. Then there exists a number $b > 0$ depending only on $A$ such that for any choice of $\{\varphi_{\alpha_j}\}$ and $(a_1, a_2)$, and for every integer $k > 0$ and a real number $R > bkM$, the Lagrangian $\widetilde{L} = \widetilde{L}_{f,R}$ and its Hamiltonian perturbation by $H = -k H_{f, R}$ intersect transversely to each other. In addition, $CF^\bullet(\widetilde{L}; H)$ is generated by the following two types of elements.
    \begin{itemize}
        \item On the interior region: degree 1 generators $v_1, v_2, \cdots, v_{l}$, where $l = 2g(L_f) + b(L_f) - 2.$
        \item On each cylindrical end $Z_{\alpha, r}^+$: $k$ degree $0$ generators $x_\alpha^j x_\alpha^e$ with $\theta_\alpha^\perp = \varphi_\alpha$ and $k$ degree $1$ generators $x_\alpha^j x_\alpha^f$ with $\theta_\alpha^\perp = \varphi_\alpha+\pi$. The index $j$ is from $1$ to $k$ if $a_1 \alpha_2 - a_2\alpha_1 >0$, and from $0$ to $k-1$ if $a_1 \alpha_2 - a_2\alpha_1 < 0$. 
    \end{itemize}
    In particular, $CF^0(\widetilde{L}; H)$ is generated by the elements of the form $x_\alpha^j x_\alpha^e$.
\end{proposition}

$\hfill \square$

%\vspace{0.3em}
%\begin{figure}[ht]
%    \includegraphics[width=0.4\textwidth]{fig-gen_new.png}
%    \caption{}
    %\label{homotopy-method}
%\end{figure}
%\vspace{0.3em}
We end this section with the following definition. These generators will play a critical role when describing a canonical basis, as explained in Remark~\ref{rem:innermost-gen-morse}.
\begin{definition}
    The generators of the form $x_\alpha^0 x_\alpha^e$ in Proposition~\ref{prop:self-intersection} are called the \textit{innermost generators}.
\end{definition}

\section{Holomorphic Disks Bounded by Tropical Lagrangians}

In this section, we observe holomorphic disks bounded by the Lagrangians constructed in the previous section. This includes computing the area of these disks and the ``argument constraints", which implies the non-existence of several different types of holomorphic disks that are counted in the construction of wrapped Floer cohomology. 

We consider the following holomorphic maps throughout this section, which are the building blocks of the Floer cohomology groups defined in the next section.
\begin{notation}\label{not:disk-counted}
    Suppose that $S_d$ is any disk with $d+1$ boundary points $\zeta_0, \zeta_1, \cdots, \zeta_d$, in counterclockwise order. Also, let $u: S_d \to (\bC^*)^2$ be a holomorphic map satisfying the following.
    \begin{itemize}
        \item The boundary component between $\zeta_j$ and $\zeta_{j+1}$ maps to $L_{f,R}(-k_jH_{f,R})$.
        \item The boundary component between $\zeta_0$ and $\zeta_{d}$ maps to $L_{f,R}(-k_dH_{f,R})$.
        \item $k_0 < k_1 < \cdots < k_d$.
    \end{itemize}
\end{notation}

\subsection{Energy of Holomorphic Disks}\label{subsect:energy}

Throughout this section, we will only consider holomorphic disks $u: S_d \to (\bC^*)^2$ whose vertices lie on the same cylindrical end $Z_{\alpha,r}^+$. Namely, for a fixed $\alpha$, we assume that $u(\zeta_0), \cdots, u(\zeta_d)$ are either of the form $x_\alpha^j x_\alpha^e$ or $x_\alpha^j x_\alpha^f$ for some nonnegative integer $j$.

Recall that for a cylindrical end $Z_{\alpha,r}^+$ of the Lagrangian $L_f$, its canonical 1-form
$$\lambda_\alpha = ({\theta_\alpha - \arg r} )dp_\alpha + (p_\alpha^\perp - \log r) d \theta_\alpha^\perp$$
is a multi-valued 1-form satisfying $\omega = d \lambda_\alpha$. 

For a holomorphic disk $D \subseteq (\bC^*)^2$ with boundary on a Lagrangian $L$, its (topological) energy is defined as the integral
\begin{equation}
    E(D) = \int_{D} \omega.
\end{equation}
By Stokes' Theorem, this integral is equal to the line integral
\begin{equation} \label{eqn:energy-lambda}
    \int_{\round D} \lambda_\alpha,
\end{equation}
which can be computed whenever we fix a lift of $D \subseteq (\bC^*)^2$ to the universal cover of $(\bC^*)^2$, i.e. real lifts of the $\theta_1$ and $\theta_2$ coordinates along $D$.

By Proposition~\ref{prop:trop-approx}, $L_f$ is the graph of a function $g_\alpha: Z_{\alpha,r}^+ \to \bR$ in the region $\{p_\alpha \ge bM\}$ for some constant $bM$ depending on $f$, under the identification with the neighborhood of the zero section in $T^*Z_{\alpha,r}^+$. In this case, the line integral along a curve in this graph is determined by the value of $g$ at the endpoints. Indeed, when a curve $\gamma: [0,1] \to T^*Z_{\alpha,r}^+$ lies inside a graph $dg$ of some function $g: Z_{\alpha,r}^+ \to \bR$, then under the identification~(\ref{eqn:cyl-cotangent}) we have
\begin{equation}\label{eqn:graph-integral}
    \int_{\gamma} \lambda_\alpha = \int_{\gamma} \frac{\round g}{\round p_\alpha} dp_\alpha + \frac{\round g}{\round \theta_\alpha^\perp} d\theta_\alpha^\perp = \int_{\gamma} dg = g(\gamma(1)) - g(\gamma(0)).
\end{equation}
Moreover, as discussed in Remark~\ref{rem:cyl-end-desc-for-perturbation}, all the Lagrangians we use and their Hamiltonian perturbations have similar expressions. Therefore, if a disk $D$ lies in the cylindrical region and is bounded by Lagrangians of the form $L_f(-kH_{f, R})$, then the energy $E(D)$ is determined by the value of the corresponding functions at the vertices.
%Moreover, the Hamiltonian $H_{f,R}$ in the Definition~\ref{def:Hf-tilde} can be considered as a function $H_\alpha$ on $(p_\alpha, \theta_\alpha^\perp)$ at each cylindrical end $Z_{\alpha,r}^+$. Therefore, if a disk $D$ lies in the cylindrical region and is bounded by Lagrangians of the form $L_f(-kH_{f, R})$, then the energy $E(D)$ is determined by the value of the functions $g_\alpha$ and $H_\alpha$ at the vertices.

In general, the function $g$ cannot be extended to a well-defined function $\widetilde{g}$ on the whole Lagrangian because the Lagrangian may not be exact. However, we can find the energy of the disk using the homology class of the boundary.

\begin{definition} For a cylinder $Z_{\alpha,r}^+$,
    \begin{enumerate}
        \item We say that $Z_{\alpha,r}^+$ is a \textit{cylindrical end} of a Lagrangian $L$ if for some $R>0$, the intersection $L \cap \{p_\alpha \ge R, \abs{p_\alpha^\perp - \log{\abs{r}}} \le 1\}$ can be represented as a graph of $dg$ for some function $g$ on $Z_{\alpha,r}^+ \cap \{p_\alpha \ge R\}$ under the identification (\ref{eqn:cyl-cotangent}).
        \item Suppose that $Z_{\alpha,r}^+$ is a cylindrical end of a Lagrangian $L$, and let $S_\alpha^1$ be the loop in $L$ wrapping around $Z_{\alpha,r}^+$ once. Then, for a curve $\gamma$ on $L$ with its endpoints in the region $\{p_\alpha \ge R, \abs{p_\alpha^\perp- \log{\abs{r}}} \le 1\}$, define the \textit{homology class of $\gamma$ modulo the cylindrical end} $Z_{\alpha,r}^+$ be the relative homology class
        \begin{align*}
            [\gamma]_\alpha &\in H_1(L,\{p_\alpha \ge R, \abs{p_\alpha^\perp - \log{\abs{r}}} \le 1\})\\
            &\cong H_1(L)/[S_\alpha^1],
        \end{align*}
        where the latter isomorphism comes from the long exact sequence for the pair $(L, L \cap \{p_\alpha \ge R, \abs{p_\alpha^\perp  - \log{\abs{r}}} \le 1\})$.
    \end{enumerate}
\end{definition}

We can also construct a homology class $[\gamma]_\alpha$ in $H_1(L)/[S_\alpha^1]$ directly. Namely, if we fix a point $z_0 \in L \cap \{p_\alpha \ge R, \abs{p_\alpha^\perp - \log{\abs{r}}} \le 1\}$ and construct a loop $\gamma_\circ$ by connecting each endpoint to $z_0$ along a curve in $L \cap \{p_\alpha \ge R, \abs{p_\alpha^\perp - \log{\abs{r}}} \le 1\}$, then it defines a homology class $[\gamma_\circ] \in H_1(L)$ whose quotient in $H_1(L)/[S_\alpha^1]$ is equal to $[\gamma]_\alpha$.

Now, the following lemmas are proved by simple calculation.

\begin{lemmadefinition}\label{def:wrapping-number-alpha}
    Consider a class $[\gamma]_\alpha \in H_1(L)/[S_\alpha^1]$ represented by a loop $\gamma_\circ:[0,1] \to L$. Then, there exists an integer $w$ depending only on the homology class $[\gamma]_\alpha$ that satisfies the following:

    For a lift $\widetilde{\gamma_\circ}: [0,1] \to (\widetilde{\bC^*})^2$ of $\gamma_\circ$ to the universal cover of $(\bC^*)^2$, and the map evaluating the  $\theta_\alpha$ coordinate
    \begin{align*}
        \pi_{\theta_\alpha}: (\widetilde{\bC^*})^2 &\to \bR\\
        (p_1, \theta_1, p_2, \theta_2) &\mapsto -\alpha_2 \theta_1 + \alpha_1 \theta_2,
    \end{align*}
    The integer $w$ satisfies
    $$\pi_{\theta_\alpha} \circ \widetilde{\gamma_\circ}(1) - \pi_{\theta_\alpha} \circ \widetilde{\gamma_\circ}(0) = 2\pi w.$$
    We call $w = w_\alpha(\gamma)$ the \textit{wrapping number of $\gamma$ with respect to the cylindrical end} $Z_{\alpha,r}^+$.
\end{lemmadefinition}

$\hfill \square$

\begin{lemmadefinition}\label{def:integral-lambda-alpha}
    Suppose that $Z_{\alpha,r}^+$ is a cylindrical end of a Lagrangian $L$, and $[\gamma]_\alpha \in H_1(L)/[S_\alpha^1]$ represented by a loop $\gamma_\circ:[0,1] \to L$. If we denote the $p_\alpha$ value of the endpoints $\gamma_\circ(0) = \gamma_\circ(1)$ by $p_{\alpha,0}$, and let $\widetilde{\gamma_\circ}:[0,1] \to (\widetilde{\bC^*})^2$ be a lift of $\gamma_\circ$ to the universal cover of $(\bC^*)^2$, then the integral
    $$\int_{[\gamma]_\alpha} \lambda_\alpha := \int_{\widetilde{\gamma_\circ}} \lambda_\alpha - 2\pi w_\alpha(\gamma) p_{\alpha, 0} $$
    depends only on the homology class $[\gamma]_\alpha$.
\end{lemmadefinition}

$\hfill \square$

\begin{remark}\label{rem:integral-of-lambda_alpha-is-invariant}
    The wrapping number $w_\alpha(\gamma)$ is determined by the homotopy class of the loop $\gamma$ in $L$. Indeed, on the universal cover $(\widetilde{\bC^*})^2$ of $(\bC^*)^2$, lifts of the cylinder $Z_{\alpha,r}$ can be viewed as equally spaced parallel cylinders. The endpoints of $\gamma$ may lie on different lifts, and $w_\alpha(\gamma)$ represents the distance between these cylinders. 
    On the other hand, the integral $\int_{\gamma} \lambda_\alpha$ is not determined by the class of $\gamma$ in the loop space, but rather the class in the based loop space. This is because if we move the endpoints of $\gamma$ by $\Delta p_\alpha$, the integral differs by $w_\alpha(\gamma) \Delta p_{\alpha}$. This also implies that the integral is invariant under the change of base points for the curves $\gamma$ with wrapping number $w_\alpha(\gamma)=0$.
    %because if we translate $L$ and $\gamma$ in the $(p_1, p_2)$-plane, then the integral $\int_{\widetilde{\gamma_\circ}} \lambda_\alpha$ remains the same while $p_{\alpha, 0}$ doesn't. However, it is a topological invariant for the curves $\gamma$ with wrapping number $w_\alpha(\gamma)=0$.
\end{remark}

Combining (\ref{eqn:graph-integral}) and Lemma~\ref{def:integral-lambda-alpha}, we get the following formula for the map $u:S_d \to (\bC^*)^2$.

\begin{theorem}[Energy of a Disk]\label{thm:action-ftnal}
    Suppose that $L_0, L_1, \cdots, L_d$ are Lagrangians having $Z_{\alpha,r}^+$ as a cylindrical end, and let $g_0, g_1, \cdots, g_d$ be functions on $Z_{\alpha,r}^+ \cap \{p_\alpha \ge R\}$ such that $L_s$ is represented as a graph of $d g_s$. Also, fix a lift $\widetilde{L}_s$ of each $L_s$ to the universal cover $(\widetilde{\bC^*})^2$ of $(\bC^*)^2$.
    Then, for a holomorphic curve $u:S_d \to (\bC^*)^2$ satisfying
    \begin{itemize}
        \item the boundary component between $\zeta_s$ and $\zeta_{s+1}$ maps to a curve $\gamma_s$ in $L_s$,
        \item the boundary component between $\zeta_0$ and $\zeta_d$ maps to a curve $\gamma_d$ in $L_d$,
        \item Each $\zeta_j$ maps to a point inside the cylindrical region $\{p_\alpha \ge R, \abs{p_\alpha^\perp - \log{\abs{r}}} \le 1\}$, and its $(p_\alpha, \theta_\alpha^\perp)$ coordinates represent $x_j \in Z_{\alpha,r}^+$.
    \end{itemize}
    the symplectic area of $u$ is equal to
    \begin{align}
        \sum_{s=0}^d\int_{[\gamma_j]_\alpha}\lambda_\alpha + \sum_{j=0}^{d-1} (g_j(x_{j+1}) - g_j (x_j)) + g_d (x_0) - g_d(x_d) + 2\pi \sum_{s=0}^d j_s p_s,
    \end{align}
    where $j_s$ is the integer satisfying the following.
    \begin{itemize}
        \item If we pick a lift of the point $u(\zeta_s)$ on the universal cover $(\widetilde{\bC^*})^2$ so that it lies on $\widetilde{L}_s$, then it lies on the lift $\widetilde{L}_{s+1}$ translated in $\theta_\alpha$-direction by $2j_s\pi$.
    \end{itemize}
    Moreover, the integral is independent of the choice of the lift for the multi-valued 1-form $\lambda_\alpha$. 
\end{theorem}

\begin{remark}
    When all the Lagrangians $L_s$ are of the form $\widetilde{L}_{f,R}(-kH_{f,R})$ with the same $f$ and $R$, then each $u(\zeta_j)$ is of the form $x_\alpha^i x_\alpha^e$ or $x_\alpha^i x_\alpha^f$ following the notation in Proposition~\ref{prop:self-intersection}, and the integer $j_s$ in the previous theorem is equal to the exponent $i$.
\end{remark}

\subsection{Argument Constraints for Holomorphic Strips}\label{subsect:arg-constraint-for-strip}

In this section, we look at holomorphic strips between two Lagrangians of the form 
$$\widetilde{L}_{f,R}^k :=\widetilde{L}_{f,R}(-kH_{f,R})$$
for fixed $f$ and $R$. Most of these strips have ``infinite ends" on one or two cylindrical ends close to a strip with a fixed width. When the strip has only one such end in $Z_{\alpha,r}^+$-direction, then under the suitable boundedness assumption, we can show that there exists an asymptotic value of $\theta_\alpha^\perp$ and it is determined by the topological property of $L_f$. This ``argument constraint" will work as a strong constraint when computing the differential of Floer cochain complexes.

We assume that all the Laurent polynomial $f$ throughout the section satisfy Assumption~\ref{ass:trop-ftn} and Assumption~\ref{ass:cyl-end-is-not-diagonal}, so that the intersection points between their Hamiltonian perturbation can be expressed as in Proposition~\ref{prop:self-intersection}.

Recall that the angular 1-form for the cylinder $Z_{\alpha,r}$ is defined by 
$$ \eta_\alpha = -(p_\alpha^\perp - \log r)dp_\alpha + ({\theta_\alpha - \arg r} ) d\theta_\alpha^\perp. $$
$\eta_\alpha$ is a multi-valued 1-form vanishing on $Z_{\alpha, r}$, and it satisfies $d\eta_\alpha = \Rea\Omega$. Now, for a closed loop $\gamma$ on $L_f$, we can consider the integral
$$\int_\gamma \eta_\alpha \quad\in\,\,\bR/{4\pi^2}$$
after lifting it to the universal cover of $(\bC^*)^2$. The value of the integral is uniquely determined modulo $4\pi^2$, and it only depends on the homology class of $\gamma$ because $d\eta_\alpha$ vanishes on $L_f$.

The main idea is that when there is a holomorphic map $u:(\Sigma, \round\Sigma) \to ((\bC^*)^2, L_f)$ from a Riemann surface with boundary $\Sigma$, it should satisfy
$$\int_{\round\Sigma} u^*\eta_\alpha = \int_{\Sigma} u^*\Rea\Omega = 0.$$
Therefore, if the integral of $\eta_\alpha$ along a closed loop $\gamma$ is not zero, we can conclude that there is no holomorphic map $u$ with $[u(\round\Sigma)] = [\gamma] \in H_1(L_f)$.

In reality, we use a Lagrangian $\widetilde{L}_{f, R}$ and its perturbation, which is distinct from $L_f$ (and fails to be special) outside a compact region. However, the stretching region will essentially divide a map $u: S_d \to (\bC^*)^2$ into several parts, where each part is either in the interior region or in the cylindrical region. To deal with interior parts, we additionally impose the assumption, dealing with holomorphic disks bounded by $L_f$ with an infinite end at one of the cylindrical ends $Z_{\alpha, r}^+$. Namely, we consider disks with boundary on the region $L_f \cup \{p_\alpha \ge R, \abs{p_\alpha^\perp - \log\abs{r}} \le 1 \}$. The boundary of the disk has two $S^1$ components at the cylindrical end, each corresponds to the circle $S_\alpha^1$ wraps around $Z_{\alpha,r}^+$, and the circle $S_{\alpha^\perp}^1$ in the orthogonal direction that determines the width of the strip. The remaining part is captured by the relative homology $H_1 (L_f, \{p_\alpha \ge R, \abs{p_\alpha^\perp - \log\abs{r}} \le 1\})$, which is isomorphic to a quotient of $H_1(L_f)$ by the circle $S_\alpha^1$.

\begin{assumption}\label{ass:top-bound-for-boundary}
    For each cylindrical end $Z_{\alpha,r}^+$ of $L_f$ and sufficiently large $R>0$, the boundary curve of any closed Riemann surface with boundary 
    $$(D,\round D) \subseteq ((\bC^*)^2, L_f \cup \{p_\alpha \ge R, \abs{p_\alpha^\perp - \log\abs{r}} \le 1 \}),$$ 
    with fixed $[S_{\alpha^\perp}^1]$-coefficient, has its homology class modulo $Z_{\alpha,r}^+$ contained in a finite subset of $H_1 (L_f)/[S_\alpha^1]$.
\end{assumption}

We will show that this assumption is true for model examples in Section~\ref{subsect:top-bound-for-bdry}. Assuming this, we have the following theorem that controls the strips whose endpoints $u(\zeta_0)$ and $u(\zeta_1)$ lie on the same cylindrical end.

%Unfortunately, we are not working with a special Lagrangian $L_f$, but rather a Lagrangian $\widetilde{L}_{f, R}$ which fails to become special on a compact region. Thus, we will have the desired result only for sufficiently large $R$ where such error gets sufficiently small. 
%Also, we first work with the case when the homology class $[\gamma]$ is bounded, i.e. that there are only a finite number of possibilities for $[\gamma] \in H_1(L_f)$, and later we will show that this is the case.

\begin{theorem}[Argument Constraints, for Strips]\label{thm:disk-bound-arg-strip}
    Let $f$ be a Laurent polynomial satisfying Assumption~\ref{ass:top-bound-for-boundary}. Consider holomorphic maps $u:S_2 \to (\bC^*)^2$ bounded by Lagrangians of the form $\widetilde{L}_{f,R}(-k_s H_{f,R})$, following Notation~\ref{not:disk-counted}. Also, assume that both $u(\zeta_0)$ and $u(\zeta_1)$ lie on the same cylindrical end $Z_{\alpha,r}^+$. Then, for the generic choice of $\{\varphi_\alpha\} \subseteq \bR$, the following is true:
    \begin{changemargin}{}{}
        For any positive integer $k$, there exists $R_0 > 0$ such that any holomorphic map $u:S_2 \to (\bC^*)^2$ bounded by Lagrangians $\widetilde{L}_{f,R}(-k_s H_{f,R})$ with $R > R_0$ and integers $k_s \le k$ should sasisfy either that

        \begin{enumerate}
            \item $u$ is contained in the cylindrical region $\{p_\alpha \ge R^2 + 3R\}$, or
            \item the pair $(u(\zeta_0), u(\zeta_1))$ is of the form $(x_\alpha^i x_\alpha^f, x_\alpha^j x_\alpha^e)$ with $i<j$.
        \end{enumerate}
    \end{changemargin}
    
\end{theorem}

One can easily find every $u$ contained in the cylindrical region, since it can be expressed as a product of two holomorphic strips on the complex plane, each corresponding to the $\bR_{>0}$-part and the $S^1$-part of the cylindrical end. See Remark~\ref{rem:strips-in-cyl} below.

To prove the theorem, we first lift all the configurations to the universal cover of $(\bC^*)^2$. The universal cover $(\widetilde{\bC^*})^2$ of $(\bC^*)^2$ is a Cartesian product of two complex planes $\bC_\alpha$ and $\bC_{\alpha^\perp}$, each with coordinates $(p_\alpha, \theta_\alpha)$ and $(p_\alpha^\perp,  \theta_\alpha^\perp)$, and the lift $\widetilde{u}: S_2 \to (\widetilde{\bC^*})^2$ can also be expressed as a product of two holomorphic maps $u_\alpha: S_2 \to \bC_\alpha$ and $u_\alpha^\perp: S_2 \to \bC_{\alpha^\perp}$. Note that under the identification $\bC_\alpha \times \bC_{\alpha^\perp} \cong (\widetilde{\bC^*})^2$, the symplectic forms of these complex planes are
\begin{align*}
    \omega_\alpha = d\theta_\alpha \wedge d p_\alpha \quad \text{and} \quad  \omega_{\alpha^\perp} = d p_\alpha^\perp \wedge d\theta_\alpha^\perp
\end{align*}
and the complex structure $J:T_* (\widetilde{\bC^*})^2 \to T_* (\widetilde{\bC^*})^2 $ satisfies
\begin{align*}
    J \left(\frac{\round}{\round \theta_\alpha} \right) = \abs{\alpha}^2 \frac{\round}{\round p_\alpha} \quad \text{and} \quad J \left(\frac{\round}{\round p_\alpha^\perp} \right) = \abs{\alpha}^2 \frac{\round}{\round \theta_\alpha^\perp}.
\end{align*}
In other words, $\bC_\alpha$ and $\bC_{\alpha^\perp}$ are isomorphic to $\bC$ with the standard complex coordinates $({\abs{\alpha}}^{-1} \theta_\alpha, \abs{\alpha}p_\alpha )$ and $({\abs{\alpha}}^{-1}p_\alpha^\perp,  \abs{\alpha}\theta_\alpha^\perp)$, respectively.

This leads to a clearer description of the image $u(S_2)$. First, on the plane $\bC_\alpha$, each lift of the cylindrical end $Z_{\alpha,r}^+$ is described as a straight line, separated from each other by $2\pi$. The lifts of $\widetilde{L}_{f, R} (-k H_{f, R})$ have a similar configuration, except that each line is initially translated in $\theta$ direction by $k R^{-1}a_\alpha$, and then translates again by $2k \pi$ on the region $p_\alpha \ge R^2 + 3R$. If we pick $R$ sufficiently large so that $kR^{-1}\abs{a_\alpha}$ is less than $\pi$, and if we denote the lift of $\widetilde{L}_{f, R} (-k H_{f, R})$ starting at $(2s-1)\pi < \theta_\alpha < (2s+1)\pi$ by $l_{s,k}$, then the generators $x_\alpha^j x_\alpha^e$ and $x_\alpha^j x_\alpha^f$ in $CF(\widetilde{L}_{f,R}(-k_0 H_{f,R}), \widetilde{L}_{f,R}(-k_1 H_{f,R}))$ lie in the intersection of the lifts $l_{s,k_0}$ and $l_{s+j, k_1}$. In particular, if the map $u$ in Theorem~\ref{thm:disk-bound-arg-strip} has input $x_\alpha^{j_0} x_\alpha^e$ and output $x_\alpha^{j_1} x_\alpha^e$, then the sum of the width of the strips in the region $\{R \le p_\alpha \le R^2 + 3R\}$ is exactly $2(j_1 - j_0)\pi$.

\vspace{0.3em}
\begin{figure}[ht]
    \centering
    \includegraphics[width=0.7\textwidth]{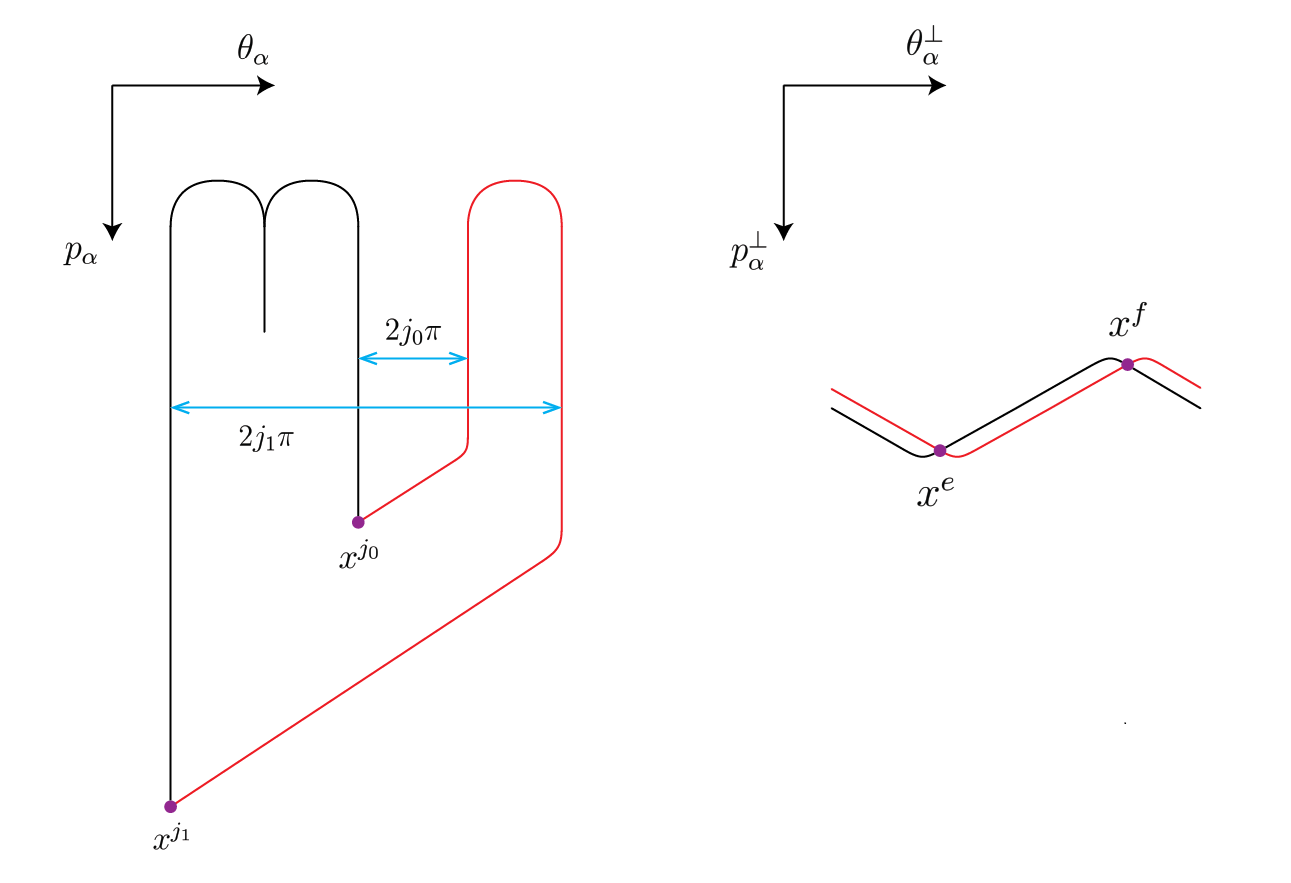}
    \caption{Example configuration for strips}
    \label{fig:configuration-for-strips}
\end{figure}
\vspace{0.3em}

\begin{remark}\label{rem:strips-in-cyl}
    If $u$ is contained in the cylindrical region $\{p_\alpha \ge R^2 +3R\}$, then due to the shape of Lagrangians in $\bC_\alpha$, the map $u_\alpha$ should be a constant map. For every $j$ we have two such $u$ with $(u(\zeta_0), u(\zeta_1)) = (x_\alpha^j x_\alpha^f, x_\alpha^j x_\alpha^e)$, each corresponding to a strip between two $S^1$ on the right side of Figure~\ref{fig:configuration-for-strips}.
\end{remark}

The main idea to prove this theorem is to divide the image of $u$ into the parts either in the special Lagrangian $L_f$ or the cylindrical end. Namely, we ``cut" $u$ using several line segments in $\theta$-direction on the $\bC_\alpha$ plane. Then, we show that for any $\varepsilon > 0 $, there exists $R>0$ and a suitable choice of cuts satisfying the following.
\begin{itemize}
    \item Along each cut, the $\theta_\alpha^\perp$ value does not vary more than $\varepsilon$. 
    \item For each part $u_j:\Sigma \to (\widetilde{\bC^*})^2$ in the interior region, the boundary $\round \Sigma$ contains exactly one cut, with $\theta_\alpha^\perp$-value contained in the interval $[\varphi'_j, \varphi'_j+\varepsilon]$, and if $[\gamma]$ is the homology class of $\round\Sigma$ in $L_f$, we have
    \begin{equation}
%        \abs*{\int_{\round \Sigma} \eta_\alpha - \left(2\pi(k_1 - k_0)\varphi'_j + \int_{[\gamma]} \eta_\alpha \right)} < \varepsilon.
        \abs*{ \,2\pi k'\varphi'_j + \int_{[\gamma]} \eta_\alpha\,} < \varepsilon
    \end{equation}
    for some integer $k' \le k$.
    \item At least one of $\varphi'_j$ should be $\varepsilon$-close to $\varphi_\alpha$.
\end{itemize}

Since the assumption~\ref{ass:top-bound-for-boundary} implies that there are only finitely many possible values for $\int_{[\gamma]} \eta_\alpha $, we can prove the non-existence of $u$ for all but finitely many values (modulo $2\pi$) of $\varphi_\alpha$.

To prove the existence of such cuts, we first consider the energy of $u$. Since both Lagrangians have $Z_{\alpha,r}^+$ as a cylindrical end (Remark~\ref{rem:cyl-end-desc-for-perturbation}), the energy depends on the endpoints and the homology class of the boundary. In particular, on the region $\{p_\alpha \le R^2 + 3R\}$ (from the interior to the stretching region, excluding the wrapping part), the energy of $u$ from Theorem~\ref{thm:action-ftnal} is divided into three different parts:
\begin{itemize}
    \item The line integral $\int_{[\round u(S_2)]_\alpha}\lambda_\alpha$. By assumption, this value has upper and lower bounds independent of $R$.
    \item The difference between Morse functions, which is of the form $R^{-3}(\mu_\varphi(\theta_1) - \mu_\varphi(\theta_2))$ for some $\varphi, \theta_1$ and $\theta_2$. There are at most $k_1 - k_0$ such terms, and their value can be arbitrarily small if we choose $R$ large enough.
    \item The term $2(j_1 - j_0)\pi(R^2 + 3R)$. Note that this term corresponds to the area of the strip-like part on $\bC_\alpha$.
    %\item The term $2\pi \sum j_s p_s$, which is equal to $2(j_1 - j_0)\pi(R^2 + 3R)$. This term appears because we are using several different lifts of the Lagrangian. Each lift differs from another by $2\pi$ in $\theta_\alpha$ direction, which adds an integer multiple of the linear term $2\pi p_\alpha$ on each $g_j$ in the Theorem~\ref{thm:action-ftnal}. By the definition of $x_\alpha^j x_\alpha^e$ and $x_\alpha^j x_\alpha^f$, the sum of these integer multipliers should be exactly $k_1 - k_0$. Note that this term corresponds to the area of the strip-like part on $\bC_\alpha$.
\end{itemize}
In particular, the following lemma shows that the energy of the map $u_2: S_2 \to \bC_{\alpha^\perp}$ has an upper bound independent of $R$.

\begin{lemma}\label{lem:u2-strip-is-small}
    The area of $u_\alpha(S_2) \cap \{p_\alpha \le R^2 + 3R\}$ is at least 
    $$(R^2 +3R - b_2 M) 2(j_1 - j_0) \pi - 3k R^{-2},$$
    where $b_2 M$ is from Proposition~\ref{prop:trop-approx} and independent of $k$ and $R$.
\end{lemma}
\begin{proof}
    Consider the function $\text{ord}_{u_\alpha}: \bC_\alpha \to \bZ _{\ge 0}$ that assigns each point $(p_\alpha, \theta_\alpha)$ the number of elements in the inverse image $u_\alpha^{-1}(p_\alpha, \theta_\alpha) \setminus \round S_2$ counted with order. Since any non-constant holomorphic map is an open mapping, this function is constant on every connected component of $\bC_\alpha \setminus u_\alpha (\round S_2)$. Now, since any two distinct lifts of the same Lagrangian are disjoint in the cylindrical region and $u(\zeta_0), u(\zeta_1)$ are only intersection points of two different Lagrangians on the boundary, the function $\text{ord}_{u_\alpha}$ is constant on each connected component of the cylindrical region $\{p_\alpha > b_2 M\}$ minus the four lifts passing through $u(\zeta_0)$ and $u(\zeta_1)$, possibly except for the boundary $u(\round S_2)$.
\end{proof}

\begin{remark}\label{rem:slits}
    The boundary of $u_\alpha$ in the cylindrical region may pass through other lifts, even though they cannot make new connected components. In this case, the boundary curve turns back in the middle of the lift, creating a slit as in the left part of Figure~\ref{fig:configuration-for-strips}.
\end{remark}

The above lemma also shows that the map $u_\alpha$ in the stretching region $\{3R < p_\alpha < R^2 +3R\}$ can be divided into several holomorphic embeddings. To achieve that, we divide its rectangular image using horizontal lines $\{p_\alpha = t\}$ passing through singular points of $u_\alpha$, and vertical lines $\{\theta_\alpha = \varphi\}$ representing the lift of one of the Lagrangians. We have finitely many horizontal and vertical lines because singular points on a non-constant holomorphic map are isolated. Therefore, this will divide the image into finitely many rectangular regions $\Sigma$ such that $u_\alpha: u_\alpha^{-1}(\Sigma) \to \Sigma$ is a homeomorphism on each connected component of $u_\alpha^{-1}(\Sigma)$. In conclusion, we get the following collection of maps. 
%of the image with the horizontal line $\{p_\alpha = t\}$ on the stretching region, i.e. where $3R \le t \le R^2 +3R$, is a union of line segments with total length $2(k_1 - k_0)\pi$ (counted with order) except for finitely many $t$ that passes through a singular value of $u_\alpha$. Note that these line segments can be divided further if there are slits as in Remark above.
\begin{lemma}\label{lem:existence-cuts-rect}
    There exists a collection $\{S_{i,j}\}_{1 \le i \le m, 1 \le j \le n}$ of pairwise disjoint open subsets of $S_2$ and real numbers
    $$3R = t_1 < t_2 < \cdots < t_m = R^2 +3R, \quad \varphi_{j,1} < \varphi_{j,2} $$
    that satisfy the following.
    \begin{enumerate}
        \item The restriction $u_\alpha \mid_{S_{i,j}}$ is a holomorphic embedding onto the rectangle
        $$\Sigma_{i,j} = \{t_{i} < p_\alpha < t_{i+1}, \,\, \varphi_{j,1} < \theta_\alpha < \varphi_{j,2} \} $$
        \item The sum of the length of the intervals $\sum_{j=1}^n (\varphi_{j,2} - \varphi_{j,1})$ is equal to $2(k_1 - k_0)\pi$.
    \end{enumerate}
\end{lemma}

$\hfill \square$

Now, each holomorphic embedding $u_\alpha \mid_{S_{i,j}}$ defines a map $u_{\alpha^\perp} \circ \left(u_\alpha \mid_{S_{i,j}}\right)^{-1} : \Sigma_{i,j} \to \bC_{\alpha^\perp}$. We showed that the area of $u_{\alpha^\perp}$ has an upper bound independent of $R$, while the total height (the length in $p_\alpha$-direction) of $\Sigma_{i,j}$ is proportional to $R^2$. Hence, if we choose $R$ large enough, then the image of a line segment in $\theta_\alpha$-direction in $\Sigma_{i,j}$ under the map $u_{\alpha^\perp} \circ \left(u_\alpha \mid_{S_{i,j}}\right)^{-1}$ will become small, which shows the existence of the desired cut.
%Hence, if we choose $R$ large enough, then the area of $u_\alpha^\perp \mid_{S_{i,j}}$ will become small, which will show the existence of the desired cut.

\begin{lemma}\label{lem:existence-cuts}
    For any positive integer $k$ and real number $\varepsilon > 0$, there exists sufficiently large $R > 0$ such that for any $u:S_2 \to (\bC^*)^2$ satisfying the assumption in Theorem~\ref{thm:disk-bound-arg-strip}, there exists a number $t \in [3R,  R^2+3R ]$ such that 
    \begin{itemize}
        \item $u_\alpha^{-1} (\{p_\alpha = t\})$ is a union of piecewise smooth curves in $S_2$, and
        \item The length of $u_{\alpha^\perp} \circ u_\alpha^{-1} (\{p_\alpha = t\})$ is less than $\varepsilon$.
    \end{itemize}
\end{lemma}

\begin{proof}
    We start from the holomorphic maps $f_{i,j} := u_{\alpha^\perp} \circ \left(u_\alpha \mid_{S_{i,j}}\right)^{-1} : \Sigma_{i,j} \to \bC_{\alpha^\perp}$. By Lemma~\ref{lem:u2-strip-is-small}, there exists a number $M_{\alpha^\perp} >0$ independent of $R > 0$ such that the image of $u_{\alpha^\perp}$ has area less than $M_{\alpha^\perp}$. This implies that
    $$\sum_{i,j} \int_{\Sigma_{i,j}} \abs{f'_{i,j}}^2 \text{dvol}_{\bC_\alpha} < M_{\alpha^\perp}. $$
    Therefore, there exists $t \in [3R,  R^2+3R ]$ such that
    $$\sum_j \int_{\varphi_{j,1}}^{\varphi_{j,2}} \abs{f'_{i,j}}^2 d\theta_\alpha < M_{\alpha^\perp} R^{-2},$$
    where each integral is taken along the slice $\Sigma_{i, j} \cap \{p_\alpha = t\}$. This implies that
    \begin{align*}
        \sum_j \int_{\varphi_{j,1}}^{\varphi_{j,2}} \abs{f'_{i,j}} d\theta_\alpha & \le \left( \sum_j \int_{\varphi_{j,1}}^{\varphi_{j,2}} \abs{f'_{i,j}}^2 d\theta_\alpha \right)^{1/2} \cdot (2(k_1 - k_0) \pi)^{1/2}\\
        & \le M_{\alpha^\perp}^{1/2} R^{-1} \cdot (2(k_1 - k_0) \pi)^{1/2}.
    \end{align*}
    This proves the lemma since the left-hand side is the length of $u_{\alpha^\perp} \circ u_\alpha^{-1} (\{p_\alpha = t\})$.
\end{proof}

Now we fix $\{p_\alpha = t\}$ from the lemma, and consider one of the map $\bar{u}_{\text{int}}:\bar{S} \to (\bC^*)^2$ cut by this real hyperplane, lying in the component $\{p_\alpha \le t\}$ containing the interior region of $L_f$. The boundary curve of this map consists of
\begin{itemize}
    \item $\gamma_{\text{int}}$, the part of the boundary of $u$ in the interior region $\{p_\alpha \le t\}$.
    \item the curve $\gamma_{\text{cut}}$ along $\{p_\alpha = t\}$
\end{itemize}
By the definition of the cut, the $\theta_\alpha^\perp$ values along $\gamma_{\text{cut}}$ lie on an interval with length less than $\varepsilon$. After fixing a lift, suppose the interval is $[\varphi', \varphi' +\varepsilon]$. Fixing a value of $\theta_\alpha^\perp$ also fixes a lift of $\eta_\alpha = -(p_\alpha^\perp - \log |r|)dp_\alpha - ({\theta_\alpha - \arg r} ) d\theta_\alpha^\perp.$ 

Since $\Bar{u}$ is a holomorphic map, we have
\begin{equation}\label{eqn:angular-1-form-for-cut}
    \int_{\gamma_{\text{cut}}}\eta_\alpha + \int_{\gamma_{\text{int}}}\eta_\alpha = 0.
\end{equation}
The former integral can become arbitrarily small by shrinking $\varepsilon$ in Lemma~\ref{lem:existence-cuts} since $dp_\alpha$ is zero along $\gamma_{\text{cut}}$, and the integral of $\abs{d\theta_\alpha^\perp}$ is less than the length of $u_{\alpha^\perp} \circ u_\alpha^{-1} (\{p_\alpha = t\})$. Also, the latter integral can be computed similarly to $\lambda_\alpha$ in the section~\ref{subsect:energy}.

\begin{lemma}\label{def:integral-eta-alpha}
    Suppose that $Z_{\alpha,r}^+$ is a cylindrical end of a special Lagrangian $L$ with $\abs{\alpha}=1$, and $[\gamma]_\alpha \in H_1(L)/[S_\alpha^1]$ represented by a loop $\gamma_\circ:[0,1] \to L$. If we denote the $\theta_\alpha^\perp$ value of the endpoints $\gamma_\circ(0) = \gamma_\circ(1)$ by $\theta_{\alpha,0}^\perp$, and let $\widetilde{\gamma_\circ}:[0,1] \to (\widetilde{\bC^*})^2$ be a lift of $\gamma_\circ$ to the universal cover of $(\bC^*)^2$, then the integral
    $$\int_{[\gamma]_\alpha} \eta_\alpha := \int_{\widetilde{\gamma_\circ}} \eta_\alpha - 2\pi w_\alpha(\gamma_\circ) \theta_{\alpha, 0}^\perp \,\,\in\,\, \bR/4\pi^2$$
    depends only on the homology class $[\gamma]_\alpha$.
\end{lemma}

\begin{remark}\label{rem:integral-eta-alpha-bigger}
    Similar result holds for the case $\abs{\alpha} > 1$, but one needs a careful choice of the path $\gamma_\circ$ and its $\theta_\alpha^\perp$-value. Namely, the endpoints of $\widetilde{\gamma_\circ}$ should have the same $\theta_\alpha$-value while their $\theta_\alpha^\perp$-value differ by $2\pi w_\alpha(\gamma_\circ)$. Such path is not a loop when the wrapping number is not a multiple of $\abs{\alpha}^2$. In addition, instead of using the $\theta_\alpha^\perp$-value of $\widetilde{\gamma_\circ}(0)$, we fix a point $X$ on $Z_{\alpha, r}^+$ and set $\theta_{\alpha, 0}^\perp$ to be the difference of the $\theta_\alpha^\perp$-value between $\widetilde{\gamma_\circ}(0)$ and the lift of $X$ on the same lift of $Z_{\alpha, r}^+$.
\end{remark}

One additional ingredient we would like to use is that the length of the boundary curve is bounded by a constant times its area. This \textit{reverse isoperimetric inequality} is proved when the domain of $u$ has no corner \cite{GrSo1}. The constant is determined by the geometry of the symplectic manifold and the curvature of the Lagrangian. Since the part we are interested in is the cylindrical region, we can modify the proof to show the following.

\begin{lemma}\label{lem:reverse-isoperimetric}
    There exists a fixed constant $N$ independent of $f$ and $R$ above such that for any cylindrical end $Z_{\alpha,r}^+$ of $L_f$ and a map $u:S_2 \to (\bC^*)^2$ in Theorem~\ref{thm:disk-bound-arg-strip},
    $$(\text{length of }u(\round S_d) \cap \{p_\alpha \ge R\}) \le N \cdot (\text{area of }u(S_d)  \cap \{p_\alpha \ge R\} ). $$
\end{lemma}
The estimates in the lemma rely on the injectivity radius assumption. In particular, it requires that the map $u$ has no `thin' part, i.e. each connected component in $\{3R \le p_\alpha \le R^2 +3R\}$ should not have width $O(R^{-1})$. We don't have this problem in our case since the width of each strip is a multiple of $2\pi$.

Even though $\Rea\Omega$ does not vanish on our Lagrangian in the cylindrical region, this lemma shows that the length of $\gamma_{\text{int}}$ inside the cylindrical region is $O(R^2)$. On the other hand, we have
\begin{equation}
    \eta_\alpha |_{\{p_\alpha \ge R\}} = kR^{-1}a_\alpha d\theta_\alpha^\perp + O(R^{-3})
\end{equation}
from Proposition~\ref{prop:trop-approx} (3) and Remark~\ref{rem:cyl-end-desc-for-perturbation}. Note that the proposition shows that the contribution of the cut-off in the region $[R,2R]$ is bounded by polynomial times $\exp(-b_3 R)$, hence $O(R^{-3})$. Also, the first term on the right side comes from the linear term $R^{-1}a_\alpha p_\alpha$ in the Hamiltonian $H_{f, R}$ (and the corresponding Hamiltonian perturbation in $\theta_\alpha$-direction). The integral of this $1$-form along $\gamma_{\text{int}}$ is also $O(R^{-1})$ since it is bounded by $kR^{-1}a_\alpha$ times the change of $\theta_\alpha$-value along the cut $\{p_\alpha = t\}$. Hence, this shows that
\begin{equation}\label{eqn:eta-alpha-plus-2piphi}
    \abs*{\int_{[\gamma]_\alpha} \eta_\alpha + 2\pi w_\alpha(\gamma) \varphi'} < \varepsilon
\end{equation}
if we choose $R > 0$ large enough.

%\begin{lemma}[Theorem~\ref{thm:action-ftnal}, for $\eta_\alpha$]
%    Consider a holomorphic map $u:S_d \to (\bC^*)^2$ in the Notation~\ref{not:disk-counted} and assume that each $\zeta_j$ maps to a point inside the same cylindrical end $\{p_\alpha \ge R, \abs{p_\alpha^\perp - \log{\abs{r}}} \le 1\}$ of $L_f$. Then, for any $\varepsilon > 0$ and positive integer $k$, we can pick large enough $R > 0$ so that for any such $u$ with $k_0 < \cdots < k_d < k$, we have
%    \begin{align}
%        \abs*{\,\,\sum_{s=0}^d\int_{[\gamma_j]_\alpha}\eta_\alpha + 2\pi \sum_{s=0}^d j_s \theta_s^\perp\,\,} < \varepsilon,
%    \end{align}
%\end{lemma}

\begin{proof}[Proof of Theorem~\ref{thm:disk-bound-arg-strip}]

    Choose a number $t \ge 3R$ that gives a cut satisfying the condition in Lemma~\ref{lem:existence-cuts}, and let $\Bar{u}_{\text{cyl}}:\Bar{S} \to (\bC^*)^2$ be the part of $u$ inside the area $\{p_\alpha \ge t\}$. 
    \vspace{0.3em}
    \begin{figure}[ht]
        \centering
        \includegraphics[width=0.7\textwidth]{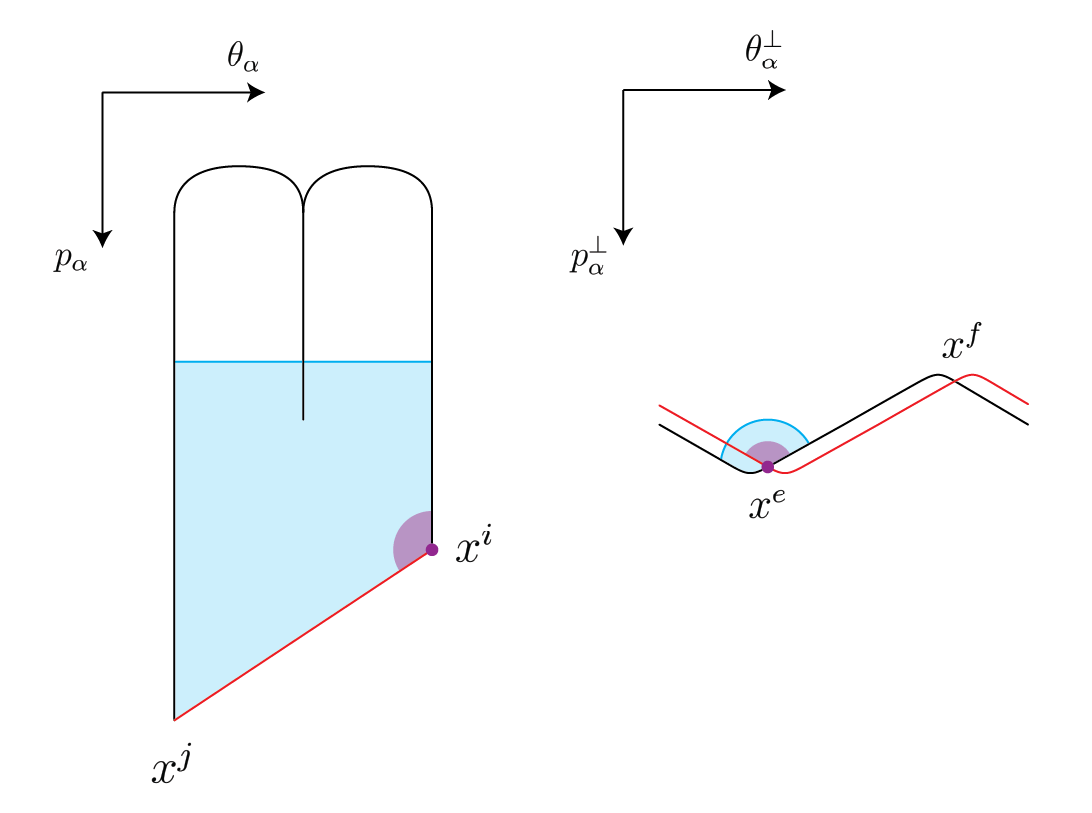}
        \caption{The map $\bar{u}_\text{cyl} : \bar S \to (\bC^*)^2$ in the proof of Theorem~\ref{thm:disk-bound-arg-strip}, when the input and output generators are of the form $x_\alpha^j x_\alpha^e$. The highlighted point and the shaded region represent the image of $u$ near $\zeta_0$.}
        \label{fig:proof-disk-bound-arg-strip}
    \end{figure}
    \vspace{0.3em}
    In the region $\{p_\alpha \ge 3R\}$, the Lagrangian $\widetilde{L}_{f,R}$ and its perturbations by $H_{f,R}$ can be represented as the product of a curve in $\bC_\alpha$ and a loop in $\bC_{\alpha^\perp}$, as in Figure~\ref{fig:proof-disk-bound-arg-strip}. Now, since the $\bC_{\alpha^\perp}$-part of the pair of generators $(u(\zeta_0), u(\zeta_1))$ is either $(x^e, x^e )$, $(x^e, x^f)$, or $(x^f, x^f)$, the image of $\Bar{u}_{\text{cyl}}$ in $\bC_{\alpha^\perp}$ contains a half-neighborhood of the intersection point $\theta_\alpha^\perp = \varphi_\alpha + \frac{1}{2}k'R^{-1}a_\alpha^\perp$. This implies that the $\{p_\alpha = t\}$ part of the boundary of $\Bar{u}_{\text{cyl}}$ should pass through this value of $\theta_\alpha^\perp$.

    Now, the equation~(\ref{eqn:eta-alpha-plus-2piphi}) implies that
    \begin{equation}
        \abs*{\int_{[\gamma]_\alpha} \eta_\alpha + 2\pi w_\alpha(\gamma) \varphi_\alpha} < 2\varepsilon
    \end{equation}
    for some homology class $[\gamma]_\alpha \in H_1(L_f)/[S_\alpha^1]$. Since Assumption~\ref{ass:top-bound-for-boundary} implies that there are finitely many possibilities for $[\gamma]_\alpha$, we can pick the value of $\varphi_\alpha$ avoiding $2\varepsilon$-neighborhood of the integral of $\eta_\alpha$ along these homology classes if we choose $\varepsilon$ small enough.
\end{proof}

%This theorem implies that when restricted to a single cylindrical end, $\mu^d$ works similarly to the one with product cylinder $R \times S^1 \subseteq \bC^* \times \bC^*$. In particular, we have the following result for the differential $\mu^1$ and the multiplication $\mu^2$.

%\begin{corollary}
%    For the generic choice of $\{\varphi_\alpha\} \subseteq \bR$, The following is true.
%    \begin{itemize}
%        \item $\mu^1 (x_\alpha^j x_\alpha^e)$ is generated by the interior generators $v_1, v_2, \cdots, v_l$.
%        \item $\mu^2 (x_\alpha^i x_\alpha^e, x_\alpha^j x_\alpha^e)$ is equal to $x_\alpha^{i+j} x_\alpha^e$ plus a linear combination of $v_1, v_2, \cdots, v_l$.
%    \end{itemize}
%\end{corollary}

%\subsection{Basic Properties}
%ssssssssssssss Assumption Here

%\begin{proposition}
    
%\end{proposition}

%\section{Temp}
%\cite{Hic1}

%%%%%%%%%%%%%%%%%%%%%Fukaya category definition

\section{Floer Theory for Tropical Lagrangians}\label{sect:floer-theory}
    This section aims to define a Fukaya-type category for tropical Lagrangians, mainly by constructing a Floer theory between them. Since we are using a family of cylindrical Lagrangians $\widetilde{L}_{f, R}$ instead of the original special Lagrangian $L_f$, we will first consider the Floer theory for $\widetilde{L}_{f, R}$ and then take a limit of them. Also, because of the technical difficulty of directly implementing the entire $A_\infty$ structure via this limit, we will use the localization approach of Abouzaid-Seidel in \cite{AS1}.

    We continue to assume that all the Laurent polynomial $f$ throughout the section satisfy Assumption~\ref{ass:trop-ftn} and Assumption~\ref{ass:cyl-end-is-not-diagonal}, so that the intersection points between their Hamiltonian perturbation can be expressed as in Proposition~\ref{prop:self-intersection}.

\subsection{Transversality and Bounding Holomorphic Disks}
    The first thing to show is that the Floer theory for the Lagrangians $\widetilde{L}_{f, R}$ is well defined. Due to the choice of Lagrangians, Hamiltonians, and the complex structure, several issues may occur when following the standard way to define the Floer theory.

    For simplicity, we will only consider Lagrangians not bounding a holomorphic disk in this section. Also, the Gromov compactness argument requires holomorphic disks to be contained in a compact set. In our setting, we have the following compactness result.
    
    \begin{theorem}\label{thm:Gromov-compactness}
        For any finite number of Lagrangians $L_j$ of the form $\widetilde{L}_{f, R}(-kH_{f, R})$, there exists a compact subset of $(\bC^*)^2$ that contains the image of any holomorphic disks $u:S_d \to (\bC^*)^2$ bounded by $\cup_{j} L_j$.
    \end{theorem}
    \begin{proof}
        It is enough to show that the image of $u$ should lie inside $\{p_\alpha \le R_{\max}^2 +3R_{\max}\} $ on the region $\{\abs{p_\alpha^\perp - \log\abs{r}} \le 1\}$, where $R_{\max}$ is the maximum value of $R$. Outside this set, each Lagrangian is a ray in positive $p_\alpha$-direction in the $(p_\alpha, \theta_\alpha)$-plane. Now, suppose that the function $p_\alpha \circ u$ has its maximum bigger than $R_{\max}^2 +3R_{\max}$. By the maximum modulus principle, the maximum should be achieved on one of the boundary points. However, since $p_\alpha$ is the real part of the holomorphic function
        $$g(z) = \frac{1}{\abs{\alpha}^2} (\alpha_2 \log(z_1) - \alpha_1 \log(z_2)), $$ we can apply the Schwarz reflection principle on $\exp (g) \circ u$ to holomorphically extend over the boundary, contradicting the maximality.
    \end{proof}

    In addition, the Lagrangian and the Hamiltonian construction guarantees the following transversality property, which allows us to define the Floer theory and the $A_\infty$ operations between the Lagrangians for large enough $R > 0$.

    \begin{theorem}\label{thm:transversality-main}
        For any two distinct pairs $(f_1, k_1)$ and $(f_2, k_2)$ of a Laurent polynomial and an integer, two Lagrangians
        $$\widetilde{L}_{f_1 , R}(-k_1 H_{f_1, R}), \quad \widetilde{L}_{f_2, R}(-k_2H_{f_2, R})$$
        intersect transversely for large enough $R > 0$ for a generic choice of the translation vector $(a_1, a_2)$.
    \end{theorem}

    \begin{proof}
        On the interior region where $\widetilde{L}_{f, R}$ is equal to $L_f$, after HyperKähler rotating and translating the picture, it is reduced to the case when both are complex curves and one of them is translated by $R^{-1}(a_1, a_2)$ for a fixed vector $(a_1, a_2)$ in the $(\theta_1, \theta_2)$-plane. The case $f_1 = f_2$ comes from Lemma~\ref{lem:Riemann-translation-intersection}. If $f_1$ and $f_2$ are distinct, there are only finitely many intersection points where we can separate individually. Indeed, in a local chart, both curves can be written as a graph of holomorphic functions $g_1(z)$ and $g_2(z)$, and the translation corresponds to adding an (imaginary) constant on one of these functions. Since the non-transverse intersection points happen when $g_1(z) = g_2(z)$ and $g_1'(z) = g_2'(z)$, we can achieve transversality for large enough $R > 0$ by avoiding zeros of $g_1' - g_2'$.

        Outside the interior region, these Lagrangians do not intersect unless they share the cylindrical ends with the same $(p_1, p_2)$-direction. In this case, we can identify both Lagrangians as sections $\sigma_1, \sigma_2$ of the cotangent bundle $T^*Z_{\alpha,r}^+$. Let $a_{\alpha,j}$ and $a_{\alpha,j}^\perp$ (for $j=1,2$) be the $a_\alpha$ and $a_\alpha^\perp$ components determined by the translation vector $(a_1, a_2)$ of each Lagrangian. Now, on the region $\{R \le p_\alpha \le R^2 + 3R \}$, there exist points $r_1, r_2$ on the fiber such that $\sigma_j = r_j + k_j R^{-1} a_{\alpha,j}^\perp + O(R^{-2})$ for $j=1,2$, and they do not intersect for large $R>0$ if we choose the translation vector $(a_1, a_2)$ so that $k_1 a_{\alpha,1}^\perp \neq k_2 a_{\alpha,2}^\perp$. Similarly, on the region $\{p_\alpha \ge R^2 + 3R\}$ where the wrapping happens, we can achieve the transversality for the case when $k_1 a_{\alpha,1} \neq k_2 a_{\alpha,2}$.
    \end{proof}

\subsection{Floer Theory for Tropical Lagrangians}\label{subsect:FloerThoery-basic}
    Now, we define the Floer theory for the Lagrangians $\widetilde{L}_{f, R}$ with cylindrical ends. This part follows the standard approach, and the reader may find the details in \cite{Sei1}. In this section we only consider the Floer theory between $\widetilde{L}_{{f_1}, R}$ and $\widetilde{L}_{{f_2}, R}$ with the same $R>0$. After taking a limit $R \to \infty$, this will describe the Floer theory for the corresponding special Lagrangians $L_{f_1}$ and $L_{f_2}$.

    Suppose that $L_0$ and $L_1$ are transverse Lagrangians of the form $\widetilde{L}_{f, R}(-kH_{f, R})$, not bounding any holomorphic disks. Let $\mathcal{J}$ be the space of $\omega$-compatible almost complex structures. For each $J \in \mathcal{J}$ and $x_0, x_1 \in L_0 \cap L_1$, let $\widehat{\mathcal{M}}(x_0, x_1; J)$ be the moduli space of $J$-holomorphic strips $u: \bR \times [0,1] \to (\bC^*)^2$ satisfying
    \begin{equation}
        \begin{cases}
            u(\bR \times \{0\} ) \subseteq L_0,\quad u(\bR \times \{1\} ) \subseteq L_1,\\
            \lim_{s \to -\infty} u(s,t)  = x_1, \quad \lim_{s \to \infty} u(s,t)  = x_0.
        \end{cases}
    \end{equation}
    
    Let $\mathcal{M}(x_1, x_0; J)$ be the quotient of $\widehat{\mathcal{M}}(x_1, x_0; J)$ by the $\bR$-action given by the reparametrization of $u$ (translation in $t$-direction). Also, let $\overline{\mathcal{M}}(x_1, x_0; J)$ be the Gromov compactification of $\mathcal{M}(x_1, x_0; J)$. The only added elements on the compactification we are interested in are broken holomorphic strips: disk bubbling does not occur since $L_0$ and $L_1$ do not bound any disk, and sphere bubbling is generically a codimension 2 phenomenon. The standard Gromov compactness argument can be applied to show that $\mathcal{M}(x_1, x_0; J)$ is compact since all disks in $\mathcal{M}(x_1, x_0; J)$ lie in the compact region as shown in Theorem~\ref{thm:Gromov-compactness}.
    %from the observation in [reference]]]]]] that all disks in $\mathcal{M}(x_1, x_0; J)$ lie in the compact region $\cup_{\alpha}\{p_\alpha \le R^2 +3R\}$.

    %[Mention that disks lie inside a compact subset]]]]]

    A generic $\omega$-compatible almost-complex structure $J$ is regular, which implies that the moduli space $\overline{\mathcal{M}}(x_0, x_1; J)$ is a smooth manifold. In this case, Gromov compactness guarantees that the zero-dimensional part of $\overline{\mathcal{M}}(x_0, x_1; J)$, which is equal to the zero-dimensional part of $\overline{\mathcal{M}}(x_0, x_1; J)$, consists of a finite number of points. 

    Since the Lagrangians we consider are not exact, we are required to work with Floer chain complexes over the \textit{Novikov field} defined by 
    $$\Lambda := \left\{ \sum_{i=0}^\infty a_i T^{\lambda_i} \,\,\mid\,\, \lambda_i \in \bR, \,\,a_i \in \bC,\,\, \lim_{i \to \infty } \lambda_i = +\infty \right\}$$
    Also, each Lagrangian we consider is equipped with a \textit{local systems}, a vector bundle over the Lagrangian with unitary holonomy over the Novikov field. This assigns a unitary element $\text{hol}(\gamma)$ in the coefficient field for each homology class of loops $\gamma$ in the Lagrangian (See, for example, Remark 2.11 in \cite{Aur1}).

    %As noted in Remark~[remark in arg argument section]]]], we impose an additional condition that $J$ is sufficiently close to the canonical complex structure $J_{\text{std}}$ inside the region $\cup_{\alpha}\{p_\alpha \le R^2 +3R\}$.
    %and sufficiently close to $J_{\text{std}}$ inside. This is possible because elements of $\overline{\mathcal{M}}(x_0, x_1; J)$ 
    %not passing through the region should be contained entirely inside one of the cylindrical ends $\{p_\alpha > R^2 + 3R\}$, on which both $L_0$ and $L_1$ are honest cylinders. 
    
    Now, we define $CF(L_0, L_1)$ as the vector space over $\Lambda$ generated by the intersection points of $L_0$ and $L_1$. Also, for a given choice of a regular almost-complex structure $J_{L_0, L_1}$, if we let $\mathcal{M} (x_1, x_0; [u], J_{L_0, L_1})$ be the subset of $\mathcal{M} (x_1, x_0; J_{L_0, L_1})$ containing elements with homology class $[u]$, there is a differential $\mu^1$ on $CF(L_0, L_1)$ defined by
    $$\mu^1 (x_1) = \sum_{\substack{[u];\, x_0 \in L_0 \cap L_1}} \#\mathcal{M} (x_1, x_0; [u],J_{L_0, L_1})T^{E([u])} \text{hol}([\round u])  x_0$$
    where $E(u)$ is the geometric energy of $u$ and $\text{hol}(\round u)$ is the holonomy determined by the local system. We only count the zero-dimensional part of the moduli space $\mathcal{M} (x_1, x_0; J_{L_0, L_1})$. The analysis on the $1$-dimensional part of $\mathcal{M} (x_0, x_1; J_{L_0, L_1})$ gives that $(\mu^1)^2 = 0$. We will also rescale each generator to simplify the calculation; this is explained in the last part of this section.

    To extend this notion and obtain a graded complex, we equip each Lagrangian with a \textit{brane structure}. A brane structure consists of two additional data in addition to the Lagrangian $L$. First, to obtain a graded complex, we begin from a phase map $L \to S^1$ on each Lagrangian given by the argument of the holomorphic volume form $\Omega$ on the symplectic manifold. Then, a grading on $L$ is given by the choice of the lift $L \to \bR$ of the phase map. Secondly, as detailed in Section 11j of \cite{Sei1}, a choice of Pin structure for each Lagrangian gives rise to an orientation of the moduli space $\mathcal{M}(x_0, x_1; J_{L_0, L_1})$ which defines Floer complexes over a field over characteristic not equal to 2. We often abuse the notation and refer to a Lagrangian brane as the underlying Lagrangian submanifold.
    
    Also, there are higher structure maps $\mu^d$ on the chain complexes of Lagrangians. For $d \ge 2$, let $\mathcal{R}_{d+1}$ be the space of disks with $d+1$ boundary marked points labeled in $\zeta_0, \cdots, \zeta_d$ in counterclockwise order, and let $\mathcal{S}_{d+1} \to \mathcal{R}_{d+1}$ be the universal family. The Deligne-Mumford-Stasheff compactification $\overline{\mathcal{R}}_{d+1}$ of $\mathcal{R}_{d+1}$ and the corresponding universal family $\overline{\mathcal{S}}_{d+1} \to \overline{\mathcal{R}}_{d+1}$ is obtained by adding the tree-configuration of these disks. Then, we choose a family of strip-like ends for the disks in $\mathcal{R}_{d+1}$, given by the map
    \begin{align*}
        \epsilon_{j}^{d+1}: \bR_{>0} \times [0,1] \times \overline{\mathcal{R}}_{d+1} \to \overline{\mathcal{S}}_{d+1} &\qquad \text{for } j=1,\cdots,d,\\
        \epsilon_{0}^{d+1}: \bR_{<0} \times [0,1] \times \overline{\mathcal{R}}_{d+1} \to \overline{\mathcal{S}}_{d+1} &.
    \end{align*}
    Such maps can be chosen to be compatible with the gluing construction, i.e. with the map
    $$\overline{\mathcal{R}}_{d+1} \times \overline{\mathcal{R}}_{l+1} \times [R,\infty) \to \overline{\mathcal{R}}_{d+l}$$
    given by gluing some $\zeta_j$ in a disk in $\overline{R}_{l+1}$ to $\zeta_0$ in a disk in $\overline{R}_{d+1}$.

    Now, let $L_0, \cdots, L_d$ be mutually transverse Lagrangians of the form $\widetilde{L}_{f, R} (-kH_{f, R})$ not bounding any holomorphic disks. For each previous choice of almost complex structures $J_{L_j, L_{j+1}}$ for $j=0,\cdots,d-1$ and $J_{L_0, L_d}$, we can choose a family of almost complex structures
    $$J_{L_{j_0}, \cdots, L_{j_l}}: \overline{\mathcal{S}}_{l+1} \to \mathcal{J}$$
    for every $0 \le j_0 < j_1 <\cdots<j_l \le d$ to be compatible with the gluing map above. Then, for an input generators $x_j \in L_j \cap L_{j+1}$ with $j=0,\cdots,d-1$ and an output generator $x_0 \in L_0 \cap L_d$, define $\overline{\mathcal{M}}(x_1,\cdots,x_d,x_0;J_{L_0,\cdots,L_d})$ to be the set of pairs $(S,u)$ where $S \in \overline{\mathcal{R}}_{d+1}$ and a $J_{L_0,\cdots,L_d}$-holomorphic map $u: S \to (\bC^*)^2$ that sends each boundary component between $\zeta_j$ and $\zeta_{j+1}$ into $L_j$ in the cyclic order and limit to $x_j$ in the (positive or negative) strip-like ends. These moduli spaces are also compact smooth manifolds with boundary for a generic choice of almost complex structures so that we can define the map
    $$\mu^d: CF^\bullet(L_{d-1}, L_d) \otimes \cdots \otimes CF^\bullet(L_0, L_1) \to CF^\bullet(L_0, L_d)[2-d]$$
    of degree $2-d$ by
    $$\mu^d (x_d, \cdots, x_1) = \sum_{[u];\,x_0 \in L_0 \cap L_d} \#\mathcal{M} (x_1, \cdots,x_d , x_0;[u], J_{L_0, L_1}) T^{E([u])} \text{hol}([\round u]) x_0,$$
    where $\mathcal{M} (x_1, \cdots,x_d , x_0;[u], J_{L_0, L_1})$ consists the elements of $\mathcal{M} (x_1, \cdots,x_d , x_0;J_{L_0, L_1})$ with fixed homology class $[u]$. The maps $\mu^d$ form an $A_\infty$-structure.

    \begin{remark}\label{rem:use-standard-complex-structure}
        The standard continuation map argument shows that the Floer cohomology groups are independent of the choice of almost complex structures. Also, for simplicity, we will assume that the standard complex structure is regular and count holomorphic curves for the remaining sections. In reality, we need to use a sequence of almost complex structures converging to the standard complex structure. This can potentially lead to the existence of $J$-holomorphic disks with an area bigger than the holomorphic curves we consider. It therefore requires more careful analysis involving the notion of filtered $A_\infty$-algebras, where we first work on the Floer cohomology groups modulo $T^{E_n}$ for an increasing sequence $E_n$ and then take a limit $E_n \to \infty$.
    \end{remark} 
    %We will assume that the almost complex structures are sufficiently close to the standard complex structure since it is necessary to give an estimate in Theorem~\ref{thm:disk-bound-arg-strip}.
    %%%%[May move to later section]]]]]]]

    Even though we are required to work in the Novikov field, it is sometimes convenient to rescale the generator and cancel out the energy term $T^{E([u])}$. For example, if the Lagrangians $L_0, L_1$ are exact, then there exists a function $f_j$ on each Lagrangian $L_j$ such that $\lambda |_{L_j } = df_j$ where $\lambda = p_1 d\theta_1 + p_2 d\theta_2$ is the primitive of the symplectic form. Then, rescaling each generator $x \in CF^\bullet(L_0, L_1)$ to $T^{f_1(x) - f_0 (x)} x$ eliminates the term $T^{E([u])}$ on the differential, hence allows to work over the coefficient field $\bC$. The Lagrangians $\widetilde{L}_{f, R} (-kH_{f, R})$ are not exact in general, but we rescale the generators on each cylindrical end to simplify the calculation.

    \begin{notation}\label{not:action-rescale}
        For a generator $x \in CF^\bullet(L_0, L_1)$ on the cylindrical end $Z_{\alpha,r}^+$, we rescale the generator by
        $$ x \,\,\mapsto\,\, T^{g_1 (z_\alpha(x)) - g_2 (z_\alpha(x)) + 2\pi j(x) p_\alpha(x)} x$$
        where $z_\alpha(x) = (p_\alpha(x), \theta_\alpha^\perp(x)) \in Z_{\alpha,r}^+$ is the $(p_\alpha, \theta_\alpha^\perp)$-part of $x$, $j_\alpha(x) \in \bZ$ is the $\theta_\alpha$-distance between $L_0$ and $L_1$ on the universal cover, and $g_0, g_1$ are functions on $Z_{\alpha,r}^+ \cap \{p_\alpha \ge R\}$ such that each $L_i$ is represented as a graph $d g_i$, all of them following the notation in Theorem~\ref{thm:action-ftnal}.
    \end{notation}

    By Theorem~\ref{thm:action-ftnal}, the exponent on the differential after rescaling is equal to $\int_{[\round u]_\alpha} \lambda_\alpha$. In particular, if $u$ is contained in the cylindrical end $Z_{\alpha,r}^+$, then $[\round u]_\alpha = 0$ and the exponent is eliminated.

    \begin{remark}
        We can also use a `global' version rather than separately rescaling the generators on each cylindrical end. To do that, we first pick a finite number of curves $\gamma_1, \cdots, \gamma_t$ on each Lagrangian $L$ so that $L \setminus \bigcup_{j} \gamma_j$ is exact. Then, we can rescale the generators by finding a primitive $f_j$ of $\lambda$ on each Lagrangian. We have additional terms on the exponent, which come from the intersection number of $\round u$ with each $\gamma_j$. However, since the curves $\gamma_j$ pass through one or more cylindrical ends, the exponents might not be eliminated even when $u$ is contained in the cylindrical region.
    \end{remark}

\subsection{Homotopy Method}
    In this section, we elaborate on how to deal with the `stretching parameter' $R>0$. The main ingredient to deal with the parameter is to find a chain map called \textit{continuation map} between the Floer cochain complexes with different $R>0$. To take advantage of the analysis on holomorphic disks done in Theorem~\ref{thm:disk-bound-arg-strip}, we count a series of disks bounded by a fixed Lagrangian rather than counting disks with moving boundary conditions. This construction, called the homotopy method, is elaborated in Section 10e in \cite{Sei1}.
    
    When defining the wrapped Floer theory between the special Lagrangians $L_f$, we would want to take a limit $R \to \infty$ to separate the cylindrical region from the Lagrangian. The description of the Lagrangian $\widetilde{L}_{f, R}$ in Definition~\ref{def:Lf-tilde} and Remark~\ref{rem:cyl-end-desc-for-perturbation} is given by the graph of a differential $dg$ on the cotangent bundle of the cylindrical end $Z_{\alpha,r}^+$, so the fiberwise linear isotopy gives an isotopy between $\widetilde{L}_{f, R}$ with different $R>0$. For each pair of distinct Lagrangians of the form $L_{f,R}^k := \widetilde{L}_{f, R}(-kH_{f, R})$ the transversality is preserved along the isotopy for large enough $R>0$, which implies that the intersection points between them vary smoothly. Following the notation for the generators, we can view each cylindrical generator $x_\alpha^j x_\alpha^e$ and $x_\alpha^j x_\alpha^f$ as a path of intersection points between $\widetilde{L}_{f, R}(-kH_{f, R})$ with varying $R$. Since the interior region stays constant, the interior intersection points correspond to constant paths.

    For a fixed pair of Laurent polynomials $(f_0, f_1)$ and a pair of nonnegative integers $(k_0, k_1)$, set $L_{0,R} := L_{f_0,R}^{k_0}$ and $L_{1,R} := L_{f_1,R}^{k_1}$ for every sufficiently large $R>0$. Also, fix a family $R_s$ of positive numbers over $s \in [0,1]$ which strictly and continuously increases from $R_0$ to $R_1$, and write $CF^\bullet (L_0, L_1)^{[0,1]}$ for the vector space generated by paths $x:[0,1] \to (\bC^*)^2$ where $x(s) \in L_{0,R_s} \cap L_{1,R_s}$. This is isomorphic to any $CF^\bullet (L_{0,R_s}, L_{1,R_s})$ as a vector space.

    We start from the cochain map construction representing the first order map of the $A_\infty$ functor. Given the generators $x_0, x_1$ of $CF^\bullet (L_0, L_1)^{[0,1]}$, recall that the Floer theory between $L_{0, R_s}$ and $L_{1, R_s}$ is obtained by counting index zero elements in the moduli space of strips $\mathcal{M}(x_1, x_0)^s := \mathcal{M}(x_1(s), x_0(s))$. Now, for each $\nu \ge 1$, define $\mathcal{H}^\nu (x_1, x_0)^{[0,1]}$ as a moduli space of $\nu$-tuples $((s_1, u_1), \cdots, (s_\nu, u_\nu ))$ such that $0 \le s_1 \le \cdots \le s_\nu \le 1$ and $u_i \in \mathcal{M}(\tilde{x}_{i}, \tilde{x}_{i-1})^{s_i}$ for a sequence of generators $\tilde{x}_0 = x_0, \tilde{x}_1, \cdots, \tilde{x}_\nu = x_1$. The elements of $\mathcal{H}^\nu (x_1, x_0)^{[0,1]}$ can be visualized as a chain of $\nu$ disks with 2 boundary points as Figure~\ref{fig:htpy-method}. For completeness, we formally define $\mathcal{H}^0 (x_1, x_0)^{[0,1]}$ to be a point if $x_0 = x_1$ and empty otherwise.
    The boundary of the moduli space $\mathcal{H}^\nu (x_1, x_0)^{[0,1]}$ is a union of products
    $$\prod_{i=1}^\nu \mathcal{M}(\tilde{x}_{i}, \tilde{x}_{i-1})^{s_\nu}$$
    cut out by one or more equations of the form $s_{i} = s_{i+1}$.
    
    \vspace{0.3em}
    \begin{figure}[ht]
        \centering
        \includegraphics[width=0.7\textwidth]{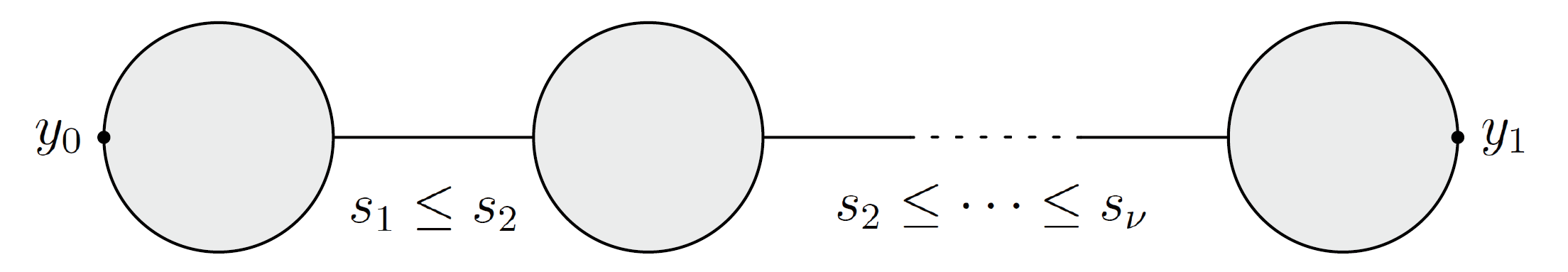}
        \caption{An element in the moduli space $\mathcal{H}^\nu (x_1, x_0)^{[0,1]}$}
        \label{fig:htpy-method}
    \end{figure}
    \vspace{0.3em}
    
    %To show the regularity of $\mathcal{H}^\nu (x_1, x_0)^{[0,1]}$, one needs to show that these products are regular, and it can be achieved when the pair $(s_i, u_i)$ are pairwise distinct. 
    If the Lagrangians are exact, we can find an action functional on the generators and show that each $u_i$ strictly decreases the action functional. The Lagrangians $L_{f, R}^k$ are generally not exact, but we can get a similar result from the analysis on the strips done in Theorem~\ref{thm:disk-bound-arg-strip}. 

    \begin{proposition}\label{prop:htpy-mtd-disk-counting}
        For a choice of $\{\varphi_\alpha\}$, $J$, and sufficiently large $R>0$ satisfying Theorem~\ref{thm:disk-bound-arg-strip}, the moduli space $\mathcal{M}(x_1, x_0)^s$ cannot contain non-constant maps except for the following cases.
        \begin{itemize}
            \item Both generators are in the same cylindrical end $Z_{\alpha,r}^+$, where $x_1$ is of the form $x_\alpha^{j_0} x_\alpha^e$ and $x_0$ if of the form $x_\alpha^{j_1} x_\alpha^f$.
            \item $x_1$ is of the form $x_\alpha^{j} x_\alpha^e$ or $x_\alpha^{j} x_\alpha^f$, and $x_0$ is a generator on the interior region.
            \item Both generators are in the interior region.
        \end{itemize}
    \end{proposition}

    \begin{proof}
        This is a direct corollary of Theorem~\ref{thm:disk-bound-arg-strip}, together with the fact that there is no strip from the generator on $Z_{\alpha,r}^+$ to the generator on $Z_{\beta,r'}^+$ with $\alpha \neq \beta$. This is because the maximum of $p_\beta$ on the strip should be achieved on one of the generators, and the output generator is not a local maximum since the strip covers at least one of the obtuse-angled regions between the rays in the $(p_\beta, \theta_\beta)$-plane.
    \end{proof}

    %The proposition reduces the problem to the interior region where $L_{0, R}$ and $L_{1, R}$ are both Hyperkähler rotations of a complex curve.

    Moreover, Gromov compactness argument can be applied to $\mathcal{H}^\nu (x_0, x_1)^{[0,1]}$, since the strips are contained in a compact subset. We also assume that these moduli spaces are regular when using the standard complex structure; see Remark~\ref{rem:use-standard-complex-structure}. This gives rise to a map $\Gamma: CF^\bullet(L_{0,R_0}, L_{1,R_0}) \to CF^\bullet(L_{0,R_1}, L_{1,R_1})$ defined by
    \begin{equation}\label{eqn:htpy-mtd-gamma}
        \Gamma(x_0(0)) = \sum_{[u_1],\cdots,[u_\nu];\,\nu \ge 0} \# \mathcal{H}^\nu (x_1 ,x_0;[u_1], \cdots,[u_\nu])^{[0,1]} \prod_{i=1}^\nu \left(T^{E([u_i])} \text{hol}([\round u_i]) \right) x_1(1),
    \end{equation}
    where $\mathcal{H}^\nu (x_1, x_0;[u_1], \cdots,[u_\nu])^{[0,1]}$ is a subset of $\mathcal{H}^\nu (x_1, x_0)^{[0,1]}$ with a fixed homology class of each map $u_i$. There are also terms representing the action difference while moving generators along the path between $u_{i-1}$ and $u_i$, but they are eliminated after the rescaling in Notation~\ref{not:action-rescale}. An analysis of the boundary of 1-dimensional components shows that $\Gamma$ is a chain map.

    For the higher order maps $F_{R_0, R_1}^d$, we consider an analogous construction similar to $\mu^{d}$. Namely, we count the tree-configuration of disks bounded by three different Lagrangians, but rather than fixing a time $s$, we count the disk with a moving boundary condition to achieve the regularity of the moduli space. The detailed construction can be found in Section 10e of \cite{Sei1}. The map
    $$F_{R_0, R_1}^d: CF^\bullet(L_{d-1, R_0}, L_{d,R_0}) \otimes \cdots \otimes CF^\bullet(L_{0, R_0}, L_{1, R_0}) \to CF^\bullet(L_{0, R_1}, L_{d, R_1})[1-d],$$
    together with $F_{R_0, R_1}^1 = \Gamma$, form a collection of maps satisfying the conditions to be a $A_\infty$-functor; the $A_\infty$ category will be defined in the next section.

    We conclude this section with the corollary of Proposition~\ref{prop:htpy-mtd-disk-counting} that the map $F_{R_0, R_1}^1$ is an isomorphism between vector spaces for sufficiently large $R_0$. Since $F_{R_0, R_1}^1$ is a chain map, it shows that the differential structure on $ CF^\bullet (L_{0, R}, L_{1, R})$ is identical for sufficiently large $R>0$. 

    \begin{proposition}\label{prop:htpy-mtd-q-iso}
        For a choice of $\{\varphi_\alpha\}$, $J$, and sufficiently large $R>0$ satisfying Theorem~\ref{thm:disk-bound-arg-strip}, the map $\Gamma$ defined by (\ref{eqn:htpy-mtd-gamma}) is a (vector space) isomorphism for $R_1 > R_0 > R$.
    \end{proposition}

    \begin{proof}
        For degree reasons, we only need to count isolated points in $\mathcal{H}^\nu (x_1, x_0)^[0,1]$ when $x_1$ and $x_0$ are both of the form $x_\alpha^j x_\alpha^e$, or when they are either interior generators or of the form $x_\alpha^j x_\alpha^f$. The only non-constant disks that can contribute to the calculation are strips between interior generators (when $x_1$ and $x_0$ are both interior generators). However, since the Lagrangian is stationary in the interior region, we can freely move $s_i$ around $[0,1]$ so that the configuration is not isolated.
    \end{proof}

\subsection{Homotopy Colimits of \texorpdfstring{$A_\infty$}{A∞}-Categories}
    Throughout this section, we fix a finite collection of Laurent polynomials $f$, together with a choice of a brane structure. Also, choose a generic value of $\{\varphi_\alpha\} \subseteq \bR$ which makes the Theorem~\ref{thm:disk-bound-arg-strip} holds. 
    For $R>0$ and a positive integer $K$, we construct a category $\mathcal{O}_{R,\le n}$, whose objects are Lagrangians of the form $L_{f,R}^k :=\widetilde{L}_{f, R}(-kH_{f, R})$ for $k=0,1,\cdots,n$. By the transversality result in Theorem~\ref{thm:transversality-main}, we can define a Floer cohomology ring for each pair of objects in $\mathcal{O}_{R,\le n}$ for sufficiently large $R>0$. The morphism space is defined as follows.
    \begin{equation*}
        \mathcal{O}_{R, \le n} (L_{{f_0},R}^{k_0}, L_{{f_1},R}^{k_1}) 
    = \begin{cases}
        CF^\bullet (L_{{f_0},R}^{k_0}, L_{{f_1},R}^{k_1}) & \text{if } k_0 < k_1 \le n \\
        \mathbb{K} \cdot \text{id} & \text{if } k_0 = k_1, \text{ and } f_0 = f_1 \\
        0 & \text{otherwise}.
    \end{cases}
    \end{equation*}
    There are $A_\infty$-operators $\mu^d$ as defined in Section~\ref{subsect:FloerThoery-basic}. The identity morphisms $\text{id}$ are added formally to make this category strictly unital. One can check that $\mathcal{O}_{R,\le n}$ form an $A_\infty$-category, and for $d \ge 3$, the map
    %m each identity morphism is a multiplicative identity, and the value of $\mu^d$ is zero if there is at least one identity map in the input.
    $$\mu^d: CF^\bullet(L_{{f_{d-1}}, R}^{k_{d-1}}, L_{f_{d}, R}^{k_{d}}) \otimes \cdots \otimes CF^\bullet(L_{f_{1}, R}^{k_{1}}, L_{f_{0}, R}^{k_{0}}) \to CF^\bullet(L_{f_{0}, R}^{k_{0}}, L_{f_{d}, R}^{k_{d}})[2-d]$$
    can be nonzero only when $k_0 < k_1 < \cdots < k_d$.

    The homotopy method in the previous section gives rise to an $A_\infty$-functor $F_{R, R'}:\mathcal{O}_{R, \le n} \to \mathcal{O}_{R', \le n}$, and it is a quasi-isomorphism for sufficiently large $R>0$.

    \begin{theorem}
        For a fixed positive integer $n$, there exists $R_0>0$ such that the functor $F_{R, R'}:\mathcal{O}_{R, \le n} \to \mathcal{O}_{R', \le n}$ is a quasi-isomorphism for every $R' > R > R_0$.
    \end{theorem}

    \begin{proof}
        The only thing we need to show is that the map
        $$F_{R, R'}^1(L_{f_1}^{k_1}, L_{f_2}^{k_2}): CF^\bullet (L_{{f_1}, R}^{k_1}, L_{{f_2}, R}^{k_2}) \to CF^\bullet (L_{{f_1}, R'}^{k_1}, L_{{f_2}, R'}^{k_2})$$
        is a chain isomorphism for every $R' > R > R_0$ and pair of objects in the category $\mathcal{O}_{R, \le n}$, which is obvious from Proposition~\ref{prop:htpy-mtd-q-iso} since we have a finite number of Lagrangians in $\mathcal{O}_{R, \le n}$.
    \end{proof}
    From the theorem, we can pick an increasing sequence of numbers $R_n$ for each positive integer $n$ so that $\mathcal{O}_{R_m, \le n}$ is well defined for every $m \ge n$, and the functor $F_{R_{m_1}, R_{m_2}}: \mathcal{O}_{R_{m_1}, \le n} \to \mathcal{O}_{R_{m_2}, \le n}$ is a quasi-isomorphism for every $m_2 \ge m_1 \ge n$.

    On the $n$-direction, the category $\mathcal{O}_{R, \le n+1}$ contains $\mathcal{O}_{R, \le n}$ as a full subcategory, in addition to the objects $L_{f, R}^{n+1}$ that are left-orthogonal to $\mathcal{O}_{R, \le n}$ in the sense that any morphism space from $L_{f, R}^{n+1}$ to an object in $\mathcal{O}_{R, \le n}$ is zero. Combining the embedding $\mathcal{O}_{R, \le n} \hookrightarrow \mathcal{O}_{R, \le n+1}$ with the functor $F_{R,R'}$, we get a diagram of $A_\infty$-categories $\{\mathcal{O}_{R_{n},\le n}\}_{n \in \bZ_{\ge 0}}$ over the nonnegative integers. 
    
    The direct limit of this diagram of $A_\infty$-categories is called the \textit{homotopy colimit} and indicates the limiting behavior when stretching the cylindrical region to infinity. To construct the homotopy colimit $\hclim{n \to \infty} \mathcal{C}_n$ of a diagram $\{\mathcal{C}_n\}_{n \in \bZ_{\ge 0}}$, we first produce a category $\mathcal{G} = \text{Groth}_{n \ge 0} \,\mathcal{C}_n$ from the \textit{Grothendieck construction}, so that the objects are $\bigsqcup_{n\ge0} \mathcal{C}_n$ and the morphism space between an object $L_m$ in $\mathcal{C}_m$ and an object $L_n$ is
    \begin{equation*}
        \mathcal{G} (L_m, L_n) = \begin{cases}
            0 & \text{if } m > n, \\
            \mathcal{C}_n(F_{m,n}(L_m), L_n) & \text{if } m \le n.
        \end{cases}
    \end{equation*}
    Then, the homotopy colimit is defined as the localization of the Grothendieck construction
    $$\hclim{n \to \infty} \mathcal{C}_n := \left(\Groth{n \ge 0} \mathcal{C}_n \right)[A^{-1}]$$
    at the collection $A$ of morphisms $L_m \to F_{m,n} L_m$ corresponding to the identity map in $H^0 \mathcal{C}_n (F_{m,n} L_m, F_{m,n} L_m )$ for $m \le n$ and $L_m \in \mathcal{C}_m$. The localization is done by taking quotients over the cone of elements in $A$, in the sense of \cite{LO1}. The detailed construction can be found in Appendix A of \cite{GPS2}.

    \begin{definition}
        The category $\mathcal{O}$ is the $A_\infty$-category defined by
        $$\mathcal{O} = \hclim{n \to \infty} \mathcal{O}_{R_{n},\le n},$$
        where the functor between $\mathcal{O}_{R_{n},\le n}$ and $\mathcal{O}_{R_{n+1},\le n+1}$ is the composition
        $$\mathcal{O}_{R_{n},\le n} \xrightarrow{F_{R_n, R_{n+1}}} \mathcal{O}_{R_{n+1},\le n} \xhookrightarrow{\quad} \mathcal{O}_{R_{n+1},\le n+1}$$
        of a quasi-isomorphism and an embedding.
    \end{definition}

    In our case, each functor in the diagram is a fully faithful functor, and all the added objects are left-orthogonal to the previously added ones. In this case, the homotopy colimit looks like a `union' of these categories.

    \begin{proposition}[\cite{GPS2}, Proposition A.14]
        Suppose that a diagram $\{\mathcal{C}_n\}_{n \in \bZ_{>0}}$ satisfies the following properties.
        \begin{enumerate}
            \item All functors $\mathcal{C}_m \hookrightarrow \mathcal{C}_n$ for $m \le n$ are fully faithful.
            \item Every $\mathcal{C}_n$ is generated by the images of $\mathcal{C}_m$ for $m<n$ and objects which are left-orthogonal to all these images (call such object \textit{right-new}).
        \end{enumerate}
        Then, the limit $\hclim{n \to \infty}\mathcal{C}_n $ is generated by right-new objects, and the map
        $$\Groth{n \ge 0}\mathcal{C}_n (L_m, -) \to \hclim{n \to \infty} \mathcal{C}_n (L_m, -)$$
        is a quasi-isomorphism between $\mathcal{C}_m$-modules for any right-new object $L_m \in \mathcal{C}_m$.
    \end{proposition}
    $\hfill \square$

    %The proposition tells that $\mathcal{O}_\infty$ is generated by the Lagrangians of the form $L_{f,R_n}^n$ for $n \ge 0$, and the morphism spaces between these Lagrangians are quasi-isomorphic to the Floer cochain complex between them.

    \begin{corollary}\label{cor:category-O-generators}
        The $A_\infty$-category $\mathcal{O}$ satisfies the following.
        \begin{enumerate}
            \item $\mathcal{O}$ is generated by the images $L_{f}^n$ of the Lagrangian $L_{f, R_n}^n$ for $n \ge 0$.
            \item For $m \le n$ and $R \ge R_n$, there exists a diagram of quasi-isomorphisms
            $$CF^\bullet (L_{f, R}^m, L_{f,R}^n ) \xleftarrow{\sim} CF^\bullet (L_{f, R_n}^m, L_{f,R_n}^n ) \xrightarrow{\sim} \mathcal{O}(L_{f}^m, L_{f}^n ).$$
        \end{enumerate}
    \end{corollary}

    \begin{proof}
        The morphism space from $L_{f_1, R_m}^m$ to $L_{f_2, R_n}^n$ in the Grothendieck construction is the Floer cochain complex between $F_{R_m, R_n} (L_{f_1, R_m}^m) = L_{f_1, R_n}^m$ and $L_{f_2, R_n}^n$.
    \end{proof}

%    To describe this category, we first consider the category
%    $$\mathcal{O}_{\infty, \le n} := \hclim{m \to \infty}\mathcal{O}_{R_{m},\le n}$$
%    for a fixed integer $n$. Since every functor in the diagram $\{\mathcal{O}_{R_{m},\le n}\}_{m \ge n}$ is a quasi-equivalence, its limit is also quasi-equivalent to each of these categories.
%    \begin{lemma}
%        For every $m \ge n$, the natural functor $F_{R_m, \infty}:\mathcal{O}_{R_{m},\le n} \to \mathcal{O}_{\infty, \le n}$ is a quasi-equivalence.
%    \end{lemma}
%    \begin{proof}
%        a
        
%        \begin{center}
%            \begin{tikzcd}
%                \mathcal{O}_{R_n, \le n} \arrow{r}{\text{id}} \arrow{d}{\text{id}} & \mathcal{O}_{R_n, \le n} \arrow{r}{\text{id}} \arrow{d} & \mathcal{O}_{R_n, \le n} \arrow{r}{\text{id}} \arrow{d} & \cdots \arrow{r} & \mathcal{O}_{R_n, \le n} \arrow{d} \\
%                \mathcal{O}_{R_n, \le n} \arrow{r}{F} & \mathcal{O}_{R_{n+1}, \le n} \arrow{r}{F} & \mathcal{O}_{R_{n+2}, \le n} \arrow{r}{F} & \cdots \arrow{r} & \mathcal{O}_{\infty, \le n} \\
                %A \arrow{r}{\varphi} \arrow[swap]{d}{\varrho_f} & B \arrow{d}{\varrho_g} \\%
                %A_f \arrow{r}{\varphi_f}& B_g
%            \end{tikzcd}
%        \end{center}
%    \end{proof}

\subsection{Partially Wrapped Fukaya Category}
    The final step is to use the localization process introduced in \cite{ASx} on the category $\mathcal{O}$ to define the desired wrapped Fukaya category as a partially wrapped Fukaya category. We continue to fix a finite collection $B$ of Laurent polynomials $f$.

    We begin from an observation that there exists a distinguished collection of morphisms
    $$e_{f,R}^k \in CF(L_{f,R}^k, L_{f,R}^{k+1}) \cong \mathcal{O}(L_f^k , L_f^{k+1})$$
    for every $f$, $k$, and sufficiently large $R>0$, called \textit{quasi-units}. This element is defined by counting solutions to a Cauchy-Riemann equation with a moving boundary condition. Namely, the domain $S$ is a disk with a single strip-like end, and the boundary condition is given by the isotopy $L_{f, R}^{\kappa(t)} = \widetilde{L}_{f, R}(-\kappa(t)H_{f, R})$ where $\kappa:\round S \to [k,k+1]$ is a smooth map monotonically increasing from $k$ to $k+1$ in the counterclockwise direction which is constant on the strip-like end. The multiplication map
    $$\mu^2 (\cdot, e_{{f_0},R}^{k_0}) : CF^\bullet(L_{f_0,R}^{k_0 +1}, L_{f_1,R}^{k_1}) \to CF^\bullet(L_{f_0,R}^{k_0}, L_{f_1,R}^{k_1})$$
    and
    $$\mu^2 (e_{f_1,R}^{k_1}, \cdot) : CF^\bullet(L_{f_0,R}^{k_0}, L_{f_1,R}^{k_1}) \to CF^\bullet(L_{f_0,R}^{k_0}, L_{f_1,R}^{k_1+1})$$
    can be viewed as a continuation map, since gluing a disk with one boundary marked point with one of the inputs of another disk gives a holomorphic strip with moving boundary condition. In fact, the quasi-unit $e{{f}, R}^{k}$ is the image of the identity morphism in $HF^\bullet(L_{f, R})$ under these continuation maps.

    \begin{proposition}\label{prop:quasi-unit-htpy-mtd}
        The quasi units are sent to quasi-units on the cohomology level through the functor $F_{R,R'}^1:CF^\bullet(L_{f_0, R}^{k_0}, L_{f_1, R}^{k_1}) \to CF^\bullet(L_{f_0, R'}^{k_0}, L_{f_1, R'}^{k_1})$ defined by the homotopy method in (\ref{eqn:htpy-mtd-gamma}).
    \end{proposition}
    \begin{proof}
        %One can consider a cobordism between two corresponding configurations, as seen in Figure~\ref{fig:quasi-unit-htpy-method}. 
        This can be proved by considering the moduli space of configurations that start with the quasi-unit from $R_t$ and then continue by homotopy method from it to $R'$, as $R_t$ varies from $R$ to $R'$. 
        This moduli space has two main boundaries, one for $R_t = R$ corresponding to $F_{R,R'}^1(e_{f_0, R}^{k_0})$ and the other for $R_t = R'$ corresponding to $e_{f_1, R}^{k_1}$.
        The only other boundaries of this 1-dimensional moduli space consist of configurations where a Floer strip breaks off at $R'$ at the end of the homotopy method, which does not affect the equality at cohomology level.
    \end{proof}
    
    \vspace{0.3em}
    \begin{figure}[ht]
        \centering
        \includegraphics[width=0.8\textwidth]{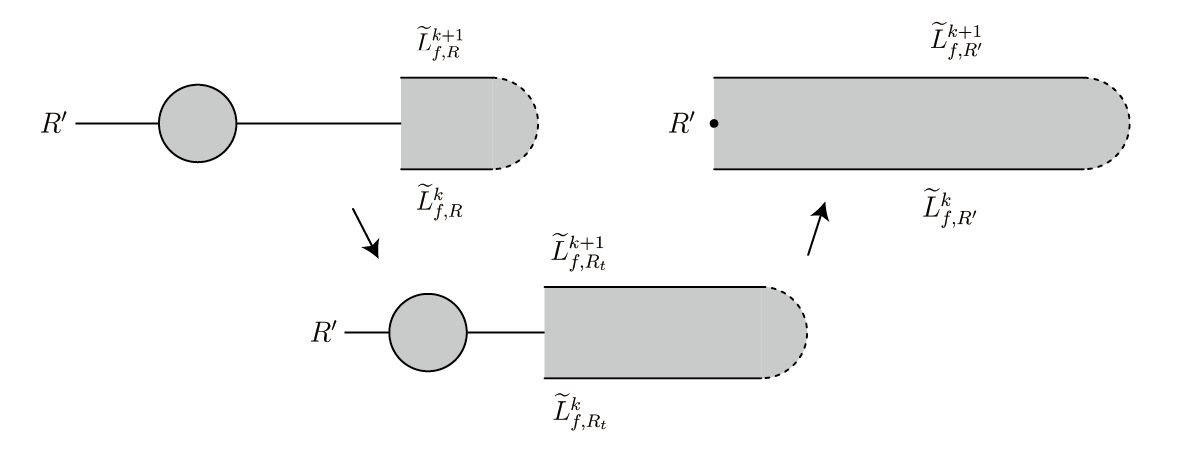}
        \caption{Moduli space in the proof of Proposition~\ref{prop:quasi-unit-htpy-mtd}. The dashed lines denote the moving boundary.}
        \label{fig:quasi-unit-htpy-method}
    \end{figure}
    \vspace{0.3em}

    Now, for a fixed finite collection $B$ of Laurent polynomials $f$, we set $Z$ as the collection of the images of $e_{f,R_k}^k$ inside $\mathcal{O}$ along the quasi-isomorphisms in Corollary~\ref{cor:category-O-generators}, and define the wrapped Fukaya category by the localization
    $$\mathcal{F}_{B} := Z^{-1}\mathcal{O}.$$

    This category has a more explicit description given by a homotopy colimit of $\mathcal{O}$-modules, as detailed in Section 3 of \cite{AA1}. The only additional geometrical ingredient needed is the continuation map
    $$G=G_{f_0, f_1, R}^{k_0, k_1}: CF^\bullet(L_{f_0,R}^{k_0 }, L_{f_1,R}^{k_1}) \to CF^\bullet(L_{f_0,R}^{k_0 +1}, L_{f_1,R}^{k_1+1}) $$
    obtained by counting strips that solve a Cauchy-Riemann equation with a moving boundary condition. The gluing argument shows that the map $G$ fits into the following diagram that commutes on the cohomology level.

    \begin{center}
        \begin{tikzcd}[row sep=huge,column sep=huge]
            CF^\bullet(L_{f_0,R}^{k_0 +1}, L_{f_1,R}^{k_1}) \arrow{r}{\mu^2 (e_{f_1,R}^{k_1}, \cdot)}\arrow{d}{\mu^2 (\cdot, e_{{f_0},R}^{k_0})} & CF^\bullet(L_{f_0,R}^{k_0 +1}, L_{f_1,R}^{k_1+1})\arrow{d}{\mu^2 (\cdot, e_{{f_0},R}^{k_0})} \\
            CF^\bullet(L_{f_0,R}^{k_0 }, L_{f_1,R}^{k_1}) \arrow[swap]{r}{\mu^2 (e_{f_1,R}^{k_1}, \cdot)}\arrow[swap]{ur}{G_{f_0, f_1, R}^{k_0, k_1}} & CF^\bullet(L_{f_0,R}^{k_0 }, L_{f_1,R}^{k_1+1})
        \end{tikzcd}
    \end{center}

    Taking a homotopy colimit of these maps gives an analogous quasi-units and continuation maps in the $\mathcal{O}$The diagram above commutes after taking the limit since $H^\bullet\mathcal{O}(L_{f_0}^{k_0}, L_{f_1}^{k_1})$ is isomorphic to $HF^\bullet (L_{f_0, R}^{k_0}, L_{f_1, R}^{k_1})$ for suffienctly large $R>0$ by Corollary~\ref{cor:category-O-generators}.

    \begin{proposition}\label{prop:HW-is-limit-HfR}
        For each pair of objects $L_{f_0}^{k_0}$, $L_{f_1}$, there is a natural isomorphism
        \begin{equation}
            \colim{k_1 \to \infty} HF^\bullet (L_{f_0, R}^{k_0}, L_{f_1, R}^{k_1}) \to \colim{k_1 \to \infty} H^\bullet \mathcal{O} (L_{f_0}^{k_0}, L_{f_1}^{k_1}) \to H^\bullet \mathcal{F}_{B}(L_{f_0}, L_{f_1})
        \end{equation}
        for sufficiently large $R>0$.
    \end{proposition}
    \begin{proof}
        The map on the left comes from Corollary~\ref{cor:category-O-generators}. For the map on the right, see Corollary 3.24 in \cite{AA1}.
    \end{proof}

\section{Distinguished Basis of the Floer Cohomology}\label{sect:main-example}
In this section, we will consider the family of complex curves
$$D_q = \{1 -z_1 -z_2 + q^{-1}z_1 z_2\}$$
for $q \in \bC^*$ and their HyperKähler rotation $L_q$ as a model case for computing the Floer theory defined in the previous section. In particular, the main goal is to compute and find a distinguished basis of the Floer cohomology ring $HW^0 (L_q, L_q)$ defined as the following.

\begin{definition}\label{def:wrapped-Floer-coh}
    For a pair of Laurent polynomials $(f_0, f_1)$ satisfying the Assumption~\ref{ass:top-bound-for-boundary}, the \textit{wrapped Floer cohomology} between $L_{f_0}$ and $L_{f_1}$ is defined by
    \begin{equation}
        HW^\bullet(L_{f_0}, L_{f_1}) := H^\bullet \mathcal{F}_{B} (L_{f_0}, L_{f_1}).
    \end{equation}
    for any collection $B$ of Lagrangians containing both $L_0$ and $L_1$.
\end{definition}

By Proposition~\ref{prop:HW-is-limit-HfR}, the computation can be reduced to the one for Floer cohomology ring $HF^\bullet (L_{f_0, R}^{k_0}, L_{f_1, R}^{k_1})$ for large enough $R > 0$ and difference $k_1 - k_0$.

In \cite{Hic1}, the wrapped Floer cohomology for tropical Lagrangians is calculated for a wide variety of $f$, with different constructions for both Lagrangians (which are still Hamiltonian isotopic to ours) and the cohomology ring. The main part of this section would be to relate our construction to those and give a more concrete description of generators by finding a distinguished basis (Section~\ref{subsect:distinguished-basis}) and their corresponding functions on the mirror (Section~\ref{subsect:basis-functions}).

One notable thing in the family $\{L_q\}$ is that these Lagrangians are not exact for $q \in \bC^* \setminus S^1$. Namely, the image of the boundary map $H^2((\bC^*)^2, L_q) \to H^1(L_q)$ on the long exact sequence of the pair $((\bC^*)^2, L_q)$ is generated by a single element, which is the difference of two $S^1$ on two opposite (parallel) cylindrical ends, and the integral of $\lambda = p_1 d\theta_1 + p_2 d \theta_2$ along this curve is $2\pi \log|q|$. This can create a lot of technical difficulties, including the disk bubbling and the absence of the action functional, but most of them can be dealt with in our setting by using those special Lagrangians.

\subsection{Holomorphic Disks bounded by the Tropical Lagrangians}\label{subsect:top-bound-for-bdry}
In this and all the following sections, we will only consider the case when $L_f$ has genus $0$.

Recall that for each cylindrical end $Z_{\alpha,r}^+$ of the Lagrangian $L_f$, the angular 1-form is a multi-valued $1$-form $\eta_\alpha$, which is a primitive of $\Rea\Omega$ (See Definition~\ref{defn:canonical-angular-1-form}). By Lemma~\ref{def:integral-eta-alpha} the path integral of this $1$-form is determined up to $4\pi^2$, and since $\Rea\Omega$ vanishes on a special Lagrangian, the path integral of this $1$-form over $\gamma:[0,1] \to (\bC^*)^2$ depends on the homotopy class of $\gamma$ fixing its endpoints. Moreover, the integral is independent of the endpoints (and only depends on the homotopy class of $\gamma$) when $\gamma$ is nullhomotopic in $(\bC^*)^2$ since the integral is then equal to the integral of $\Rea\Omega$ over the disk bounded by $\gamma$. For the next lemma, we only need a special case of angular 1-form
$$\eta = p_1 dp_2 + \theta_2 d \theta_1$$
when $r=1$ and $\alpha = (1,0)$.

Now, let $f = \sum_{\alpha \in A} c_\alpha z^\alpha$ be a Laurent polynomial satisfying Assumption~\ref{ass:trop-ftn} (which implies the existence of its cylindrical ends).

\begin{lemma}\label{lem:unobstructed-Lf}
    If $L_f$ has genus zero and $\{\arg{c_\alpha}\}_{\alpha \in A} \cup \{2\pi\}$ is a rationally independent subset of $\bR$, then there are no holomorphic disks bounded by the Lagrangian $L_f$.
\end{lemma}

\begin{proof}
    Suppose there is a nontrivial holomorphic disk $D$ with boundary $\round D$ contained in $L_f$. Since $\Rea\Omega$ vanishes on $D$, we have $$\int_{\round D} \eta = \int_{D} \Rea\Omega = 0.$$
    On the other hand, the integral $\int_{\round D} \eta$ only depends on the homology class $[\round D]$ in $H_1(L_f)$, which is generated by loops $S_\alpha^1$ wrapping around each cylindrical end $Z_{\alpha,r}^+$ once. Note that because of the well-definedness, we only need to look at the curves that are nullhomotopic in $(\bC^*)^2$, i.e. that are in the image of the boundary map $H_2((\bC^*)^2, L_f) \to H_1 (L_f)$ for the long exact sequence of the pair $((\bC^*)^2, L_f)$. Also, by Proposition~\ref{prop:trop-approx}, each cylindrical end of $L_f$ is exponentially close to a HyperKähler rotation of the cylinder $z^{\alpha} = r$ with respect to $p_\alpha = \frac{1}{\abs{\alpha}^2}(\alpha_2 p_1 - \alpha_1 p_2) $, where $r = {c_{\beta_1} / c_{\beta_2}}$ for some distinct $\beta_1, \beta_2 \in A$. Therefore, by pushing $S_\alpha^1$ to infinity along the cylinder, we can conclude that the integral of $\eta$ along $S_\alpha^1$ is equal to the integral along the circle where $p_\alpha, p_\alpha^\perp, \theta_\alpha$ are constants and $\theta_\alpha^\perp$ increases or decreases by $2\pi$. On this circle, $\eta$ is equal to
    $$\left(\alpha_1 \alpha_2 \arg{r} + \frac{\alpha_1^2}{\abs{\alpha}^2}\theta_\alpha^\perp \right) d\theta_\alpha^\perp.$$
    Now, for any linear combination $\sum_{\alpha} n_\alpha [S_\alpha^1] \in H_1(L_f)$ that is nullhomotopic in $(\bC^*)^2$, if we fix points $P \in L_f$ and $P_\alpha \in S_\alpha^1 \subseteq Z_{\alpha,r_\alpha}^+$ whose $\theta_2$-coordinates are rational multiples of $2\pi$, and make a loop with homology class $\sum_{\alpha} n_\alpha [S_\alpha^1]$ by connecting each $P_\alpha$ to $P$, then the integral along the loop is equal to $2 \pi \sum_{\alpha} n_\alpha \alpha_1 \alpha_2 \arg {r_\alpha}$ plus some rational multiple of $4\pi^2$. Therefore, if we assume the rational independence of $\arg c_\alpha$, the integral of $\eta$ is zero only when its homology class is zero. Note that there are one-dimensional possibilities of $n_\alpha $'s that make the sum zero because the sum of $S_\alpha^1$ is zero under suitable orientation. This proves the non-existence of holomorphic disk $D$ because the symplectic area of $D$ should be $\int_{D} \omega = \int_{\round D}\lambda = 0$.
    \end{proof}

    \begin{remark}
        The rational independence assumption is crucial but can be alleviated for the simple cases we would like to consider. First, if $L_f$ has three cylindrical ends (including the case when $f=1+z_1 +z_2$, i.e. $L_f$ is a Lagrangian pair of pants), there are no nontrivial curves that are nullhomotopic in $(\bC^*)^2$, and the unobstructedness immediately follows. Also, for the Lagrangian $L_q$, the only nontrivial curve we need to consider is a sum of circles in two opposite cylindrical ends, and the integral of $\eta$ along this curve is $2\pi \arg q$, which implies that there are no holomorphic disks bounded by $L_q$ when $\arg q \notin 2\pi \bQ$. (In fact, we will show that $L_q$ does not bound a disk when $q \notin \bR$.)
    \end{remark}

   % [extend to $\widetilde{L}$]
    Even though this lemma proves that $L_f$ is unobstructed, it does not imply that the Lagrangian $\widetilde{L}_{f, R}$ does not bound any holomorphic disks. However, for the Lagrangian $L_q$, we have unobstructedness for these Lagrangians as well, as well as a partial control for the boundary of a holomorphic disk with one 'strip-like end' at one of the cylindrical ends. These disks normally exist and are counted in the wrapped Floer cohomology ring.

    \vspace{0.3em}
    \begin{figure}[ht]
        \centering
        \includegraphics[width=0.8\textwidth]{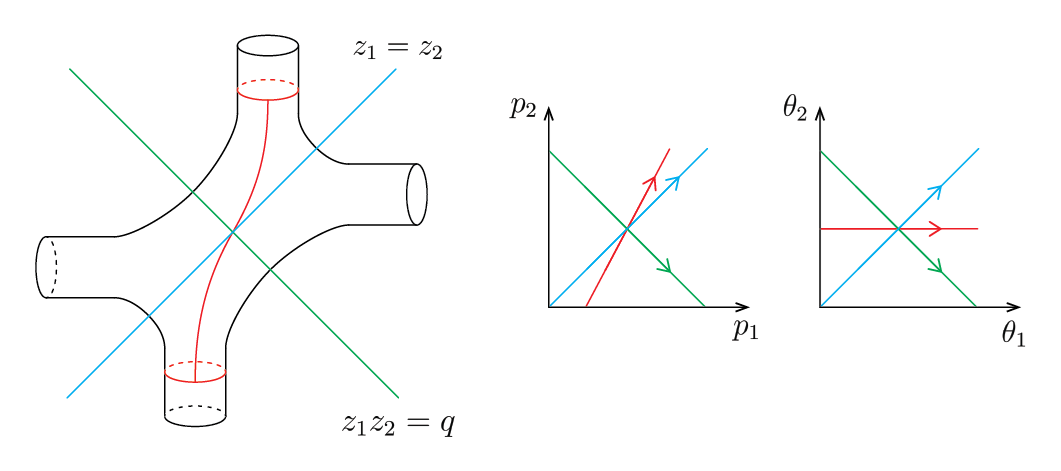}
        \caption{The left picture shows the boundary curve of the disk $D_0$ as well as the hypersurfaces $\Sigma_1$ and $\Sigma_2$ used in Lemma~\ref{lem:ass-top-bound-for-boundary}. The right picture shows their orientation.}
        \label{fig:Lq-ass-top-bdry}
    \end{figure}
    \vspace{0.3em}

    First, we note that the relative homology $H_2((\bC^*)^2, L_q)$ is 2-dimensional, generated by a torus $T^2 \subseteq (\bC^*)^2$ and a disk $(D_0, \round D_0)$ depicted in the Figure~\ref{fig:Lq-ass-top-bdry}. Then, we need two holomorphic curves $\Sigma_1$ and $\Sigma_2$ in $(\bC^*)^2$ satisfying the following.
    \begin{itemize}
        \item $\Sigma_1 \cap L_q = \Sigma_2 \cap L_q = \emptyset.$
        \item $(D_0, \round D_0)$ (with any orientation) intersect positively with one of $\Sigma_1$ and $\Sigma_2$, and negatively with the other.
    \end{itemize}
    The first condition implies that the intersection numbers between $\Sigma_1, \Sigma_2$ with $[D, \round D] \in H^2((\bC^*)^2, L_q )$ are well-defined. In addition, the second condition implies that the $D_0$-component of our holomorphic disks is bounded because of the following lemma.
    \begin{lemma}[\cite{GH1}, Section 4.3]\label{lem:complex-intersection-is-positive}
        For any two complex subvarieties $V, W$ of complementary dimension meeting at point $p$, the intersection multiplicity at $p$ is always positive.
    \end{lemma}

    For simplicity, we abbreviate $\widetilde{L}_{q,R}^k := \widetilde{L}_{f_q , R}(-kH_{f_q , R})$ from now on, where $f_q = 1-z_1 - z_2 + q^{-1}z_1 z_2$ is the Laurent polynomial defining $L_q$.
    \begin{lemma}\label{lem:ass-top-bound-for-boundary}
        If $q \notin \bR$, then for every $R>0$ and a positive integer $k$ the Lagrangian $\widetilde{L}_{q,R}^k$ is unobstructed and satisfies Assumption~\ref{ass:top-bound-for-boundary}.
    \end{lemma}

    \begin{proof}
        For $q \notin \bR$, consider the following complex curves $\Sigma_1$ and $\Sigma_2$ in $(\bC^*)^2$.
        \begin{equation}
            \Sigma_1 = \{z_1 = z_2\}, \quad \Sigma_2 = \{z_1z_2 = q\}.
        \end{equation}
        To see that they don't intersect with $L_q$, we note that $\Sigma_1$ and $\Sigma_2$ are HyperKähler rotations of $\{z_1 = \Bar{z}_2\}$ and $\{z_1 \Bar{z}_2 = q\}$, respectively. One can easily check that these two don't intersect with $D_q$ when $q$ is not real. This implies that for $j=1,2$, there exists a well-defined intersection number between $\Sigma_j$ and a relative homology class
        \begin{align*}
            [D, \round D] &\in H_2((\bC^*)^2, L_q)\text{,\quad or}\\
            [D, \round D] &\in H_2 ((\bC^*)^2, L_q \cup \{p_\alpha \ge R, \,\abs{p_\alpha^\perp - \log\abs{r}} \le 1\})
        \end{align*}
        for any cylindrical end $Z_{\alpha,r}^+$ of $L_q$. Finiteness comes from the compactness of $(D, \round D)$. The homology group $H_2((\bC^*)^2, L_q)$ is generated by a torus and $(D_0, \round D_0 )$, and the group $H_2 ((\bC^*)^2, L_q \cup \{p_\alpha \ge R, \abs{p_\alpha^\perp - \log\abs{r}} \le 1\})$ has an additional generator since the homology group $H_1(\{p_\alpha \ge R,\, \abs{p_\alpha^\perp - \log\abs{r}} \le 1\}) \cong H_1(T^2)$ consists of a generator not in the $S^1$ direction of the cylindrical end $Z_{\alpha,r}^+$ (This is what we called $[S_{\alpha^\perp}^1]$ in Assumption~\ref{ass:top-bound-for-boundary}). The intersection number between $\Sigma_j$ and a torus is zero since it can be chosen outside the lines $p_1 = p_2$ and $p_1 + p_2 = \log |q|$. On the other hand, we can choose a representative of $(D_0, \round D_0)$ to be a cylinder connecting two circles in two opposite cylindrical ends (see Figure~\ref{fig:Lq-ass-top-bdry}). One can observe that this cylinder has intersection numbers $1$ and $-1$ with $\Sigma_1$ and $\Sigma_2$, with signs depending on the orientation of the cylinder.

        Now unobstructedness directly comes from Lemma~\ref{lem:complex-intersection-is-positive}, since any holomorphic disks should have nonnegative intersection numbers with both $\Sigma_1$ and $\Sigma_2$, hence only generated by $T^2$. This implies that the area of the holomorphic disk is $0$ because the boundary is trivial. For the second part, let $D_{\alpha}$ be a generator of $H_2 ((\bC^*)^2, L_q \cup \{p_\alpha \ge R, \abs{p_\alpha^\perp - \log\abs{r}} \le 1\})$ other than $T^2$ and $D_0$. If the intersection number between $D_{\alpha}$ and $\Sigma_j$ is $N_j$ for $j=1,2$, then the $D_0$-component cannot exceed $\max\{|N_1 |, |N_2|\}$ because of the positiveness of the intersection number. Since $D_{\alpha}$ is the only generator that has a $[S_{\alpha^\perp}^1]$-component, the result follows.
    \end{proof}

%%%%%%%%%%%%%%%%%11111111111%%%%%%%%%%%%%%%%%%%

\subsection{Generators of \texorpdfstring{$HW^0$}{HW⁰}}\label{subsect:distinguished-basis}

Throughout the remaining sections, we describe a homological mirror symmetry of tropical Lagrangians $L_q$. In Section~\ref{subsect:FloerThoery-basic}, we observed that the Floer theory is defined over a non-archimedean field $\Lambda$ called the Novikov field. Also, the Lagrangian is equipped with the following unitary rank 1 (Novikov) local system.
\begin{notation}
    We equip $L_q$ with a local system which has a trivial holonomy over the cylindrical ends corresponding to $p_1 \to -\infty$ and $p_2 \to -\infty$, and some unitary holonomy $\xi$ around the circle on the cylindrical ends corresponding to $p_1 \to \infty$ and $p_2 \to \infty$.
\end{notation} 
Finally, there exists a suitably chosen brane structure determined by the signs in the equation of the curve, which will be omitted for the remainder of the section.

The construction of tropical Lagrangians and Hamiltonians in Section~\ref{subsect:Lagn-and-Ham-def} are closely related to the notion of \textit{monomially admissible Lagrangians}, which is used to describe a mirror symmetry for toric varieties. Namely, if a toric variety $\check{X}_\Sigma$ is given by the fan $\Sigma \subseteq \bR^n$, then it defines a Laurent polynomial
$$W_{\Sigma} = \sum_{\alpha \in \Sigma} c_\alpha z^\alpha : (\bC^*)^n \to \bC$$
called the Hori-Vafa superpotential for $\check{X}_{\Sigma}$ \cite{HV1}. On the symplectic side, in $(\bC^*)^n$, we have a notion of a monomial division as a taming condition for this Laurent polynomial, as well as the notion for the monomial admissible Lagrangians as follows.

\begin{definition}[\cite{Han1}]
    A \textit{monomial division} $\Delta_\Sigma$ for $W_{\Sigma} = \sum_{\alpha \in \Sigma} c_\alpha z^\alpha$ is an assignment of a closed set $C_\alpha \subseteq \bR^n$ to each monomial in $W_\Sigma$ such that the following conditions hold.
    \begin{itemize}
        \item $\{C_\alpha\}_{\alpha \in \Sigma}$ cover the complement of a compact subset of $\bR^n$.
        \item There exist constants $k_\alpha \in \bR_{>0}$ such that for all $z \in (\bC^*)^n$ outside a compact subset, if the maximum
        $$\max_{\alpha \in \Sigma} |c_\alpha z^\alpha |^{k_\alpha}$$
        is achieved by $|c_\alpha z^\alpha |^{k_\alpha}$, then $\Log(z) \in C_\alpha$.
        \item $C_\alpha$ is a subset of the open star of the ray $\alpha$ in the fan $\Sigma$.
    \end{itemize}
    A Lagrangian $L \subseteq X$ is \textit{$\Delta_\Sigma$-monomially admissible} if $\arg (c_\alpha z^\alpha )$ is zero on $L \cap \Log^{-1}(C_\alpha)$ outside a compact set.
\end{definition}

One can construct an $A_\infty$-category $\Fuk_{\Delta_\Sigma}(X, W_\Sigma)$ whose objects are $\Delta_\Sigma$-admissible Lagrangians, and there exists a homological mirror statement described as follows.

\begin{theorem}[\cite{Han1}]\label{thm:mono-adm-mirror-symmetry}
    Let $\check{X}_{\Sigma}$ be a smooth complete toric variety with Hori-Vafa superpotential $W_\Sigma$. Also, let $\text{Tw}^\pi \mathcal{P}_{\Delta_\Sigma} (X, W_\Sigma)$ be a category of twisted complexes generated by the full subcategory of $\check{X}_{\Sigma}$ consisting of tropical Lagrangian sections. Then, the $A_\infty$-categories
    \begin{equation}\label{eqn:mono-adm-mirror-symmetry}
        \text{Tw}^\pi \mathcal{P}_{\Delta_\Sigma} (X, W_\Sigma) \simeq D^b \Coh(\check{X}_{\Sigma})
    \end{equation}
    are quasi-equivalent.
\end{theorem}

The analysis of tropical cylindrical ends in Section 2 implies that the Lagrangians $\widetilde{L}_{q, R}^k$ are monomially admissible with respect to a monomial division for a fan $\Sigma$ whose rays are in the same direction as the cylindrical ends of $L_q$. Moreover, these Lagrangians are Hamiltonian isotopic to the tropical Lagrangians $L(\phi)$ constructed in \cite{Hic1}, where $\phi$ is a tropical polynomial, namely the tropicalization of $f_q$, and the Lagrangian is constructed as a surgery of two Lagrangian sections over $\Log: (\bC^*)^n \to \bR^n$. This construction covers a wider variety of tropical Lagrangians as well as the following partial description of the corresponding objects on the mirror.

\begin{theorem}[\cite{Hic1}]\label{thm:L-phi-is-O-D'}
    Let $D$ be a base-point-free divisor of a toric variety $\check{X}_{\Sigma}$ with tropicalization given by the tropical skeleton $V(\phi)$. Then, the corresponding tropical Lagrangian $L(\phi)$ is homologically mirror to a structure sheaf $\mathcal{O}_{D'}$, where $D$ and $D'$ are rationally equivalent.
\end{theorem}

Since the Lagrangians $L(\phi)$ can be non-exact, the theorem is based on the following assumption to fit them into the framework in the monomially admissible Fukaya-Seidel category, which will also be assumed to be true in this paper.
\begin{assumption}\label{ass:mono-adm-non-exact}
    We assume that the monomially admissible Fukaya-Seidel category can be extended to include unobstructed Lagrangian submanifolds and that the appropriate analogues of Theorem~\ref{thm:mono-adm-mirror-symmetry} hold in this setting.
\end{assumption}

The relation between the tropical Lagrangians $L(\phi)$ and $\widetilde{L}_{q, R}$ are summarized as follows.

\begin{enumerate}
    \item If $\phi_q$ is the tropical polynomial of the Laurent polynomial $f_q = 1-z_1- z_2 + q^{-1}z_1 z_2$ (Definition~\ref{def:trop-poly-of-f}), then $L(\phi)$ is Hamiltonian isotopic to $\widetilde{L}_{q,R}$. This is from the fact that both Lagrangians can be expressed as a union of graphs after dividing them into six pieces, one for each cylindrical end and one for each vertex of the tropical skeleton (See Figure~\ref{fig:compare-hicks-me}). On each cylindrical end, we can express both Lagrangians as a graph over the cylinder $Z_{\alpha,r}^+$. On each vertex of the tropical skeleton, both can be viewed as a section over a configuration with two triangles in $(\theta_1, \theta_2)$-coordinate torus, bounded by the edges $\theta_1 = 0$, $\theta_2 = 0$, and $\theta_1 + \theta_2 = \frac{\pi}{2}$. For $\widetilde{L}_{q,R}$, this is from the direct computation; for example, on the region where $q^{-1}z_1 z_2$ is smaller than $1$, the Lagrangian is $C^1$-close to the pair of pants $\{1 - z_1 - z_2 = 0\}$, and each edge of the triangle corresponds to (the HyperKähler rotation of) the limit case $z_1 = 1$, $z_2 = 1$, and $z_1 + z_2 = 0$. For $L(\phi)$, this is an alternate construction of the Lagrangian; see Section 1.2 and Section 3.4 in \cite{Hic2}.

    \vspace{0.3em}
    \begin{figure}[ht]
        \centering
        \includegraphics[width=0.7\textwidth]{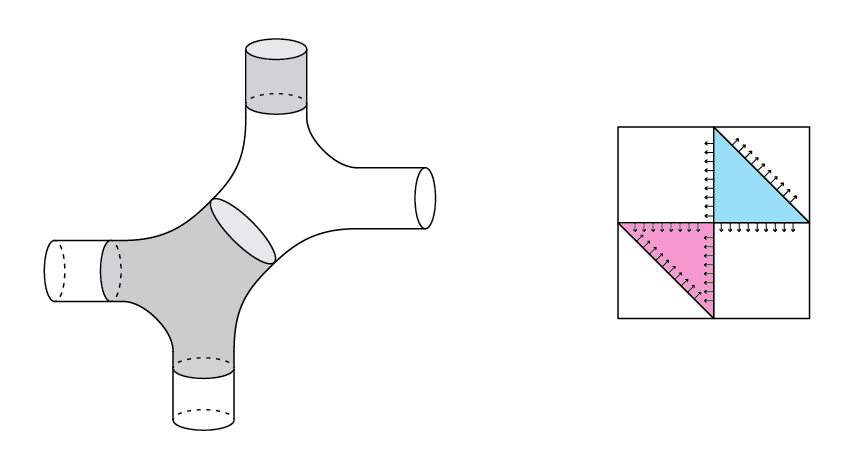}
        \caption{The left picture shows two different types of pieces. The right picture shows the projection of the interior pieces to the base space of the cotangent bundle $(\bC^*)^2 \simeq T^* T^2$. The arrows represent the direction of the ends of the Lagrangian in the fibers of the cotangent bundle.}
        \label{fig:compare-hicks-me}
    \end{figure}
    \vspace{0.3em}

    \item $L(\phi_q)$ and $\widetilde{L}_{q,R}$ are both monomially admissible with respect to the same monomial division $\Delta$, and their Hamiltonian perturbations are Hamiltonian isotopic to each other. In our setting, this is equivalent to say that $\widetilde{L}_{q, R}$ is a straight line on each $(p_\alpha, \theta_\alpha)$-plane outside of the compact region. For $L(\phi)$, see Section 5 in \cite{Hic1}. In the monomially admissible Lagrangian setting, one use a Hamiltonian $K = K(p_1, p_2)$ such that $\nabla K \cdot \alpha = 1$ on each leg, which corresponds to the term $2\pi p_\alpha$ in Definition~\ref{def:Hf-tilde}. Note that the convention is a bit different in these settings; we actually have $\nabla (2\pi p_\alpha) \cdot (\alpha_2, -\alpha_1) = 1$.

    \item Even though we are defining the wrapped Floer cohomology group $HW^\bullet (L_q, L_q)$ directly as a limit of `partially wrapped' Floer theory between $\widetilde{L}_q$ and $\widetilde{L}_q^{k}$, each partially wrapped Floer cohomology group can also be computed in the Fukaya-Seidel type approach for the monomial admissible Lagrangians, although they are now mirror to the coherent sheaves in the compactified toric variety $\check{X}_\Sigma = \mathbb{P}^1 \times \mathbb{P}^1$. Wrapping once from $\widetilde{L}_{q}^{k-1}$ to $\widetilde{L}_{q}^k$ is equivalent to the twisting $- \otimes \mathcal{L}(D)$ with the canonical toric divisor $D$ of the toric variety $\mathbb{P}^1 \times \mathbb{P}^1$, which at the basis level, is equivalent to adding functions on $D_q$ with poles of order $k$ at each of the $4$ points of $\overline{D_q} \setminus D_q$. See section 5 of \cite{Hic1} and Section 4 of \cite{Han1} for the details.
\end{enumerate}

Since we are working over the Novikov field, the Lagrangian $L_q$ is not mirror to the divisor $D_q$ in $(\bC^*)^2$, but to a divisor in $(\Lambda^*)^2$. Note also that Theorem~\ref{thm:L-phi-is-O-D'} determines the divisor up to rational equivalence, which, for example, means that any divisor defined by an equation $A + Bz_1 + Cz_2 +Dz_1 z_2 = 0$ with nonzero coefficients $A, B, C, D$ can be a possible candidate. However, we can further compute the Floer cohomology between $L_q$ and a fiber torus to show that the corresponding divisor is the divisor $D_Q$ in $(\Lambda^*)^2$ defined by
\begin{equation}\label{eqn:fq-novikov-coeff}
    f_Q := 1 - z_1 - z_2 + Q^{-1}z_1 z_2 = 0
\end{equation}
where $Q = T^{2\pi \log \abs{q}} \xi$. This idea is based on the family Floer argument in \cite{Abo2} and is explained in Section~\ref{subsect:basis-functions}, Remark~\ref{rem:fiber-torus-support}.

From these observations, we can obtain a homology-level description of our Floer cohomology groups as follows. We set $\bK:= \Lambda$ for the remainder of the sections.

\begin{theorem}\label{thm:HW0-cohomology-group}
    The following is true for every $q \in \bC \setminus \bR$ and large enough $R > 0$.
    \begin{itemize}
        \item Under the quasi-isomorphism (\ref{eqn:mono-adm-mirror-symmetry}), the Lagrangian $\widetilde{L}_{q,R}$ is mirror to a structure sheaf $\mathcal{O}_{D_Q}$. In particular, there exists a ring isomorphism
        \begin{equation}\label{eqn:HW0-iso-to-qtt}
            HW^0 (L_q, L_q) \cong \mathbb{K}[z_1^{\pm}, z_2^{\pm}] / f_Q.
        \end{equation}
        %\item For any integer $k$, the Floer cohomology group $HF^0 (\widetilde{L}_{q,R}^k , \widetilde{L}_{q,R}^0)$ is isomorphic to the quotient group
        %\begin{equation}\label{eqn:HW0-iso-to-qtt-k}
        %    \mathbb{K}[z_1^{\pm}, z_2^{\pm}] / (f_Q, z_1^{-(k+1)}, z_1^{k+1}, z_2^{-(k+1)}, z_2^{k+1})
        %\end{equation}
        \item \label{eqn:HW0-iso-to-qtt-k} For any integer $k$, the Floer cohomology group $HF^0 (\widetilde{L}_{q,R}^k , \widetilde{L}_{q,R}^0)$ is isomorphic to the subspace of $\mathbb{K}[z_1^{\pm}, z_2^{\pm}] / f_Q$ generated by the monomials $z_1^{m_1} z_2^{m_2}$ with
        $ \abs{m_1}, \abs{m_2} \le k. $
    \end{itemize}
\end{theorem}
$\hfill\square$

%Throughout this section, we assume that the Laurent polynomial $f$ satisfies Assumption~\ref{ass:top-bound-for-boundary}, and that $L_f$ has genus zero. In particular, if $b$ is the number of cylindrical ends of $L_f$, there are $2g + b - 2 = b-2$ interior generators by Lemma~\ref{lem:Riemann-translation-intersection}. Also, we will always use the standard complex structure, see Remark

Now, on the basis level, we can use the constraint for the strips defining the differential, from Theorem~\ref{thm:disk-bound-arg-strip} and the following observation.

\begin{lemma}\label{lem:arg-constraints-different-cylinders}
    Let $f$ be a Laurent polynomial satisfying Assumption~\ref{ass:trop-ftn}. Let $S$ be a holomorphic disk with $d+1$ boundary points $\zeta_0, \zeta_1 ,\cdots, \zeta_d$ in counterclockwise order, and consider holomorphic maps $u:S \to (\bC^*)^2$ bounded by Lagrangians of the form $\widetilde{L}_{f, R}^{k_s}$, following the conventions in Section~\ref{subsect:FloerThoery-basic}. Then, if $u(\zeta_0)$ lie on a cylindrical end $Z_{\alpha,r}^+$ with \textit{$x_\alpha$-exponent} $i$, i.e. it is either $x_\alpha^i x_\alpha^e$ or $x_\alpha^i x_\alpha^f$, then at least one of $u(\zeta_j)$ with $j>0$ also lie on $Z_{\alpha,r}^+$, with equal or higher $x_\alpha$-exponent if there is only one such $j>0$. 
\end{lemma}
\begin{proof}
    The proof is similar to \ref{thm:Gromov-compactness}, with a more careful analysis of the cylindrical generators. First, each $\widetilde{L}_{f, R}^{k_s}$ is a (partially tilted) ray on the cylindrical region, so the maximum of $p_\alpha \circ u$ can only be achieved on one of the intersection points $u(\zeta_j)$. Also, the output intersection point $u(\zeta_0)$ lie between $\widetilde{L}_{f, R}^{k_0}$ and $\widetilde{L}_{f, R}^{k_d}$ with $k_0 < k_d$, which implies that $p_\alpha$ cannot attain its local maximum on $u(\zeta_0)$ since $u$ covers at least one of the obtuse-angled region between the rays. Therefore, the maximum should be achieved on one of the input boundary points $\zeta_j$. 

    The last part comes from the fact that two Lagrangians $\widetilde{L}_{f, R}^{k_s}$ passing through the intersection points with the $x_\alpha$-exponent $j$ are separated by $2j\pi + O(R^{-1})$ in the cylindrical region. If there are only two generators on $Z_{\alpha,r}^+$ then image of $u$ contains a strip on this cylindrical end with width $2j\pi + O(R^{-1})$ for some nonnegative integer $j$, and the $x_\alpha$-exponent of the input generator is exactly $i+j$.
\end{proof}

The strips that are counted in the differential should be from degree 0 generators to degree 1 generators, and there are three such possibilities.
\begin{enumerate}
    \item Strips contained in one of the cylindrical region $\{p_\alpha \ge R^2 + 3R\}$. We can express this strip as a product of two strips in a complex plane after lifting it to the universal cover $(\widetilde{\bC^*})^2$ with the identification $(\widetilde{\bC^*})^2 \cong \bC_\alpha \times \bC_\alpha^\perp$ described in the proof of Theorem~\ref{thm:disk-bound-arg-strip}. There are two strips from $x_\alpha^i x_\alpha^e$ to $x_\alpha^i x_\alpha^f$, corresponding to two strips in $\bC_\alpha^\perp$, which are canceled out with each other and do not contribute to the differential.

    \item Strips from $x_\alpha^i x_\alpha^e$ to $x_\alpha^j x_\alpha^f$ with $i < j$, that passes through the interior region.

    \item Strips from $x_\alpha^i x_\alpha^e$ to an interior generator.
\end{enumerate}

On the chain level, all the degree 0 generators are of the form $x_\alpha^i x_\alpha^e$. If we choose the translation vector $(a_1, a_2)$ for the Hamiltonian to satisfy $a_1 < 0 < a_2$, then the index $i$ of these generators are from $0$ to $k-1$ for two legs on the negative direction (where $\alpha = (0, -1), (-1,0)$) and from $1$ to $k$ for two legs on the positive direction (where $\alpha = (0, 1), (1,0)$) 
%See also Figure \ref{fig:Lq-HW0-generators}.

\begin{notation}
    For the remainder of this section, we set
    \begin{align*}
        x^i x^e & := x_{(1,0)}^i x_{(1,0)}^e, \quad\quad x^{-i} x^e := x_{(-1,0)}^{i} x_{(-1,0)}^e,\\
        y^i y^e & := x_{(0,1)}^i x_{(0,1)}^e, \quad\quad y^{-i} y^e := x_{(0,-1)}^{i} x_{(0,-1)}^e.
    \end{align*}
    to denote generators in $HF^0 (\widetilde{L}_{q,R}^k , \widetilde{L}_{q,R}^0)$. The range of the index $i$ is from $-k$ to $k$.
\end{notation}

On the other hand, one can show that the subspace in Theorem~\ref{thm:HW0-cohomology-group}(\ref{eqn:HW0-iso-to-qtt-k}) has dimension $4k-1$, with generators $z_1^i$ and $z_2^i$ for $-k < i \le k$ (Note that $z_1^0 = z_2^0 = 1$). Therefore, the differential $\mu^1:CF^0 (\widetilde{L}_{q,R}^k , \widetilde{L}_{q,R}^0) \to CF^1 (\widetilde{L}_{q,R}^k , \widetilde{L}_{q,R}^0)$ has rank 1, which gives the following description on the chain level. Recall that the generators of the form $x_\alpha^0 x_\alpha^e$ are called \textit{innermost generators}.

\begin{proposition}[Distinguished Basis of $HW^0(L_q, L_q)$]\label{prop:distinguished-basis-of-Hw0}
    Let $q \in \bC \setminus \bR$. Then, there exist constants $a_i$ and $b_i$ for $i \in \bZ$ such that the Floer cohomology groups $HF^0 (\widetilde{L}_{q,R}^k , \widetilde{L}_{q,R}^0)$ has a basis consists of the elements
    \begin{itemize}        
        \item $x^ix^e + b_i y^0 y^e$, \qquad for $-k < i \le k$ with $i \neq 0$.
        \item $y^iy^e + a_i x^0 x^e$, \qquad for $-k < i \le k$ with $i \neq 0$.
        \item $x^0 x^e + y^0 y^e$.
    \end{itemize}
    In particular, these generators form a basis for the wrapped cohomology group $HW^0(L_q, L_q)$. Moreover, $x^0 x^e + y^0 y^e$ is the cohomological unit of $HW^0(L_q, L_q)$.
\end{proposition}

\begin{proof}
    We will show that $x^0 x^e + y^0 y^e$ is the unit in the next section. Assuming this, the result directly comes from the fact that the differential has rank 1. Also, the image of the differential is contained in the subspace generated by interior generators since the differential of innermost generators cannot contain any generators on the cylindrical region. This fact, together with Proposition~\ref{prop:htpy-mtd-q-iso} and Proposition~\ref{prop:HW-is-limit-HfR}, shows that the coefficients and generators are passed through higher Floer cohomology groups and the wrapped Floer cohomology $HW^0(L_q, L_q)$.
\end{proof}

%%%%%%%%%%%%%%%%%222

\subsection{The Product Structure}\label{subsect:product}
In this section, we compute the product structure on the ring $HW^0(L_q, L_q)$. One needs to count holomorphic triangles instead of holomorphic strips for the product structure. We have similar constraints to the one in Theorem~\ref{thm:disk-bound-arg-strip}.
    
\begin{theorem}[Argument Constraints, for Triangles]\label{thm:disk-bound-arg-triangle}
    Let $f$ be a Laurent polynomial satisfying Assumption~\ref{ass:top-bound-for-boundary}. Let $S$ be a holomorphic disk with three boundary points $\zeta_0, \zeta_1, \zeta_2$ in counterclockwise order, and consider holomorphic maps $u:S \to (\bC^*)^2$ bounded by Lagrangians of the form $\widetilde{L}_{f, R}^{k_s}$, following the conventions in Section~\ref{subsect:FloerThoery-basic}, and where $j=0,1,2$. Also, assume that all three $u(\zeta_s)$ lie on the same cylindrical end $Z_{\alpha,r}^+$, and of the form $x_\alpha^{j_s} x_\alpha^e$ for some integer $j_s$. Then, for the generic choice of $\{\varphi_\alpha\} \subseteq \bR$, the following is true:
    \begin{changemargin}{}{}
        For any positive integer $k$, there exists $R_0 > 0$ such that any holomorphic map $u:S \to (\bC^*)^2$ bounded by Lagrangians $\widetilde{L}_{f, R}^{k_s}$ with $R > R_0$ and integers $k_s \le k$ should either satisfy $j_0 \le \max{(j_1, j_2)}$ or be contained in the cylindrical region.
        %For any positive integer $k$, there exists $R_0 > 0$ such that any such holomorphic map $u:S \to (\bC^*)^2$ bounded by Lagrangians $\widetilde{L}_{f, R}^{k_s}$ with $R > R_0$ and integers $k_s \le k$ should satisfy either $ \max{(j_1, j_2)} = j_0$ or $\min{(j_1, j_2)} = 0$.
    \end{changemargin}
    Moreover, the equality of the last inequality can be dropped when $Z_{\alpha,r}^+$ contains an innermost generator.
\end{theorem}

\begin{proof}
    The proof is essentially the same as Theorem~\ref{thm:disk-bound-arg-strip}. The only difference is that the argument only works when each part of $u$ in the interior region is bounded by one of the three Lagrangians. Figure~\ref{fig:arg-constraints-triangle} shows two types of possible configurations in the $(p_\alpha, \theta_\alpha)$-plane, together with the cut (dashed line). The first configuration, where each cut has its endpoints on the same Lagrangian, can be excluded by the argument constraints in Theorem~\ref{thm:disk-bound-arg-strip}. Note that in the $(p_\alpha^\perp, \theta_\alpha^\perp)$-plane, the image of $u$ should be sufficiently close to the small triangle near  $\varphi_\alpha$.
    
    In the second case, where each cut has its endpoints on different Lagrangians, the exponent of the output generator is bounded by one of the input generators, which implies that $j_0 \le \max{(j_1, j_2)}$. The equality holds when the strip between these generators is `thin', i.e. has width $O(R^{-1})$, which exists only when the cylindrical end does not contain an innermost generator.
    %The proof is essentially the same as Theorem~\ref{thm:disk-bound-arg-strip} but needs a more careful analysis on the disk $u$ in the cylindrical region (above the cut $\{p_\alpha = t\}$). Figure~\ref{fig:arg-constraints-triangle} shows two possible types of configurations in the $(p_\alpha, \theta_\alpha)$-plane, together with the cuts (dotted lines). The first configuration, where each cut has its endpoints on the same Lagrangian, can be excluded by the argument constraints in Theorem~\ref{thm:disk-bound-arg-triangle}. Note that in the $(p_\alpha^\perp, \theta_\alpha^\perp)$-plane, the image of $u$ should be sufficiently close to the small triangle near  $\varphi_\alpha$.  
    
    \vspace{0.3em}
    \begin{figure}[ht]
        \centering
        \includegraphics[width=0.7\textwidth]{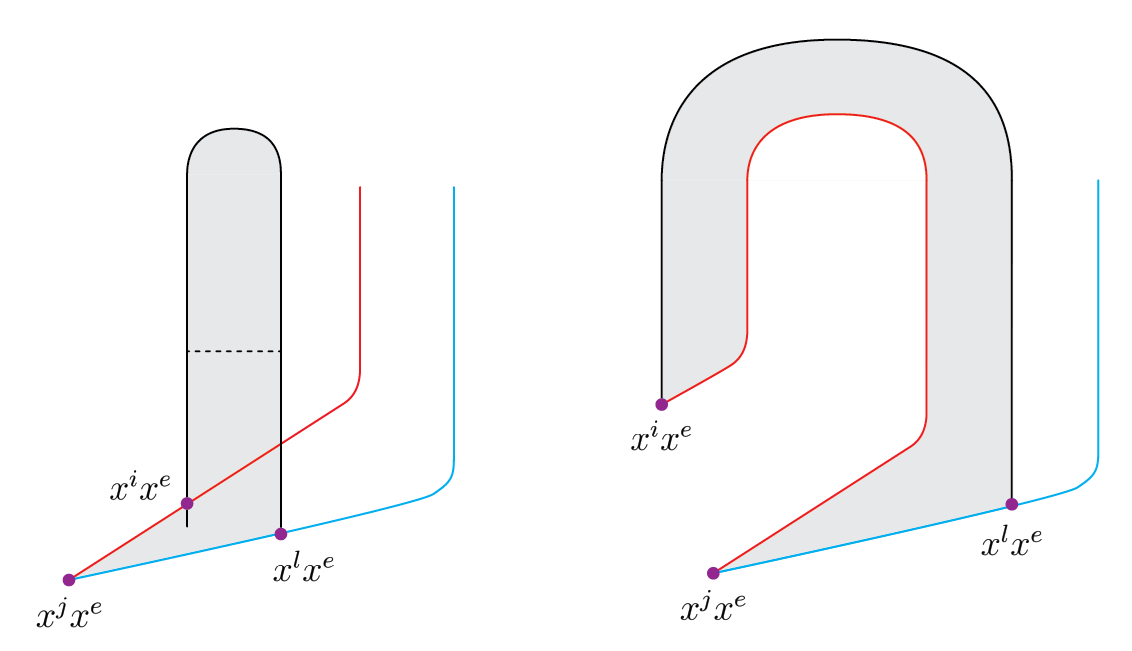}
        \caption{Two types of possible configuration of disks in Theorem~\ref{thm:disk-bound-arg-triangle} not contained in the cylindrical region.}
        \label{fig:arg-constraints-triangle}
    \end{figure}
    \vspace{0.3em}
    
    %In the second case, where each cut has its endpoints on different Lagrangians, the exponent of the output generator is bounded by one of the input generators, which implies that $j_0 \le \max{(j_1, j_2)}$. However, if both strip parts between these generators are not `thin', i.e. have width $2k\pi + O(R^{-1})$ for some integer $k>0$, then we can still show that the $\theta_\alpha^\perp$-value on $u(S) \cap \{p_\alpha = t\}$ is sufficiently close to $\varphi_\alpha$ by the following lemma.

    %\begin{lemma}
    %    Let $\Sigma_{R,W} = [0,R] \times [0,W] \subseteq \bR \times i\bR = \bC$ be a strip with width $W$ and length $R$, and let $\varepsilon>0$. Then, there exists a constant $C>0$ independent of $R,W,\varepsilon$ such that the following is true.
    %    \begin{changemargin}{}{}
    %        If a continuous function $f:\Sigma_{R,W} \to \bC$ that is holomorphic in the interior of the domain satisfies $\abs{\Ima f} < \varepsilon$ on $\Sigma_{R,W}$, then $\abs{\Rea f} < CR \varepsilon / W$ on $\Sigma_{R,W}$.
    %    \end{changemargin}
    %\end{lemma}

    %Here, Each strip part is thin when $\min{(j_1, j_2)} = 0$ and $\max{(j_1, j_2)} = j_0$, respectively, which concludes the proof of Theorem~\ref{thm:disk-bound-arg-triangle}
    
\end{proof}

%Even though the second case of the above figure does not exactly fit in the framework in the proof of Theorem~\ref{thm:disk-bound-arg-strip}, we can still expect that similar results will hold for these cases. First, the Assumption~\ref{ass:top-bound-for-boundary} implies that there is only a finite number of possibilities for the homology class of the boundary of $u$ in the interior region. Also, if we consider the part of $u$ in the cylindrical region, this can be viewed as a map from the strip of width $2k\pi + O(R^{-1})$ in the $(p_\alpha, \theta_\alpha)$-plane from the strip of width $O(R^{-3})$ in the $(p_\alpha^\perp, \theta_\alpha^\perp)$-plane. Considering that the length of the former strip is $O(R^2)$, we can conjecture that the $\theta_\alpha^\perp$-value on the cut $u(S) \cap \{p_\alpha = t\}$ is still sufficiently close to $\varphi_\alpha$ if $k \neq 0$ (i.e. the strip is not thin). From this observation, we conjecture the following stronger version of the theorem. %and consider the composition with the projection $\pi_\alpha: (\widetilde{\bC^*})^2 \to C_\alpha$ and $\pi_\alpha: (\widetilde{\bC^*})^2 \to C_{\alpha^\perp}$ after lifting to the universal cover $(\widetilde{\bC^*})^2$ of $(\bC^*)^2$%

%\begin{conjecture}\label{conj:disk-bound-arg-triangle}
%    The holomorphic map $u:S \to (\bC^*)^2$ in Theorem~\ref{thm:disk-bound-arg-triangle} should satisfy either $ \max{(j_1, j_2)} = j_0$ or $\min{(j_1, j_2)} = 0$.
%\end{conjecture}

Together with Lemma~\ref{lem:arg-constraints-different-cylinders}, the theorem tells us that we have three possibilities for the holomorphic triangles on the degree 0 level.

\begin{enumerate}
    \item\label{tri:i-plus-j} Holomorphic triangles contained in one of the cylindrical region $\{p_\alpha \ge R^2 + 3R\}$. Similar to the case of strips, we can express this triangle as a product of a triangle in the $(p_\alpha, \theta_\alpha)$-plane and a triangle in the $(p_\alpha^\perp, \theta_\alpha^\perp)$-plane. There is one triangle for each pair $(i,j)$, corresponding to $x_\alpha^i x_\alpha^e \times x_\alpha^j x_\alpha^e = x_\alpha^{i+j} x_\alpha^e$.

    \item\label{tri:alpha-beta} Triangles with inputs $x_\alpha^i x_\alpha^e$ and $x_\beta^j x_\beta^e$, where $\alpha \neq \beta$. By Lemma~\ref{lem:arg-constraints-different-cylinders}, the output is either $x_\alpha^{i'} x_\alpha^e$ with $i' \le i$ or $x_\beta^{j'} x_\beta^e$ with $j' \le j$.

    \item\label{tri:alpha-alpha} Triangles with inputs $x_\alpha^i x_\alpha^e$ and $x_\alpha^j x_\alpha^e$, passing through the interior region. By Theorem~\ref{thm:disk-bound-arg-triangle}, the output is $x_\alpha^{l} x_\alpha^e$ with $l \le \max(i,j)$.
\end{enumerate}
The equalities can be dropped whenever the cylindrical end $Z_{\alpha,r}^+$ contains an innermost generator. Conveying these observations to the product map $\mu^2$, we have the following corollary.

\begin{corollary}\label{cor:prod-structure-basic}
    The following is true for the degree 0 generators $x_\alpha^i x_\alpha^e$ in $HF^\bullet (\widetilde{L}_{f,R}^{k_0} , \widetilde{L}_{f,R}^{k_1})$ for any Laurent polynomial $f$ satisfying Assumption~\ref{ass:top-bound-for-boundary}.
    \begin{enumerate}
        \item For every $\alpha$ and $i,j \ge 0$, there exist scalars $a_l$ such that
        $$\mu^2 (x_\alpha^i x_\alpha^e , x_\alpha^j x_\alpha^e) = x_\alpha^{i+j}x_\alpha^e + \sum_{l \le \max{(i,j)}} a_l x_\alpha^l x_\alpha^e.$$
        \item For every $\alpha \neq \beta$ and $i,j \ge 0$ there exist scalars $a_l$, $b_l$ such that
        $$\mu^2 (x_\alpha^i x_\alpha^e , x_\beta^j x_\beta^e) = \sum_{i' \le i} a_{i'} x_\alpha^{i'} x_\alpha^e \,\,+\,\, \sum_{j' \le j} b_{j'} x_\beta^{j'} x_\beta^e.$$
        \item In the results above, the equality on the index of summation can be dropped whenever the cylindrical end $Z_{\alpha,r}^+$ contains the innermost generator $x_\alpha^0 x_\alpha^e$.
    \end{enumerate}
\end{corollary}

The computation of product structures, as well as finding corresponding elements in (\ref{eqn:HW0-iso-to-qtt}), involves a more detailed counting of these triangles. However, these observations still give various properties of the product structure, leading to a partial description of the isomorphism (\ref{eqn:HW0-iso-to-qtt}) on the basis level.

We begin from the following corollary, which is a part of Proposition~\ref{prop:distinguished-basis-of-Hw0}.

\begin{lemma}\label{lem:unit-of-Hw0}
    For $q \in \bC \setminus \bR$, the product of two innermost generators in $CF^0 (\widetilde{L}_{q,R}^k , \widetilde{L}_{q,R}^0)$ is itself when they are equal and $0$ otherwise. Moreover, $x^0 x^e + y^0 y^e$ is the cohomological unit of $HW^0 (L_q, L_q)$.
\end{lemma}
\begin{proof}
    The first part is directly obtained from Corollary~\ref{cor:prod-structure-basic} since all the equality on the summation index can be dropped. For the second part, we observe that the cohomological unit is a linear combination of innermost generators. Indeed, if the highest order term on the cylindrical end $Z_{\alpha,r}^+$ is a multiple of $x_\alpha^i x_\alpha^e$ with $i>0$, then the product of the unit with $x_\alpha^j x_\alpha^e$ contains a multiple of $x_\alpha^{i+j} x_\alpha^e$, which is impossible. Then, the first part implies that the coefficient of every innermost generator is $1$.
\end{proof}

%Now, we look at non-innermost generators. Corollary~\ref{cor:prod-structure-basic}(1) implies that the generator $x_\alpha^i x_\alpha^e$ can be expressed as a sum of the $i$-th power of $x_\alpha^1 x_\alpha^e$ and the lower order terms. However, the stronger version of the theorem from Conjecture~\ref{conj:disk-bound-arg-triangle} implies that the lower order terms do not exist for some cases.

%For the remaining part of this section, we assume that Conjecture~\ref{conj:disk-bound-arg-triangle} is true.

%\begin{lemma}\label{lem:xi-is-power-of-x1-ver1}
%    Let $a_i$ and $b_i$ be scalars in Proposition~\ref{prop:distinguished-basis-of-Hw0}. Then, for $i < 0$, we have
%    \begin{itemize}
%        \item $a_i = a_1^i$, and $x^i x^e + b_i y^0 y^e$
%    \end{itemize}
%\end{lemma}

%%%%%%%%%%%%%%%%%333

\subsection{HMS for the tropical Lagrangians}\label{subsect:basis-functions}

Now, we look at the non-innermost generators. The argument we use here is based on the idea of family Floer homology that provides a homological mirror symmetry statement for SYZ fibrations \cite{Abo2}. Starting from the observation that each fiber of the SYZ fibration $F_p \subseteq X$ is homologically mirror to a skyscraper sheaf on a point on its mirror, it uses the Floer cohomology $CF^\bullet (L, F_p)$ as a building block to construct a mirror space from a Lagrangian submanifold $L \subseteq X$.

As a part of Assumption~\ref{ass:mono-adm-non-exact}, we further assume the following.
\begin{assumption}\label{ass:fib-torus-is-skyscraper}
    We assume that each fiber torus in $(\bC^*)^2$ is homologically mirror to a skyscraper sheaf on a point on its mirror in the analogous version of Theorem~\ref{thm:mono-adm-mirror-symmetry}.
\end{assumption}
Since we are working with Novikov coefficients, the mirror skyscraper sheaf is of the form $\mathcal{O}_\rho$ for a point $\rho \in (\Lambda^*)^2$. We denote a fiber torus with a local system corresponding to the sheaf $\mathcal{O}_\rho$ by $F_\rho$. Under the assumption, the module structure
\begin{equation}\label{eqn:module-structure-fib-torus}
    HW^\bullet(L_f, L_f) \otimes HF^\bullet(F_\rho, L_f) \to HF^\bullet(F_\rho, L_f)
\end{equation}
computes the evaluation at each point for each element in the wrapped Floer cohomology. To be specific, if $x \in HW^\bullet(L_q, L_q)$ corresponds to a function $g \in \bK[x_1^\pm , x_2^\pm] / f_Q$ on $D_Q$, then 
\begin{equation}\label{eqn:fiber-torus-x*e-is-f(p)e}
    x \cdot e = g(\rho)e
\end{equation}
for any generator $e \in HF^\bullet(F_\rho, L_q)$. Analogous results to simpler cases lead to the following observation.

\begin{proposition}\label{prop:val-is-2pip}
    If $F_{\rho}$ is the fiber torus on $(p_1, p_2) \in \bR^2$, then $\val(\rho) = (-2\pi p_1 , -2\pi p_2)$.
\end{proposition}

\begin{proof}
    This can be proved by using product cylinders $\bR \times S^1 \subseteq \bC^* \times \bC^*$ intersecting with $F_{\rho}$. We can easily see that the generator that corresponds to the function $z_1$ (on the cylinder $\{z_2 = c\}$) bounds a triangle when computing the module structure (\ref{eqn:module-structure-fib-torus}) whose weight is $T^{-2\pi p_1}$ after rescaling the generators. This shows that $\val(\rho_1) = -2\pi p_1$, and similarly the cylinder $\{z_1 = c\}$ shows that $\val(\rho_2) = -2\pi p_2$.
\end{proof}

The negative sign is essential for the surface on the mirror to have the same tropical skeleton as $L_q$. This is because each ray in the tropical skeleton corresponds to the locus where two terms in the Laurent polynomial attain the lowest valuation. On the other hand, while as a tropical skeleton of a Lagrangian, it corresponds to the region where two terms in the polynomial have the highest magnitude. 

\begin{remark}\label{rem:fiber-torus-support}
    For a Lagrangian $L$ in $(\bC^*)^2$, or in a SYZ fibration in general, the valuation support $\text{Supp}(L)$ of $L$ is the image of the set of $\rho$ with $HF^\bullet (F_\rho, L) \neq 0$ under the valuation map $\val: (\Lambda^*)^2 \to \bR^2$. On the mirror side, we can show that the homology group between the skyscraper sheaf on $\rho$ and a structure sheaf $\mathcal{O}_{D_F}$ for a divisor $D_F$ defined by a function $F$ is nonzero only when $F(\rho)=0$. Therefore, the calculation of the valuation support gives possible candidates for the mirror. 
    
    For the case of $D_q$, we can show that the corresponding divisor on the mirror is defined by the equation (\ref{eqn:fq-novikov-coeff}) by computing the valuation support at infinity. Based on the obvious observation that $HF^\bullet (F_\rho, L)$ should be zero when $F_\rho \cap L = \emptyset$, we can show that the Floer cohomology between $F_\rho$ and $\widetilde{L}_{q, R}^k$ is zero unless $\val(\rho) = (-2\pi p_1, -2\pi p_2)$ lies on a compact subset of $\bR^2$ or one of the cylindrical ends of $L_q$. In particular, since $L_q$ has cylindrical ends corresponding to $p_1 = \log\abs{q},\, p_2 \to \infty$ and $p_2 = \log\abs{q},\, p_1 \to \infty$, the coefficient $Q$ in (\ref{eqn:fq-novikov-coeff}) has the term $T^{-2 \pi \log \abs{q}}$. Also, if $\val(\rho)$ is on the cylindrical end $Z_{\alpha,r}^+$, then the computation of the Floer cohomology is reduced to the case of two $S^1$ in the $(p_\alpha^\perp , \theta_\alpha^\perp)$-plane. There may exist strips other than two small strips between these circles, but they pass through the interior region, and their Novikov valuation diverges to infinity as $p_\alpha \to \infty$. Hence, by comparing the holonomy between two strips (of small area) between these circles, we can conclude that $L_q$ is mirror to $\mathcal{O}_{D_Q}$.
\end{remark}

The computation of the module structure (\ref{eqn:module-structure-fib-torus}) involves the counting of triangles bounded by two Lagrangians of the form $\widetilde{L}_{q, R}^k$ and a fiber torus. However, we have the argument constraints similar to Theorem~\ref{thm:disk-bound-arg-triangle}, which completely determines the poles and their order of corresponding functions.

Let $f$ be a Laurent polynomial satisfying Assumption~\ref{ass:top-bound-for-boundary}, and $Z_{\alpha,r}^+$ be one of the cylindrical end of $L_f$. We consider a fiber torus $F_p$ with $p \in Z_{\alpha, r}^+ \cap \{p_\alpha > R^2 +4R\}$, so that the torus lies outside of all the cylindrical generators. After fixing the collection $\{\varphi_\alpha\} \subseteq \bR$, we perturb the part of the torus in the $(p_\alpha^\perp, \theta_\alpha^\perp)$-plane to be the graph of the Morse function $\mu_{\varphi_\alpha}$ in (\ref{eqn:mu-phi}), so that $\theta_\alpha^\perp$-value of one of the intersection points with $\widetilde{L}_{f, R}^{k_1}$ lie close to $\varphi_\alpha$. We call this generator $e \in HF^\bullet(F_\rho, L_f)$.

We count the following holomorphic maps in the module structure (\ref{eqn:module-structure-fib-torus}).
\begin{notation}\label{not:map-u-fibertorus}
    Let $S$ be a holomorphic disk with three boundary points $\zeta_0, \zeta_1, \zeta_2$ in counterclockwise order, and let $u: S \to (\bC^*)^2$ be a holomorphic map bounded by Lagrangians $\widetilde{L}_{f, R}^{k_0}$, $\widetilde{L}_{f, R}^{k_1}$, and $F_\rho$, where $k_0 < k_1$ and the $p_\alpha$-coordinates of $\val({\rho})$ is bigger than $R^2 + 4R$. Also, assume that the intersection points $u(\zeta_0)$ and $u(\zeta_2)$ correspond to the generator $e$, and $u(\zeta_1)$ is a degree 0 generator $x_\beta^i x_\beta^e$.
\end{notation}

\begin{definition}
    For the map $u: S \to (\bC^*)^2$ in Notation~\ref{not:map-u-fibertorus}, the \textit{$x_\alpha$-exponent} of $u$ is the integer $j$, such that on the lift of $u$ to the universal cover of $(\bC^*)^2$ the distance between the corresponding lifts of $\widetilde{L}_{f, R}^{k_0}$ and  $\widetilde{L}_{f, R}^{k_1}$ is separated in the $\theta_\alpha$-direction by $2j\pi + O(R^{-1})$ on the region $\{p_\alpha < R^2 + 2R\}$. The sign is positive when $\widetilde{L}_{f, R}^{k_1}$ has bigger $\theta_\alpha$-value.
\end{definition}

The only type of map $u$ contained in the cylindrical region is the triangle with the third vertex $x_\alpha^j x_\alpha^e$ for some $j$ (See the left picture in Figure~\ref{fig:arg-constraints-fibertorus}). In this case, the $x_\alpha$-exponent of $u$ is equal to $j$, and after rescaling the generators, the weight of this triangle is equal to $T^{2j\pi p_\alpha} \xi_{\alpha}^j$, where $p_\alpha$ is the $p_\alpha$-component of $\val(\rho)$, and $\xi_\alpha$ is the holonomy of the local system on the fiber torus around the $\theta_\alpha$-circle. In particular, on each cylindrical end of $L_q$, this corresponds to the monomials $z_1^{\pm j}$ and $z_2^{\pm j}$, respectively.

\begin{theorem}[Argument Constraints, with Fiber Torus]\label{thm:disk-bound-arg-fibertorus}
    Let $f$ be a Laurent polynomial satisfying Assumption~\ref{ass:top-bound-for-boundary}, and $Z_{\alpha,r}^+$ be one of the cylindrical end. Then, for the generic choice of $\{\varphi_\alpha\} \subseteq \bR$, The following is true:
    \begin{changemargin}{}{}
        For any positive integer $k$, there exists $R_0 > 0$ such that for any such holomorphic map $u:S \to (\bC^*)^2$ in Notation~\ref{not:map-u-fibertorus} bounded by Lagrangians $\widetilde{L}_{f, R}^{k_s}$ with $R > R_0$ and $k_s \le k$, if the image of $u$ is not contained in the cylindrical region, then the $x_\alpha$-exponent of $u$ is nonpositive.
    \end{changemargin}
\end{theorem}

\begin{proof}
    \vspace{0.3em}
    \begin{figure}[ht]
        \centering
        \includegraphics[width=0.8\textwidth]{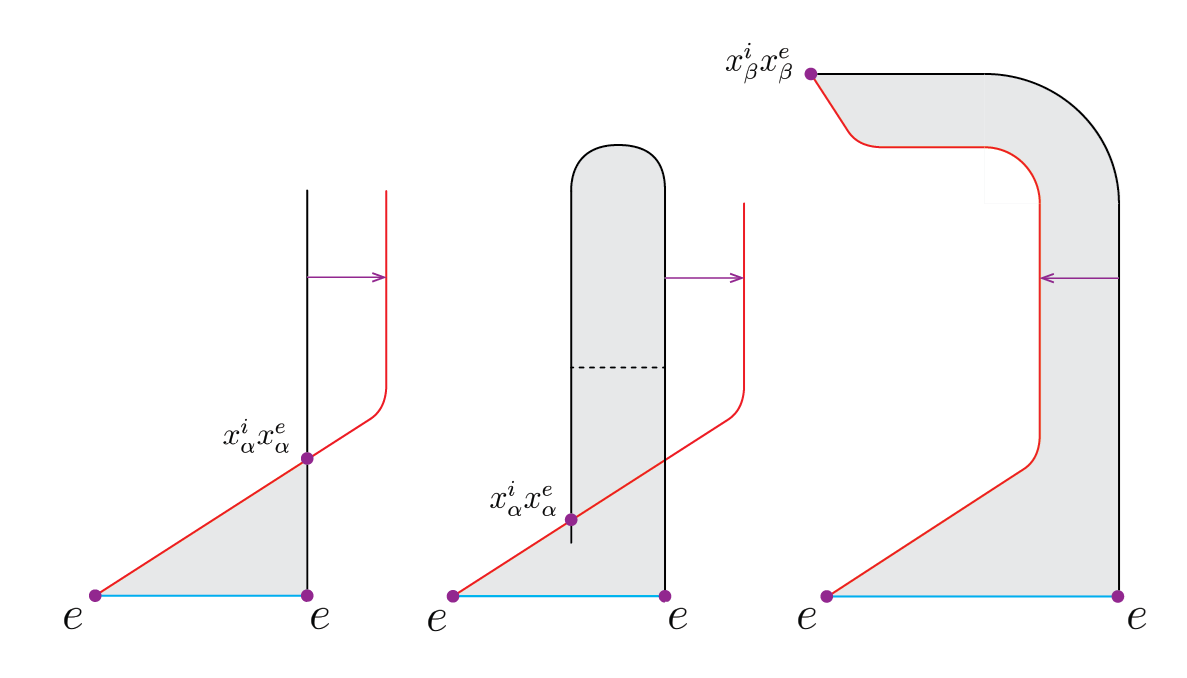}
        \caption{Three different types of maps in Notation~\ref{not:map-u-fibertorus}. Each arrow indicates the part corresponds to its $x_\alpha$-exponent.}
        \label{fig:arg-constraints-fibertorus}
    \end{figure}
    The proof is the same as Theorem~\ref{thm:disk-bound-arg-triangle}. Figure~\ref{fig:arg-constraints-fibertorus} shows three possible types of such maps. The second type, where the $x_\alpha$-exponent is nonnegative, does not exist by the argument constraint, while maps in the third type have nonpositive $x_\alpha$-exponents.
\end{proof}

Now, as we push $\rho$ to infinity in the $p_\alpha$ direction, the exponent of the (rescaled) weight of each disk in the above theorem increases linearly with factor $2j\pi$. For example, if we consider the cylindrical end of $L_q$ in the direction $\alpha = (-1,0)$, then the theorem implies the following.
\begin{itemize}
    \item Each generator $x^i x^e$ with $i >0$ has a pole at $z_1 = 0$ of order $i$ with coefficient $1$ and no lower negative power of $z_1$.
    \item Generators on the other leg do not have a pole at $z_1 = 0$.
\end{itemize}
Similarly, we can determine the order of poles on each of the four punctures on $D_q$. Since the only meromorphic functions that do not have poles on any of the punctures are constant functions, we can determine the corresponding functions of each generator in Proposition~\ref{prop:distinguished-basis-of-Hw0} up to constant.

\begin{theorem}[HMS of the Lagrangian $L_q$]\label{thm:HMS-for-Lq}
    For $q \in \bC \setminus \bR$, there exist constants $A_i$ and $B_i$ for $i \in \bZ_{>0}$ such that the generators of $HW^0 (L_q, L_q)$ correspond to the following functions on $D_Q$ under the isomorphism (\ref{eqn:HW0-iso-to-qtt}).
    \begin{itemize}
        \item $x^i x^e + b_i y^0 y^e$,\quad for $i>0$,\quad corresponds to $z_1^i + B_i$.
        \item $x^i x^e + b_i y^0 y^e$,\quad for $i<0$,\quad corresponds to $z_1^i$.
        \item $y^i y^e + a_i x^0 x^e$,\quad for $i>0$,\quad corresponds to $z_2^i + A_i$.
        \item $y^i y^e + a_i x^0 x^e$,\quad for $i<0$,\quad corresponds to $z_2^i$.
        \item $x^0 x^e + y^0 y^e$,\quad the cohomological unit,\quad corresponds to $1$.
    \end{itemize}
\end{theorem}

\begin{proof}
    The only thing left to show is that there are no constant terms for the generators on the cylindrical end with an innermost generator. To prove that, we note that the squares of these generators are
    $$ (x^{i}x^e + b_{i} y^0 y^e)^2 = x^{2i}x^e + \sum_{i < j \le 0} c_j x^j x^e + b_i^2 y^0 y^e$$
    for some coefficients $c_j$ by Corollary~\ref{cor:prod-structure-basic}. The left-hand side consists of the term $B_i z_1^i$, while the right-hand side only consists of the terms with the lower power of $z_1^{-1}$. This concludes that $B_i = 0$ for $i <0$. The same is true for $A_i$.
\end{proof}

\subsection{Further Example: Pair of Pants}\label{subsect:pair-of-pants}
In this section, we discuss an example where Assumption~\ref{ass:cyl-end-is-not-diagonal} is dropped, i.e. the case when there are cylindrical ends in a diagonal direction. The example we consider is \textit{the complex pair of pants} $D_{\text{pants}}$ and its HyperKähler rotation, \text{the (tropical) pair of pants} $L_\PoP$, defined by the Laurent polynomial
$$f_\PoP(z_1, z_2) := 1 - z_1 - z_2.$$
This Lagrangian submanifold of $(\bC^*)^2$ has three cylindrical ends. Two of them are in the negative $p_1$ and $p_2$ direction (hence satisfy $\abs{\alpha}=1$), and the other end is $Z_{\alpha, r}^+$ with $\alpha = (1, -1)$ and $r = -1$. The corresponding Lagrangian is equipped with the following local system and brane structure.

\begin{notation}\label{not:PoP-brane-structure}
    We equip $L_\PoP$ with a trivial local system, and a brane structure such that when restricted to the $S^1$ factor of each cylindrical end, only the one on the diagonal end has a nontrivial spin structure.
\end{notation}

As in Remark~\ref{rem:fiber-torus-support}, one can find the suitable local system and the brane structure by observing the support of $L$. In particular, for a cylindrical end determined by two terms $c_1 z^{\alpha_1} = \pm c_2 z^{\alpha_2}$ of the Laurent polynomial $f$, the holonomy (from the local system) around the $S^1$ factor is determined by the unitary part of $c_1 c_2^{-1}$, while the brane structure is determined by the sign. Interpreting the equation as $c_1 z^{\alpha_1} = \mp (-c_2) z^{\alpha_2}$ instead is equivalent to using the brane structure with the opposite orientation as well as multiplying the holonomy from the local system by $-1$; see Section 6 in \cite{Cho1}.

Even though the tropical pair of pants $L_{\text{pants}}$ does not satisfy Assumption~\ref{ass:cyl-end-is-not-diagonal}, it satisfies the other two assumptions required to define the wrapped Floer cohomology as in Section~\ref{sect:floer-theory} and Definition~\ref{def:wrapped-Floer-coh}. The smoothness of the tropical polynomial (Assumption~\ref{ass:trop-ftn}) is obvious from the definition. Also, the finiteness assumption (Assumption~\ref{ass:top-bound-for-boundary}) can be proved by a topological argument. Indeed, the image of the boundary map
$$ H_2 ((\bC^*)^2, L_f \cup \{p_\alpha \ge R, \abs{p_\alpha^\perp - \log\abs{r}} \le 1 \}) \xrightarrow{\delta} H_1 (\{p_\alpha \ge R, \abs{p_\alpha^\perp - \log\abs{r}} \le 1 \}) $$
is 1-dimensional, hence there is at most one homology class with a given $ \left [S_{\alpha^\perp}^1 \right ] $-coefficient.

The main obstacle in this case is that there are exactly $\abs{\alpha}^2$ different points labeled as $x_\alpha^j x_\alpha^e$ (and $x_\alpha^j x_\alpha^f$). From the discussion in Remark~\ref{rem:alpha-bigger}, the argument coordinates $(\theta_1 , \theta_2)$ of each point are given by
\begin{equation}\label{eqn:theta_a-and-theta_perp}
    \begin{cases}
        \theta_\alpha = c_1 ,\\ \abs{\alpha}^2 \theta_\alpha^\perp = c_2 + 2k\pi\quad (k=0,1,\cdots,\abs{\alpha}^2 - 1)
    \end{cases}
\end{equation}
for some fixed constants $c_1, c_2 \in \bR$. If we lift the whole configuration over the universal cover of $(\bC^*)^2$ and divide into the $(p_\alpha, \theta_\alpha)$-plane and the $(p_\alpha^\perp, \theta_\alpha^\perp)$-plane, the Morse function on the $S^1$ factor of the cylinder has $\abs{\alpha}^2$ maxima (and minima), each of which corresponds to a degree 1 generator $x_\alpha^j x_\alpha^f$ (and a degree 0 generator $x_\alpha^j x_\alpha^e$, respectively). For simplicity, we use the following notation instead of naming them individually.

\vspace{0.3em}
\begin{figure}[ht]
    \centering
    \includegraphics[width=0.8\textwidth]{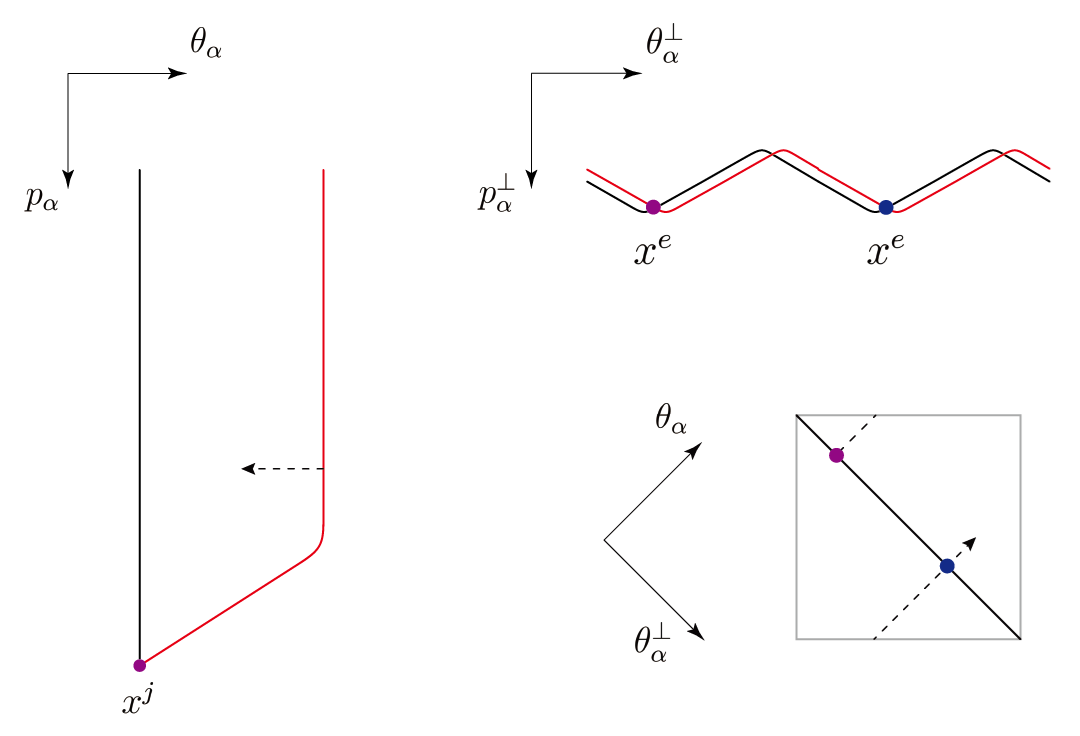}
    \caption{Intersection points of the type $x_\alpha^j x_\alpha^e$, for the case $\abs{\alpha}^2=2$. The bottom right part of the figure shows the $(\theta_1, \theta_2)$ coordinates of two such points. The dashed line shows the direction of the wrapping.}
    \label{fig:alpha-bigger}
\end{figure}
\vspace{0.3em}

\begin{notation}[Notation~\ref{not:def-xjxe}, generalized]\label{not:def-xjxe-alpha-bigger}
    For a Laurent polynomial $f$ and a cylindrical end $Z_{\alpha, r}^+$ of the Lagrangian $L_f$, we call the intersection points between $\widetilde{L}_{f,R}$ and $\widetilde{L}_{f,R}^k$
    \begin{itemize}
        \item of the type $x_\alpha^{j}x_\alpha^{e}$ when $h(p_\alpha) = 2j\pi$ and $
        \abs{\alpha}^2 \theta_\alpha^\perp \equiv \varphi_\alpha + \frac{1}{2} R^{-1} a_\alpha^\perp \pmod{2\pi}$
        \item of the type $x_\alpha^{j}x_\alpha^{f}$ when $h(p_\alpha) = 2j\pi$ and $
        \abs{\alpha}^2 \theta_\alpha^\perp \equiv \varphi_\alpha + \pi + \frac{1}{2} R^{-1} a_\alpha^\perp \pmod{2\pi}$
    \end{itemize}
    where $0 \le j \le k$ are integers, $R>0$ is sufficiently large and $h$ is the function in (\ref{eqn:flow-of-HfR}).
\end{notation}

\begin{remark}\label{rem:indexing-xjxe-diag}
    Each point of the type $x_\alpha^j x_\alpha^e$ on the Lagrangian $\widetilde{L}_{f,R}$ can be indexed by the variable $k$ in the expression (\ref{eqn:theta_a-and-theta_perp}) after fixing the real lift $c_1$ of $\arg r$. Recall that we choose $\varphi_\alpha$ to be a real number in Definition~\ref{def:Lf-tilde}, so one can set $c_2 = \varphi_\alpha + \frac{1}{2} R^{-1} a_\alpha^\perp$. Indexing the generators of the same type in $CF^\bullet (\widetilde{L}_{f,R}^{k_0}, \widetilde{L}_{f,R}^{k_1})$ is, however, more complicated since one need to consider the lifts of both Lagrangians. Namely, two such Lagrangians intersecting at a point of the type $x_\alpha^j x_\alpha^e$ are separated by $2j\pi$ in $\theta_\alpha$-direction on the universal cover, so it represents different points (in $\widetilde{L}_{f,R}$) on each lift unless $j$ is the multiple of $\abs{\alpha}^2$. For example, when $\abs{\alpha}^2 = 2$ (see Figure~\ref{fig:alpha-bigger}), there are two different generators of the type $x_\alpha^e$ along the $S^1$ part of the cylinder, and wrapping once makes one such generator on one lift coincide with the other generator on the other lift.
    %If we consider the points before perturbation (on $\widetilde{L}_{f,R}$), the $\theta_\alpha^\perp$-coordinates are the same and the $\theta_\alpha$-coordinates differ by $2j\pi$.
\end{remark}

Due to the remark above, it is more reasonable to fix $\theta_\alpha^\perp$ and vary $\theta_\alpha$ when indexing the generators in $Z_{\alpha, r}^+$. This gives a better description of the triangles counted in the product structure and the module structure (\ref{eqn:module-structure-fib-torus}), since those contained in the cylindrical region have their $\theta_\alpha^\perp$-coordinates vary by at most $O(R^{-1})$.

\begin{notation}\label{not:indexing-xjxe-diag}
    Let $Z_{\alpha,r}^+$ be a cylindrical end of $L_f$, and fix a real lift of $\arg{r}$. Then, for any $k_0, k_1$ and sufficiently large $R>0$, we denote an intersection point $X \in \widetilde{L}_{f,R}^{k_0} \cap \widetilde{L}_{f,R}^{k_1}$ on the cylindrical region $Z_{\alpha,r}^+$ by $x_\alpha^{j_0, j_1} x_\alpha^e$ if there is a lift $\widetilde{X}$ of $X$ to the universal cover of $(\bC^*)^2$ such that
    \begin{itemize}
        \item The $\theta_\alpha^\perp$-coordinate of $\widetilde{X}$ is $\varphi_\alpha + O(R^{-1})$;
        \item The lift of $\widetilde{L}_{f,R}^{k_0}$ and $\widetilde{L}_{f,R}^{k_1}$ containing $\widetilde{X}$ have $\theta_\alpha$-values (before wrapping, i.e. on the region $p_\alpha \le R^2+3R$) equal to $\arg{r} + 2j_0\pi +O(R^{-1})$ and $\arg{r} + 2j_1\pi +O(R^{-1})$, respectively.
    \end{itemize}
\end{notation}

Remark~\ref{rem:indexing-xjxe-diag} explains that the point $x_\alpha^{j_0, j_1} x_\alpha^e$ is of the type $x_\alpha^{j_1 - j_0} x_\alpha^e$. Also, $x_\alpha^{i_0, i_1} x_\alpha^e$ and $x_\alpha^{j_0, j_1} x_\alpha^e$ represent the same point if and only if 
$i_1 - i_0 = j_1 - j_0$ and $ i_0 \equiv j_0 \pmod{\abs{\alpha}^2}$. We will use these notations when finding the precise expression for elements in the wrapped Floer cohomology and their corresponding functions.

%The next remark is about the local system on the diagonal cylindrical end.

\begin{remark}\label{rem:diag-end-local-system}
    When computing the weight of a holomorphic disk $u$ coming from the holonomy along $\round u$, one needs to fix an identification between two fibers for each intersection point of the Lagrangians. For the generators of the type $x_\alpha^j x_\alpha^e$ where $j$ is the multiple of $\abs{\alpha}^2$, such identification can be chosen canonically because the generator corresponds to the same point in both Lagrangians (before the Hamiltonian perturbation). For a general $j$, we fix an identification between fibres over the points $(\theta_\alpha, \theta_\alpha^\perp) = (\arg{r}, \varphi_\alpha + k\pi)$ with different $k$. This assigns a unitary number $\text{hol}(\gamma)$ to each ``$\frac{1}{\abs{\alpha}^2}$-turn'' $\gamma$ compatible with the holonomy along the loops. If the holonomy around the $S^1$ factor of the cylindrical end is $\xi$, we pick one of its $\abs{\alpha}^2$-th roots $\xi_\alpha = \xi^{1/\abs{\alpha}^2}$ and set the holonomy around both each $\frac{1}{\abs{\alpha}^2}$-turn to be $\xi_\alpha$, making it consistent with the translation in $\theta_\alpha$-direction.
\end{remark}

Most of the arguments in previous sections can be extended and proved if we replace the terms $x_\alpha^j x_\alpha^e$ and $x_\alpha^j x_\alpha^f$ by ``a point of the type $x_\alpha^j x_\alpha^e$'' and ``a point of the type $x_\alpha^j x_\alpha^f$'', respectively. In particular, the argument constraints for strips (Theorem~\ref{thm:disk-bound-arg-strip}) show that three types of strips contribute to the differential of a point of the type $x_\alpha^i x_\alpha^e$.

\begin{enumerate}
    \item Holomorphic strips contained in the cylindrical region $\{p_\alpha \ge R^2 + 3R\}$. These are the strips from a point of the type $x_\alpha^i x_\alpha^e$ to the two neighboring points of the type $x_\alpha^i x_\alpha^f$.
    \item Holomorphic strips from $x_\alpha^i x_\alpha^e$ to $x_\alpha^j x_\alpha^f$ with $j < i$.
    \item Holomorphic strips from $x_\alpha^i x_\alpha^e$ to an interior point.
\end{enumerate}

The first type of strips contributes to the highest order terms (i.e. the terms with the highest $x_\alpha$-exponent) in the differential. All such strips have the same symplectic area, and under a suitable spin structure, they cancel out with each other when we take a differential of the sum of the points of the type $x_\alpha^i x_\alpha^e$. Also, under the identification given in Remark~\ref{rem:diag-end-local-system}, the weights of the strips from the local system are all equal to $1$, as the holonomy along each half-turn is identical. Therefore, if a degree 0 element in the Floer chain complex has zero differential, it should include all the highest order generators. 

\begin{lemma}\label{lem:Pop-highest-is-sum}
    Suppose that the Laurent polynomial $f$ and a cylindrical end $Z_{\alpha,r}^+$ of $L_f$ satisfy all the assumptions for Theorem~\ref{thm:disk-bound-arg-strip}. Then, the term with the highest $x_\alpha$-exponent of any closed degree 0 element of the wrapped Floer chain complex is a scalar multiple of the sum of all points of the type $x_\alpha^j x_\alpha^e$.
\end{lemma}
$\hfill\square$

Consider the filtration $\left\{ W_{f,k} \right\}_k$ of $HW^\bullet (L_f, L_f)$ given by the subspace $W_{f,k}$ consisting of closed elements generated by $x_\alpha^j x_\alpha^e$ with $j \le k$. There is a natural isomorphism
$$ W_{f,k} \xleftarrow{\sim} HF^0 (\widetilde{L}_{f,R}^k , \widetilde{L}_{f,R}^0)$$ 
for sufficiently large $R>0$ by Proposition~\ref{prop:htpy-mtd-q-iso} and Proposition~\ref{prop:HW-is-limit-HfR}. Also, the lemma above implies that the quotient $W_{f, k} / W_{f, k-1}$ is contained in a space generated by the sum of the points of the type $x_\alpha^k x_\alpha^e$ for each $\alpha$. In other words, increasing the wrapping number $k$ by 1 would still add at most one generator on each leg.

Returning to $L_\PoP$, we choose $(a_1, a_2)$ in Definition~\ref{def:Hf-tilde} to be a pair of distinct negative numbers, so that the innermost generators $x_\alpha^0 x_\alpha^e$ only appear in the diagonal cylindrical end. Lemma~\ref{lem:Riemann-translation-intersection} implies that there is $2g(L_\PoP) + b(L_\PoP) - 2 = 0 + 3 - 2 = 1$ intersection point in the interior region. Also, the argument in Section~\ref{subsect:distinguished-basis} relates the Floer theory of $L_\PoP$ to the corresponding tropical Lagrangian constructed in \cite{Hic1}. The Lagrangians $\widetilde{L}_{f}^k$ are now mirror to coherent sheaves on the toric variety $\mathbb{P}^2$, and wrapping once from $\widetilde{L}_{f}^{k-1}$ to $\widetilde{L}_{f}^k$ is equivalent to the twisting $- \otimes \mathcal{L}(D)$ with its canonical toric divisor $D$.

\begin{theorem}[Theorem~\ref{thm:HW0-cohomology-group}, for Pair of Pants]
    If we set $f = f_\PoP$ and $\bK = \Lambda$, the following is true for large enough $R > 0$.
    \begin{itemize}
        \item Under the quasi-isomorphism (\ref{eqn:mono-adm-mirror-symmetry}), the Lagrangian $L_{f,R}$ is mirror to a structure sheaf $\mathcal{O}_{D_{\text{pants}}}$. In particular, there exists a ring isomorphism
        \begin{equation}\label{eqn:PoP-HW0-iso-to-qtt}
            HW^0 (L_\PoP, L_\PoP) \cong \mathbb{K}[z_1^{\pm}, z_2^{\pm}] / f.
        \end{equation}
        \item For any integer $k$, the Floer cohomology group $HF^0 (\widetilde{L}_{f,R}^k , \widetilde{L}_{f,R}^0)$ is isomorphic to the subgroup of $\mathbb{K}[z_1^{\pm}, z_2^{\pm}] / f$ generated by the monomials $z_1^{m_1} z_2^{m_2}$ with
        $$ m_1 \ge -k, \quad m_2 \ge -k, \quad m_1 + m_2 \le k. $$
    \end{itemize}
\end{theorem}

If we denote the subgroup in the theorem by $W_{f,k}'$, one can check that $W_{f,k}' / W_{f,k-1}'$ is 3-dimensional and generated by $z_1^{-k}, z_2^{-k}$, and $z_1^k$ (or any monomial of the form $z_1^j z_2^{k-j}$ with $0 \le j \le k$). Comparing this with the filtration $\{W_{f,k}\}$, this is the case when the dimension of $W_{f,k} / W_{f,k-1}$ is maximal and generated by the sum of points of the type $x_\alpha^k x_\alpha^e$ for each $\alpha$.

\begin{proposition} [Distinguished Basis of $HW^0(L_\PoP, L_\PoP)$] \label{prop:distinguished-basis-of-Hw0-PoP}
    For $f = f_\PoP$, the following is true for $HF^0 (\widetilde{L}_{f,R}^k , \widetilde{L}_{f,R}^0)$ and $HW^0 (L_f, L_f)$ for sufficiently large $R > 0$.

    \begin{enumerate}
        \item The sum of points of the type $x_{\alpha_0}^0 x_{\alpha_0}^e$, where $\alpha_0 = (1,-1)$ is the orthogonal direction of the diagonal cylindrical end, is the cohomological unit of $HW^0 (L_f, L_f)$.
        \item \label{item:prop:distinguished-basis-of-Hw0-PoP} For each cylindrical end $Z_{\alpha, r}^+$ of $L_f$ and a positive integer $j$, there exists a closed element $X_\alpha^j$ of $HF^0 (\widetilde{L}_{f,R}^k , \widetilde{L}_{f,R}^0)$ such that
        \begin{itemize}
            \item $X_\alpha^j$ is generated by the points of the type $x_\alpha^i x_\alpha^e$ with $i \le j$, and
            \item the term in $X_\alpha^j$ with the highest $x_\alpha$-exponent is equal to the sum of points of the type $x_\alpha^j x_\alpha^e$.
        \end{itemize}
        \item For any choice of closed elements $X_\alpha^j$ in (\ref{item:prop:distinguished-basis-of-Hw0-PoP}) for each positive integer $j$, the cohomological unit and the elements $X_\alpha^j$ form a basis of $HW^0(L_f, L_f)$.
    \end{enumerate}

\end{proposition}

\begin{proof}
    We only need to show that $X_\alpha^j$ can be formed by points on the same cylindrical end, and we show this by induction on $j$. Assume that the conclusion is true for every $i < j$, and choose any closed element $X$ such that the highest $x_\alpha$-exponent is equal to the sum of points of the type $x_\alpha^j x_\alpha^e$, and the exponents of the other two ends are less than $j$. Now, consider the decomposition
    $$ X = X_\alpha + X_{\alpha_1} + X_{\alpha_2} $$
    where the terms on the right are generated by the points on $Z_{\alpha, r}^+$ and the other two legs $Z_{\alpha_1, r_1}^+$, $Z_{\alpha_2, r_2}^+$, respectively. Lemma~\ref{lem:Pop-highest-is-sum} implies that the highest term of $X_{\alpha_1}$ is a scalar multiple of the sum of all points of the type $x_{\alpha_1}^i x_{\alpha_1}^e$ with $i < j$, so we can subtract a multiple of $X_{\alpha_1}^i$ from $X_{\alpha_1}$ to reduce the highest exponent. Repeating this process on $X_{\alpha_1}$ and $X_{\alpha_2}$ results in only $X_\alpha$ being left, which concludes that $X_\alpha$ is closed and it is the element $X_\alpha^j$ we want.
\end{proof}

%Unlike the case of $L_q$ (Proposition~\ref{prop:distinguished-basis-of-Hw0}), the generators $X_\alpha^j$ are not uniquely determined, but rather determined up to generators with lower $x_\alpha$-exponent. However, for the two cylindrical ends with $\abs{\alpha}=1$, one can follow the arguments in the previous sections and get similar results.

The generators $X_\alpha^j$ are not uniquely determined but are determined up to generators with a lower $x_\alpha$-exponent. However, for the two cylindrical ends with $\abs{\alpha}=1$, one can follow the arguments in the previous sections and get similar results.

\begin{corollary}[HMS of the Pair of Pants, Part 1]\label{cor:HMS-for-LPoP-nondiag}
    For the cylindrical ends $Z_{\beta, s}^+$ of $L_\PoP$ in the negative $x$- and $y$- direction,
    \begin{enumerate}
        \item The generator $x_\beta^i x_\beta^e$ is closed for $i > 0$.
        \item Under the isomorphism (\ref{eqn:PoP-HW0-iso-to-qtt}), $x_\beta^i x_\beta^e$ corresponds to the function $z_1^{-i}$ and $z_2^{-i}$, respectively.
    \end{enumerate}
\end{corollary}

\begin{proof}
    The proof is the same as Proposition~\ref{prop:distinguished-basis-of-Hw0} and Theorem~\ref{thm:HMS-for-Lq}. Proposition~\ref{prop:distinguished-basis-of-Hw0-PoP} shows that $X_\beta^1 = x_\beta^1 x_\beta^e$ is closed, and we can prove that $x_\beta^j x_\beta^e$ is closed by induction on $i$ and the fact that each $X_\beta^j$ is closed. One needs to use the generalized version of Theorem~\ref{thm:disk-bound-arg-fibertorus} to show that the corresponding function does not have a pole at the puncture in the diagonal cylindrical end. The constant term is dropped from the proof of Theorem~\ref{thm:HMS-for-Lq} since neither cylindrical ends contain innermost generators.
\end{proof}

For the remainder of this section, consider only the case when $Z_{\alpha,r}^+$ is the diagonal end of $L_\PoP$. One can find the function corresponding to the generator $X_\alpha^j$ under (\ref{eqn:PoP-HW0-iso-to-qtt}) by observing the module structure (\ref{eqn:module-structure-fib-torus}). Let $Z_{\alpha,r}^+$ be the diagonal cylindrical end of $L_\PoP$. Each fiber torus $F_\rho$ on the cylindrical region is perturbed by the function $\mu(\theta_1, \theta_2) = \mu_{\varphi_\alpha} (\abs{\alpha}^2 \theta_\alpha^\perp)$, so that there are $2 \abs{\alpha}^2 = 4$ intersection points with $\widetilde{L}_{f,R}$ when $\val{\rho} > R^2 + 4R$. Half of these points have degree 0 and are located $O(R^{-1})$-close to $\abs{\alpha}^2$ different points satisfying $\theta_\alpha \equiv \arg r \pmod{2\pi}$ and $\abs{\alpha}^2 \theta_\alpha^\perp \equiv \varphi_\alpha \pmod{2\pi}$. The other half have degree 0 and are located $O(R^{-1})$-close to $\abs{\alpha}^2$ different points satisfying $\theta_\alpha \equiv \arg r \pmod{2\pi}$ and $\abs{\alpha}^2 \theta_\alpha^\perp \equiv \varphi_\alpha \pmod{2\pi}$. Following the notation in the previous section, we call the former \textit{of the type} $e$ and the latter \textit{of the type} $f$. Additionally, let us index the points of the type $e$ for a detailed calculation.
\begin{notation}\label{not:indexing-typee-diag}
    We denote the intersection point $E \in \widetilde{L}_{f,R} \cap F_\rho$ of type $e$ by $e_j$ for an integer $j$ (modulo $\abs{\alpha}^2$) if there is a lift $\widetilde{E}$ of $E$ to the the universal cover of $(\bC^*)^2$ such that
    \begin{itemize}
        \item The $\theta_\alpha^\perp$-coordinate of $\widetilde{E}$ is $\varphi_\alpha + O(R^{-1})$;
        \item The $\theta_\alpha$-coordinate of $\widetilde{E}$ is $\arg{r} + 2j\pi$.
    \end{itemize}
    Here, we use the same real lift of $\arg{r}$ as in Notation~\ref{not:action-rescale}.
\end{notation}

The argument constraint implies that two types of strips contribute to the differential of $CF^\bullet(F_\rho, L_f)$. First, the $S^1$ factors of $L_f$ and $F_\rho$ in the $(p_\alpha^\perp, \theta_\alpha^\perp)$-plane have $2 \abs{\alpha}^2$ strips between them. These are the only strips bounded by these Lagrangians lying inside the cylindrical region, and they are equipped with a spin structure in a way that the differential of the sum of all points of the type $e$, without holonomy, cancels out all the corresponding terms. Also, some strips pass through the interior region, and their Novikov valuation diverges to infinity as $R \to \infty$. Therefore, if we equip the fiber torus with suitable holonomy so that $HF^\bullet(F_\rho, L_f)$ is nonzero, which corresponds to the case when $f(\rho)=0$, then we get the following by observing the terms with the lowest Novikov valuation.

\begin{lemma}\label{lem:generator-of-Frho-Lf}
    Fix a sufficiently large $R>0$ and let $\rho \in (\Lambda^*)^2$ be a point such that $f_\PoP(\rho) = 0$, and that $\val(\rho) = (-2\pi p_1, -2\pi p_2)$ satisfy $p_\alpha > R^2 + 4R$. Then,
    \begin{enumerate}
        \item $HF^\bullet(F_\rho, L_f)$ is isomorphic to $H^\bullet(S^1)$ as a graded algebra.
        \item Any degree 0 generator of $HF^\bullet(F_\rho, L_f)$ is of the form $c_1 e_1 + c_2 e_2$, where $c_1, c_2 \in \bK$ are nonzero elements in the Novikov field with the same lowest valuation term.
    \end{enumerate}
\end{lemma}
$\hfill\square$

Note that to realize the holonomy from the local system as an element in the Novikov field, one needs to fix an identification between the fibers over $e_j$. Choosing a different identification is equivalent to rescaling the generators $e_1$ and $e_2$, changing the coefficients $c_1$ and $c_2$ of the generator $c_1 e_1 + c_2 e_2 \in HF^0 (F_\rho, L_f)$ accordingly.

The analogous version of Theorem~\ref{thm:disk-bound-arg-fibertorus} determines the pole part of the corresponding function at $p_1, p_2 \to \infty$. Here, \textit{the pole part at $p_1, p_2 \to \infty$} of a Laurent polynomial $g(\rho) = g(z_1, z_2)$ is defined as a part generated by the monomials $z_1^{m_1} z_2^{m_2}$ with $m_1 + m_2 > 0$. Similarly, the pole part of the module structure (\ref{eqn:module-structure-fib-torus}) means the pole part of a coefficient $g$ in the relation
$$X_{\alpha}^j \cdot E = g(\rho) E $$
where $E = c_1 e_1 + c_2 e_2$ is the generator of $HF^0 (F_\rho, L_f)$.

On the chain level, the weight of each holomorphic triangle that contributes to the module structure has the same valuation as $T^{-2\pi w p_\alpha}$, where $w$ is the wrapping number (or the ``width'') of the boundary curve (see Figure~\ref{fig:arg-constraints-fibertorus} in the proof of Theorem~\ref{thm:disk-bound-arg-fibertorus}). Since $p_1 = p_2 = p_\alpha$ along the diagonal cylindrical end, $w$ is equal to the order $m_1 + m_2$ of the monomial $z_1^{m_1} z_2^{m_2}$ on the cohomology level. Accordingly, the pole part of the module structure on the chain level is determined by the strips with positive wrapping number $w$.

%The equation (\ref{eqn:fiber-torus-x*e-is-f(p)e}) we used to find the corresponding function on $D_\PoP$ now gives additional information about the coefficients of the element $X_\alpha^j$ itself since the coefficient of $e_1$ and $e_2$ should be multiplied by the same number $f(\rho)$. The generalized version of Theorem~\ref{thm:disk-bound-arg-fibertorus} implies that the pole part of the corresponding functions $f$ is obtained by the triangles lying inside the cylindrical end (see the left picture in Figure~\ref{fig:arg-constraints-fibertorus}).

%Now, consider the point $X = x_\alpha^{j_0, j_1} x_\alpha^e$ of the type $x_\alpha^j x_\alpha^e$ with $j = j_1 - j_0$ (see Notation~\ref{not:indexing-xjxe-diag}). One must also index the intersection points in $\widetilde{L}_{f,R} \cap F_\rho$ of the type $e$ to describe the local system.

%Note that to realize the holonomy from the local system as an element in the Novikov field, one needs to fix an identification between the fibers over $e_j$. Choosing a different identification is equivalent to rescaling the generators $e_1$ and $e_2$, changing the coefficients $c_1$ and $c_2$ of the generator $c_1 e_1 + c_2 e_2 \in HF^0 (F_\rho, L_f)$ accordingly. Now, a careful analysis on these holomorphic triangles yields the following.

\begin{proposition}\label{prop:PoP-module-structure}
    The following holds for the generators on the diagonal cylindrical end $Z_{\alpha, r}^+$ of $L_\PoP$.
    \begin{enumerate}
        \item There are two generators of the type $x_\alpha^j x_\alpha^e$ for any nonnegative integer $j$, which are $x_\alpha^{k, j+k} x_\alpha^e$ for any positive even number or any positive odd number, respectively.
        \item\label{prop-sub:PoP-module-structure-functions} For a fiber torus $F_\rho$ lying on the cylindrical end $Z_{\alpha, r}^+$ with nonzero $HF^\bullet(F_\rho, L_f)$, the pole parts of the module structure (\ref{eqn:module-structure-fib-torus}) for generators of the type $x_\alpha^j x_\alpha^e$ and the type $e$ coming from holomorphic triangles are zero except for the relation
        \begin{equation}\label{eqn:fiber-torus-x*e-is-f(p)e-alpha}
            x_\alpha^{j_0, j_1} x_\alpha^e \cdot e_{j_1} = g^{j_0, j_1}(\rho) e_{j_0}
        \end{equation}
        (modulo the terms with $T^{-2\pi w p_\alpha}$, $w \le 0$) after a suitable rescaling of the generators. Here, the coefficient $g^{j_0, j_1}(\rho)$ is the following monomial on $\rho = (z_1, z_2)$:
        \vspace{0.1in}

        \begin{tabularx}{0.8\textwidth} { >{\centering\arraybackslash}X | >{\centering\arraybackslash}X >{\centering\arraybackslash}X }
             & $g^{\text{even},\,\, \text{even}+k} (\rho)$ & $g^{\text{odd},\,\, \text{odd}+k} (\rho)$ \\
            \hline\hline
            $k = 2l$ & $z_1^l z_2^l$ & $ z_1^l z_2^l$ \\
            $k = 2l+1$ & $z_1^{l+1} z_2^l$ & $- z_1^l z_2^{l+1}$
        \end{tabularx}
        \vspace{0.1in}

        \item\label{prop-sub:PoP-module-structure-generators} Under the same rescaling, $e_1 + e_2$ is the generator of $HF^0(F_\rho, L_\PoP)$.
    \end{enumerate}
\end{proposition}

Similar to the proof of Theorem~\ref{thm:HMS-for-Lq}, one can exclude the existence of holomorphic disks with positive width except for ones lying inside the cylindrical region. For each point $X = x_\alpha^{j_0, j_1} x_\alpha^e$ of type $x_\alpha^j x_\alpha^e$ with $j = j_1 - j_0$, there is exactly one such triangle that has $X$ as an input vertex. After action rescaling (Notation~\ref{not:action-rescale}), the weight of the triangle is equal to $T^{-2j\pi p_\alpha} \xi_\alpha^{j_0, j_1}$, where  $\xi_\alpha^{j_0, j_1}$ is the holonomy of the local system on the fiber torus along the path from $e_{j_0}$ to $e_{j_1}$, fixing $\theta_\alpha^\perp$ and increasing $\theta_\alpha$ by $2j\pi$.

\vspace{0.3em}
\begin{figure}[ht]
    \centering
    \includegraphics[width=0.8\textwidth]{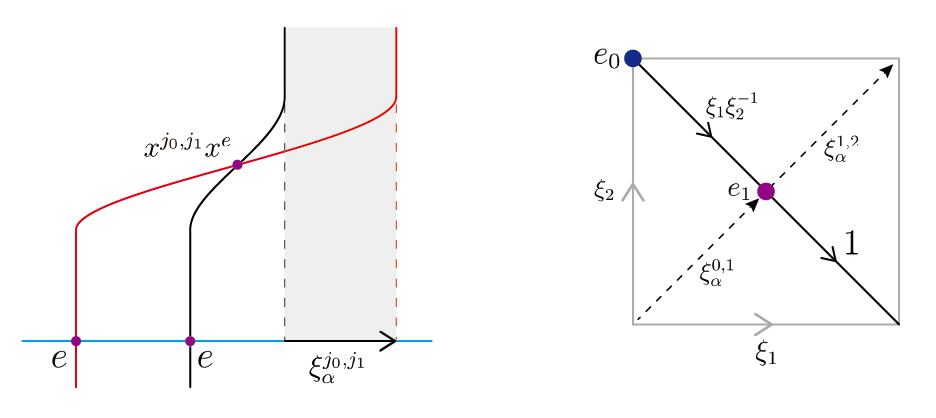}
    \caption{The local system on $F_\rho$. The left figure shows the area and the holonomy counted after action rescaling. The right figure shows the holonomy along several paths on the $S^1$ factor of $L_\PoP$ and the fiber torus. The dashed line represents the direction of the wrapping.}
    \label{fig:alpha-bigger-holonomy}
\end{figure}
\vspace{0.3em}

Next, we express the weight $\xi_\alpha^{j_0, j_1}$ from the local system in terms of the unitary component of $z_1$ and $z_2$. Recall that $\rho$ is of the form
$$ \rho = (z_1, z_2) =  \left( T^{-2\pi p_1} \xi_1,\, T^{-2\pi p_2}\xi_2 \right). $$
$F_\rho$ is the fiber torus over $(p_1, p_2) \in \bR^2$, and unitary elements $\xi_1$ and $\xi_2$ represent the holonomy of the local system on the fiber torus along the $\theta_1$ and $\theta_2$ directions, respectively. Then, we identify the fiber over $e_0$ and $e_1$ so that $\xi_\alpha^{0,1} = \xi_1$ and $\xi_\alpha^{1,2} = \xi_2$; see Figure~\ref{fig:alpha-bigger-holonomy}. Under this identification, the holomorphic triangles have the weight listed in Theorem~\ref{prop:PoP-module-structure}(\ref{prop-sub:PoP-module-structure-functions}).

Finally, there is a sign coming from the brane structure. Since we are only interested in the weight of holomorphic disks lying in the cylindrical region, we can express the Lagrangian as a product of a line and $S^1$ (on the universal cover) and look at each factor separately. In this case, the sign comes from the nontrivial brane structure on the $S^1$ factor, where the weight is determined in a way similar to the one from the local system. Namely, we have an auxiliary real line bundle $\beta_i \to C_i$ on each $S^1$ factor $C_i$, and each intersection point $y \in C_{i_0} \cap C_{i_1}$ is equipped with an element in $\Hom(\beta_{i_0, y}, \beta_{i_1, y})$; see Section 13 in \cite{Sei1} for a detailed argument. One can then realize the weight as a real number by fixing an identification between different fibers which assigns a holonomy along the ``half-turns'' (in a way similar to Remark~\ref{rem:diag-end-local-system}).

Nontriviality of the brane structure means that the line bundle $\beta$ is nontrivial (i.e. has a holonomy $-1$). This creates a inconsistency for generators of the type $x_\alpha^j x_\alpha^e$ with odd $j$, so we add additional negative sign for one of the generators, $x_\alpha^{\text{odd}, {\text{odd} + j}} x_\alpha^e$, to make it consistent; see Figure~\ref{fig:alpha-bigger-brane}. This also gives a negative sign on the corresponding coefficient $g^{\text{odd},\,\, \text{odd}+k} (\rho)$.

\vspace{0.3em}
\begin{figure}[ht]
    \centering
    \includegraphics[width=0.4\textwidth]{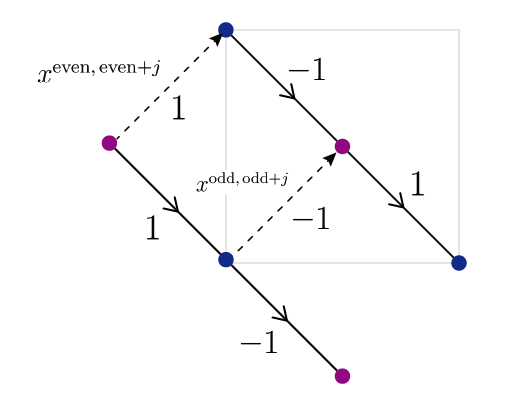}
    \caption{The brane structure on $L_\PoP$. The number represents the weight assigned to each half-turn and the generator $x_\alpha^{j_0, j_1} x_\alpha^e$ when $j = j_1 - j_0$ is odd. The dashed line represents the direction of the wrapping.}
    \label{fig:alpha-bigger-brane}
\end{figure}
\vspace{0.3em}

Applying Theorem~\ref{thm:disk-bound-arg-fibertorus} to the other two cylindrical ends implies that the function corresponding to $X_\alpha^j$ does not have the term $z_1^{m_1} z_2^{m_2}$ with negative $m_1$ or negative $m_2$. Therefore, Proposition~\ref{prop:PoP-module-structure}, together with the equation (\ref{eqn:fiber-torus-x*e-is-f(p)e}), gives the explicit expression of $X_\alpha^j$ as well as the corresponding function.

\begin{theorem}[HMS of the Lagrangian $L_\PoP$, Part 2]\label{thm:HMS-for-LPoP-diag}
    The generators of the Floer cohomology $HW^0(L_\PoP, L_\PoP)$ on the diagonal cylindrical end and their corresponding functions are as follows.

    \vspace{0.1in}

    \centering
    \begin{tabularx}{0.8\textwidth} { >{\centering\arraybackslash\hsize=.5\hsize}X | >{\centering\arraybackslash}X >{\centering\arraybackslash\hsize=.6\hsize}X }
        & generators $X_\alpha^k$ & corresponding functions on $D_\PoP$ \\
        \hline\hline
        \\[-1.2em]
        $k = 2l$ & $x_\alpha^{0,k} x_\alpha^e + x_\alpha^{1,k+1} x_\alpha^e$ & $ z_1^l z_2^l$ \\
        \\[-0.8em]
        $k = 2l+1$ & $x_\alpha^{0,k} x_\alpha^e + x_\alpha^{1,k+1} x_\alpha^e + x_\alpha^{1,k} x_\alpha^e $ & \shortstack[c]{$z_1^{l+1} z_2^l$ \\$(= -z_1^{l} z_2^{l+1} + z_1^{l} z_2^{l})$}
    \end{tabularx}
    \vspace{0.1in}
\end{theorem}

%%%%%%%%%%%%%%%%%444

%\subsection{HMS for the Diagonal End}

%In this section, we do a more detailed analysis of the holomorphic strips and prove Proposition~\ref{prop:PoP-module-structure}.

We end this section proving Proposition~\ref{prop:PoP-module-structure}(\ref{prop-sub:PoP-module-structure-generators}). Recall that the main step to prove Theorem~\ref{thm:disk-bound-arg-strip} (the argument constraint) is to divide the holomorphic strip into the interior and cylindrical part and compare the asymptotic $\theta_\alpha^\perp$-value on each side. All three different versions of the theorem assume that the cylindrical part of the strip is sufficiently close to a point on the $(p_\alpha, \theta_\alpha)$-plane so that there is a unique ``expected'' value of $\theta_\alpha^\perp$. 

For the pair of pants, we have fewer possibilities for the interior part of the disk, which gives a better version of these theorems. Indeed, there is a unique homology class $[\gamma]_\alpha \in H_1 (L) / [S_\alpha^1]$ with a given wrapping number $w_\alpha(\gamma)$ since the boundary map 
$$\delta: H_2 ((\bC^*)^2, L_\PoP) \to H_1 (L_\PoP)$$
in the long exact sequence of the pair $((\bC^*)^2, L_\PoP)$ is the zero map. Lemma~\ref{def:integral-eta-alpha} then tells us that there is a unique value of $\theta_\alpha^\perp$ (after fixing the value of $\theta_\alpha = \arg{r}$) from a disk with wrapping number $1$. In particular, we can show the non-existence of holomorphic disks following the proof of argument constraints if the cylindrical part of the disk (strips from $x^e$ to $x^f$ or from $e$ to $f$) does not pass through the corresponding $\theta_\alpha^\perp$ value. After action rescaling, such holomorphic strips have weight $T^{-2\pi w_\alpha(\gamma) p_\alpha} = T^{-2\pi p_\alpha}$ times a bounded number independent of $p$. Therefore, if we pick $\varphi_\alpha$ so that the value of $\theta_\alpha^\perp$ lies between $\varphi_\alpha$ and $\varphi_\alpha + \pi$ (i.e. on the top-left diagonal line in Figure~\ref{fig:alpha-bigger-holonomy}), then there is no strip of width $1$ to the point of type $f$ between $\varphi_\alpha + \pi$ and $\varphi_\alpha + 2\pi$ (i.e. one on the bottom-right diagonal line in Figure~\ref{fig:alpha-bigger-holonomy}), which implies the following.

\begin{lemma}\label{lem:PoP-holonomy-is-compatible}
    Let $t = T^{2\pi p_\alpha}$. Under the identification of local systems given in Figure~\ref{fig:alpha-bigger-holonomy}, there exists a number $c = 1 + O(t^{-2})$ such that $e_1 + c e_2$ is the generator of $HF^\bullet(F_\rho, L_\PoP)$.
\end{lemma}
$\hfill\square$

This estimation, together with the module structure computed in Proposition~\ref{prop:PoP-module-structure}(\ref{prop-sub:PoP-module-structure-functions}), is enough to find the exact coefficient of $e_1$ and $e_2$.

\begin{proof}[Proof of Proposition~\ref{prop:PoP-module-structure}(\ref{prop-sub:PoP-module-structure-generators})] 
    We only need to show that the coefficient $c$ in Lemma~\ref{lem:PoP-holonomy-is-compatible} has no higher-order term. To prove this, we look at $X_\alpha^1$ and its corresponding function. Let us omit $x_\alpha^e$ for simplicity, and set
    $$ X_\alpha^1 = x_\alpha^{0,1} + x_\alpha^{1,2} + a_1 x_\alpha^{1,1} + a_2 x_\alpha^{2,2} $$
    for some Novikov coefficients $a_1$ and $a_2$. We may assume that $a_2 = 0$ by subtracting a multiple of $X_\alpha^0 = x_\alpha^{1,1} + x_\alpha^{2,2}$.
    Proposition~\ref{prop:PoP-module-structure}(\ref{prop-sub:PoP-module-structure-functions}) then implies that 
    $$X_\alpha^1 \cdot (e_1 + ce_2) = (-c z_2 + a_1) e_1 + z_1 e_2$$ 
    is a multiple of $e_1 + ce_2$, so we have
    $$c \cdot (-c z_2 + a_1) =  z_1 = - z_2 + 1.$$
    $z_2 = \xi_2 t$ is of order $1$ with respect to $t$, so $O(t^{-2})$ term of $c$ does not affect the constant and $t$ term on both sides of the equation. Looking at these two terms gives 
    $$- z_2 + a_1 =  - z_2 + 1,$$
    so $a_1 = 1$ and $c = 1$.
\end{proof}

\begin{remark}
    A holomorphic disk with wrapping number $1$ exists and is counted in the above calculation. Namely, the holomorphic cylinder $\{z_1 = z_2\}$ intersects with $L_\PoP$ cleanly along the curve
    $$\{z_1 = z_2, \quad \Rea z_1 = -1/2\}$$
    and gives a holomorphic disk with one infinite end along the diagonal $Z_{\alpha, r}^+$. On the universal cover, both $\theta_1$ and $\theta_2$ values along the boundary increases from $-\frac{\pi}{2}$ to $\frac{\pi}{2}$, which means that the value of $\theta_\alpha^\perp = \frac{1}{2}(\theta_1 - \theta_2)$ is fixed at $0 $, while the value of $\theta_\alpha = \theta_1 + \theta_2$ increases by $2 \pi$. 
\end{remark}

%%%%%%%%%%%%%%%%%555

\printbibliography

\end{document}